\documentclass [twoside,reqno,12pt]{amsart}
\usepackage[left=1in,right=1in,top=1in,bottom=1in]{geometry}

\usepackage{amsfonts}
\usepackage{amssymb}
\usepackage{amsthm}
\usepackage{mathrsfs}
\usepackage{xcolor}
\usepackage{graphics}
\usepackage{comment}
\usepackage[normalem]{ulem}
\usepackage{cancel}

\newtheorem{thm}{Theorem}[section]
\newtheorem{cor}[thm]{Corollary}
\newtheorem{lem}[thm]{Lemma}
\newtheorem{prop}[thm]{Proposition}

\newtheorem{assumption}{Assumption}

\usepackage[hidelinks]{hyperref}

\theoremstyle{definition}

\numberwithin{equation}{section}

\renewcommand{\Re}{\mathrm{Re}}
\renewcommand{\Im}{\mathrm{Im}}

\newcommand{\cA}{\mathcal{A}}

\newcommand{\C}{\mathbb{C}}

\newcommand{\n}[1]{\langle #1 \rangle}
\newcommand{\N}{\mathbb{N}}

\newcommand{\R}{\mathbb{R}}
\newcommand{\cR}{\mathcal{R}}

\def\tilde{\widetilde}
\def \bfo {\begin {eqnarray*} }
\def \efo {\end {eqnarray*} }
\def \ba {\begin {eqnarray*} }
\def \ea {\end {eqnarray*} }
\def \beq {\begin {eqnarray}}
\def \eeq {\end {eqnarray}}
\def \supp {\hbox{supp }}

\def \p {\partial}

\newcommand{\cO}{\mathcal{O}}

\def\tilde{\widetilde}
\def \bfo {\begin {eqnarray*} }
\def \efo {\end {eqnarray*} }
\def \ba {\begin {eqnarray*} }
\def \ea {\end {eqnarray*} }
\def \beq {\begin {eqnarray}}
\def \eeq {\end {eqnarray}}
\def \supp {\hbox{supp }}

\def \p {\partial}

\title[Inverse problem for a nonlinear dynamical Schr\"odinger operator]{Inverse problems for a nonlinear dynamical Schr\"odinger operator with magnetic potential}

\author[Kumar]{Mandeep Kumar}
\address{M. Kumar, 
Department of Mathematics\\
Indian Institute of Technology, Ropar, Rupnagar, Punjab, 140001, India}
\email{mandeep.24maz0011@iitrpr.ac.in}

\author[Liu]{Boya Liu}
\address{B. Liu, Department of Mathematics\\
North Dakota State University, Fargo\ 
ND 58102, USA}
\email{boya.liu@ndsu.edu}

\author[Vashisth]{Manmohan Vashisth}
\address{M. Vashisth, Department of Mathematics\\
Indian Institute of Technology, Ropar, Rupnagar, Punjab, 140001, India}
\email{manmohanvashisth@iitrpr.ac.in}

\begin{document}

\begin{abstract}

We study two inverse problems for a nonlinear dynamical Schrödinger equation with time-dependent magnetic and electric potentials. Under suitable analyticity assumptions, we show that the associated Dirichlet-to-Neumann map uniquely determines the linear magnetic potential and all coefficients of the nonlinear electric potential. We establish both full-data and partial-data uniqueness results. For the partial data problem, assuming that the coefficients are known in a neighborhood of the boundary, uniqueness is obtained using measurements made on arbitrarily small open subsets of the boundary.  In addition, we establish the well-posedness of the forward problem.
\end{abstract}

\maketitle

\section{Introduction and Statement of Results}

Let $\Omega \subseteq \R^n$, $n\ge 2$, be a bounded open set with smooth boundary $\p \Omega$. For any $T>0$, let  $Q:=(0,T)\times \Omega$ denote the space-time domain, with its lateral boundary given by $\Sigma := (0,T) \times \p \Omega$. We also denote by $\nu$  the outward unit normal to $\partial\Omega$, and by $\nabla$ the gradient with respect to the spatial variable. 

Consider the nonlinear dynamical Schr\"odinger operator $\mathcal{L}_{\mathcal{A},q}$ defined by 
\begin{equation}
\label{eq:magnetic_operator}
\begin{aligned}
\mathcal{L}_{\mathcal{A},q}u &:= \mathrm{i}\partial_t u -
\left(-\mathrm{i}\nabla+\mathcal{A}(t,x)\right)^2u +
q(t,x,u)\\
&= \mathrm{i}\partial_t u+\Delta u
+2\mathrm{i} \mathcal{A}(t,x)\cdot\nabla u 
+\mathrm{i} (\nabla\cdot\mathcal{A}(t,x)) u
-|\mathcal{A}(t,x)|^2u
+q(t,x,u),
\end{aligned}
\end{equation}
with a linear time-dependent  \textit{magnetic potential}  $\mathcal{A}\in C^{\infty}\left(\overline{Q},\R^{n}\right)$, as well as  a nonlinear time-dependent \textit{electric potential} $q(t,x,u)$. Throughout this paper, we  assume that the nonlinear potential $q: \overline{Q} \times \mathbb{C} \rightarrow \mathbb{C}$
satisfies the following properties:

\begin{assumption}
\label{assump2}
The map \(\mathbb{C} \ni z \mapsto q(\cdot,\cdot, z)\) is holomorphic with values in   \(C^\infty (\overline{Q}, \C)\).
\end{assumption}

\begin{assumption}
\label{assump3}
$q (t,x, 0)= 0$ for all $(t,x) \in \overline{Q}$.
\end{assumption}
\noindent 
Then it follows from Assumptions \ref{assump2}--\ref{assump3} that  $q$ possesses the following power series expansion: 
\begin{equation}
\label{eq:expansion_q}
q(t,x, z)=\sum_{k=1}^{\infty} q_k(t,x) \frac{z^{k}}{k!}, \quad 
q_k(t,x) :=  \partial_{z}^{k} q(t,x,0),
\end{equation}
and the series converges in the $C^\infty$-topology. 
From a physical perspective, the operator $\mathcal{L}_{\mathcal{A},q}$ appears in the study of quantum mechanics of charged particles, as well as nonlinear optics, where $q(t,x,u)$ models the nonlinear polarization of the medium. 

Consider the following initial boundary value problem (IBVP) for the nonlinear operator $\mathcal{L}_{\cA, q}$:
\begin{equation}
\label{eq:ibvp_nonlinear}
\begin{cases}
\mathcal{L}_{\mathcal{A},q} u=0 & \text{ in } Q,
\\
u(0,\cdot) = 0  & \text{ in }  \Omega,
\\
u = f  & \text{ on } \Sigma.
\end{cases}
\end{equation}  
We shall establish in Theorem \ref{thm:wellposedness} that this problem has a unique solution  for sufficiently small boundary value $f$.

The main purpose of this paper is to uniquely recover the potentials $\cA(t,x)$ in   $Q$ and $q(t,x,u)$ in  $Q\times \C$ from both full and partial boundary measurements, which we describe in detail in the next two subsections.
\subsection{Full data problem}

We begin this subsection by defining the boundary measurement used in the first main result of this paper. To this end, associated with the IBVP \eqref{eq:ibvp_nonlinear}, we define the  \textit{Dirichlet-to-Neumann} (DN) map $\Lambda_{\mathcal{A},q}$ via the formula
\begin{equation}
\label{eq:def_DN_map}
\Lambda_{\mathcal{A},q} (f):= (\p_\nu +\mathrm{i}(\mathcal A(t,x)\cdot \nu))u |_{\Sigma},
\end{equation}
where $u$ is the solution to \eqref{eq:ibvp_nonlinear} with sufficiently small Dirichlet data $f$.  In view of Theorem \ref{thm:wellposedness},  $\Lambda_{\mathcal{A},q}$ is well-defined for such $f$.

The first inverse problem we study in this paper is whether $\Lambda_{\mathcal{A},q}$ determines the linear magnetic potential $\cA(t,x)$ and the nonlinear electric potential $q(t,x,u)$ uniquely. Let us observe that the coefficient  $q_{1}$ in the expansion \eqref{eq:expansion_q} corresponds to a linear  electric potential. After applying the first order linearization, the unique recovery of $\mathcal{A}$ and $q_1$ reduces to an inverse problem of determining the magnetic and electric potentials appearing in a linear dynamical Schr\"odinger equation. Therefore, the gauge invariance associated with the DN map of the linear  equation implies that the unique   recovery of $\mathcal{A}$ is possible only when its divergence $\nabla \cdot \mathcal{A}$ is known, see \cite{Kian_Soccorsi,Lai_Uhlmann_Yan,NakSunUlm_1995}.

We are now ready to state the first main result of this paper. 
\begin{thm}
\label{thm:main_result}
Let $\Omega \subseteq \R^n$, $n\ge 2$, be a bounded domain with smooth boundary $\p \Omega$. For any $0<T<\infty$, let $Q:=(0,T) \times \Omega$ and $\Sigma := (0,T) \times \p \Omega$. Suppose that the linear magnetic potentials $\mathcal{A}^{(1)}, \mathcal{A}^{(2)}: \overline{Q}  \to \mathbb{R}^{n}$, and that the nonlinear electric potentials $q^{(1)}, q^{(2)}: \overline{Q} \times \mathbb{C} \to \mathbb{C}$ satisfy Assumptions \ref{assump2}--\ref{assump3}. Let
$\Lambda_{\mathcal{A}^{(i)},q^{(i)}}$, $i=1,2$, denote the DN map \eqref{eq:def_DN_map} when $(\mathcal{A},q)=(\mathcal{A}^{(i)},q^{(i)})$. Then $\Lambda_{\mathcal{A}^{(1)},q^{(1)}}=\Lambda_{\mathcal{A}^{(2)},q^{(2)}}$ implies that $\mathcal{A}^{(1)}=\mathcal{A}^{(2)}$ in $\overline{Q}$ and $q^{(1)}=q^{(2)}$ in $\overline{Q} \times \mathbb{C}$, provided  that  $\nabla\cdot \mathcal{A}^{(1)}= \nabla\cdot \mathcal{A}^{(2)}$ in $Q$, $\mathcal{A}^{(1)}=\mathcal{A}^{(2)}$ on  $\Sigma$, and  $ q^{(1)}=q^{(2)}$ on  $\Sigma\times\mathbb{C}$. 
\end{thm}
\subsection{Partial data problem}
\label{subsec:partial_data_problem}

We now turn our attention to the second inverse problem studied in this paper. Specifically, assuming that the  potentials $\cA(t,x)$ and $q(t,x,u)$ are known near the lateral boundary $\Sigma$, our goal is to prove that measurements made on arbitrarily small open subsets of $\Sigma$ determine $\cA$ and $q$ uniquely.  

The study of partial data inverse problems in this setting was initiated in \cite{Ammari_Uhlmann}, which established a uniqueness result for the scalar potential appearing in the linear Schr\"odinger operator $-\Delta +q$. We refer readers to \cite{Bellassoued_Fraj,Fathallah,Ben_Joud,Krup_Lass_Uhl_magSchr_euc,Krupchyk_Uhlmann_stability,Liu_Salem,Yang,Zhao_Yuan} and the references therein for related results for various linear operators, as well as \cite{Lai_Lu_Zhou,Lai_Uhlmann_Yan} for nonlinear operators with the same partial data setting. 

We next   describe the partial boundary measurement used to establish our second main result of this paper. Assume that $\cO \subset \Omega$ is an open neighborhood of $\p \Omega$. Let $\Gamma \subseteq \p \Omega$ be an arbitrary nonempty open set. In what follows, we  denote $\Sigma^\sharp := (0,T)\times \Gamma$ and  define a \textit{partial DN map} by the formula 
\begin{equation}
\label{eq:revise_def_DN_map}
\Lambda_{\mathcal{A},q}^\sharp (f):= (\p_\nu +\mathrm{i}(\mathcal{A}(t,x)\cdot \nu))u |_{\Sigma^{\sharp}}.
\end{equation}
The inverse problem we are interested in is whether $\Lambda_{\mathcal{A},q}^\sharp$ uniquely determines $\cA(t,x)$ and  $q(t,x,u)$, provided that $\cA$ is known in $(0,T) \times \cO$, and $q$ is  known in $(0,T) \times \cO\times\C$. Similar to the full data case, it is necessary to assume that $\nabla \cdot \mathcal{A}$ is known in order to obtain the global uniqueness of $\mathcal{A}$. 

Our second main result of this paper is as follows.

\begin{thm}
\label{thm:main_result_partial_data}
Let $\Omega \subseteq \R^n$, $n\ge 2$, be a bounded domain with smooth boundary $\p \Omega$. For any $0<T<\infty$, let $Q:=(0,T) \times \Omega$ and $\Sigma := (0,T)\times \p \Omega$. Assume that $\cO \subset \Omega$ is an arbitrary open neighborhood of $\p \Omega$. Suppose that the linear magnetic potentials $\mathcal{A}^{(1)}, \mathcal{A}^{(2)}: \overline{Q}  \to \mathbb{R}^n$ and nonlinear electric potentials $q^{(1)}, q^{(2)}: \overline{Q} \times \mathbb{C} \to \mathbb{C}$ satisfy Assumptions \ref{assump2}--\ref{assump3}. Let
$\Lambda^\sharp_{\mathcal{A}^{(i)},q^{(i)}}$, $i=1,2$, denote the partial  DN map \eqref{eq:revise_def_DN_map} when $(\mathcal{A},q)=(\mathcal{A}^{(i)},q^{(i)})$.  Furthermore, assume that $\nabla\cdot \mathcal{A}^{(1)}= \nabla\cdot \mathcal{A}^{(2)}$  in  $Q$, and $\mathcal{A}^{(1)} = \mathcal{A}^{(2)}$ in $(0,T)\times \cO$ and $q^{(1)} = q^{(2)}$ in $(0,T)\times \cO\times \mathbb{C}$. Then  $\Lambda_{\mathcal{A}^{(1)},q^{(1)}}^\sharp=\Lambda_{\mathcal{A}^{(2)},q^{(2)}}^\sharp$  implies that $\mathcal{A}^{(1)}=\mathcal{A}^{(2)}$ in  $\overline{Q}$ and  $q^{(1)}=q^{(2)}$ in  $\overline{Q} \times \mathbb{C}$.
\end{thm}

The main contributions of this paper are as follows. In Theorem \ref{thm:main_result}, we uniquely recover the linear magnetic potential and general analytic nonlinear electric potential of a semi-linear dynamical Schr\"odinger operator from full boundary measurements. Theorem \ref{thm:main_result_partial_data} provides the corresponding partial data result. 
Furthermore, we establish the well-posedness of the nonlinear IBVP \eqref{eq:ibvp_nonlinear}  when  the nonlinear electric potential $q(t,x,u)$ is complex-valued, due to which the linearized equation becomes non-self-adjoint. 
\subsection{Previous literature}
In the last decade, inverse problems for nonlinear partial differential equations (PDEs) have gained considerable attention. The seminal work \cite{Kurylev_Lassas_Uhlmann} demonstrated  that nonlinearity can be helpful in solving inverse problems for hyperbolic equations. A crucial tool in the study of nonlinear inverse problems is the method of linearization, which was introduced  in \cite{Isakov1993OnUI} to study an  inverse problem for nonlinear parabolic equations.  This approach has been adopted to study inverse problems for various types of nonlinear PDEs. Among the extensive literature on the subject, we refer readers to  \cite{Hintz_Uhlmann_Zhai_4th_order,Hintz_Uhlmann_Zhai,Kian2021NonlinearTerms,Lassas_Liimatainen_Potenciano_Tyni,Lassas_Uhlmann_Wang,Liu_Wang,NakamuraVashisthWatanabe2021Inverse,NakamuraWatanabe2008,NakamuraWatanabeKaltenbacher2009,Uhlmann_Zhang_Quadratic} for results related to hyperbolic equations, \cite{CarsteaFeizmohammadi2021,CarsteaFeizmohammadiKianKrupchykUhlmann2021,Carstea_Nakamura_Vashisth,FeizmohammadiKianOksanen2024,FeizmohammadiLiimatainenLin2023,Kian2023StableNonlinearTerms,KianKrupchykUhlmann2023,Krupchyk_Uhlmann_gradient,Krupchyk_Uhlmann_semilinear_partial,Krupchyk_Uhlmann,Lai_Lu_Zhou,LaiZhou2023PartialData,Lassas_Liimatainen_Lin_Salo,LassasLiimatainenLinSalo2021} for elliptic equations, and \cite{CarsteaGhoshNakamura2025,CarsteaGhoshUhlmann2023,Feizmohammadi_Kian_Uhlmann,Kian_Uhlmann,Lai_Lu_Zhou,Lai_Uhlmann_Yan} for parabolic equations. 

By applying linearization, nonlinear inverse problems can often be reduced to a sequence of inverse problems for linearized equations. A well-known technique to solve inverse problems for linear PDEs is using the geometric optics solution, which was first used   in \cite{Sylvester1987AGU} to address the uniqueness of the conductivity equation. Geometric optics solutions remain one of the principal tools in the study of inverse problems for linear PDEs and have been successfully applied to a wide variety of operators and geometric settings. We refer readers to \cite{Liu_Saksala_Yan,Liu_Saksala_Yan_potential,Mishra_Purohit_Vashisth,Sahoo_Vashith} and the references therein for its applications in inverse problems related to time-dependent coefficients of hyperbolic and parabolic equations.

Let us next discuss inverse problems for the  dynamical Schr\"odinger equation. Results in this direction can be divided into two categories: time-independent and time-dependent coefficients. For linear operators without the magnetic potential, uniqueness and stability for the time-independent scalar potential were first established in \cite{BaudouinPuel2002} from a single boundary measurement. Subsequently, stability results were obtained in \cite{Bellassoued,Bellassoued_Choulli_single,Bellassoued_Choulli_DN,Bellassoued_DDS_dynamical}, among others. In the case of time-dependent coefficients, the first result concerning the uniqueness of time-dependent magnetic and electric potentials was obtained in \cite{Eskin_parabolic}. In a certain geometric setting, the author of \cite{Tetlow2022} established the recovery of a time-dependent Hermitian connection and potential appearing in a vector-valued dynamical Schr\"odinger equation. Moreover, stability estimates from full or partial boundary measurements were derived in \cite{BenAicha2017Stability,Choulli_Kian_Soccorsi, Kian_Soccorsi,Kian_Tetlow}. 

Finally, we briefly review inverse problems for nonlinear dynamical Schrödinger equations. The authors of \cite{Lassas_etal_dynamical} established the uniqueness of both time-dependent linear and nonlinear potentials from the source-to-solution map. For power-type nonlinearities, a logarithmic stability estimate from partial boundary measurements was obtained in \cite{Lai_Lu_Zhou}. We also mention the partial data uniqueness result of \cite{Lai_Uhlmann_Yan} for a different class of nonlinearities, together with a corresponding reconstruction result  \cite{KumarNakamuraVashisth2026}. A related uniqueness result  for a fourth-order nonlinear Schr\"odinger equation was recently studied in \cite{MishraPurohitSahoo2026}.
Moreover, let us mention a recent stability result of \cite{Arrepu_Zhou} concerning a time-independent nonlinearity.

\subsection{Ideas of the proof of main results}

The proof of Theorems \ref{thm:main_result} and \ref{thm:main_result_partial_data} contains several crucial components, the first of which is  the construction of geometric optics solutions to linear dynamical Schr\"odinger equations, which is carried out in Section \ref{sec:construction_GO_solution}. Specifically, the solutions are of the form
\[
v_{\lambda}(t,x)= e^{\mathrm{i}(\lambda x\cdot\alpha-\lambda^2\lvert\alpha\rvert^2t)}\bigg(m_{0}(t,x)+\frac{m_{1}(t,x)}{\lambda}+\frac{m_{2}(t,x)}{\lambda^2}+\cdots+\frac{m_{N}(t,x)}{\lambda^N}\bigg)+ R_{\lambda}(t,x),
\]  
where $\lambda\gg 1$ is a large parameter, the vector $\alpha \in \R^n \setminus \{0\}$,  the smooth amplitudes $m_j$, $j=0, \dots, N$, satisfy suitable transport equations, and $R_\lambda$ is a correction term that vanishes as $\lambda \to \infty$. These solutions are substituted into appropriate integral identities to reduce the unique recovery of coefficients to the injectivity of ray transforms or Fourier transforms. Let us highlight that it is sufficient to construct geometric optics solutions whose remainder decays in the space $H^{1}(0,T;L^{2}(\Omega))
\cap
L^{2}(0,T;H^{1}(\Omega))$, as given in Proposition \ref{prop: geom opt sol 1}, to recover the potentials $\mathcal{A}$ and $q_1$. However, the recovery of the nonlinear coefficients $q_k$, $k\ge 2$, requires a stronger decay estimate, namely, $R_\lambda \to 0$  in $H^{m+1}(Q)$ in the limit $\lambda \to \infty$, where $m>0$ is sufficiently large.  This improved estimate is established in Corollary \ref{Corollary: geom opt sol 1}.

Another major step in the proof is applying the technique of higher order linearization.  Since the seminal work \cite{Kurylev_Lassas_Uhlmann}, nonlinearity has been exploited in the study of inverse problems, and higher order linearization  plays a crucial role in it. In the proof of Theorem \ref{thm:main_result}, the linearization procedure is applied repeatedly to recover the unknown coefficients. 

The unique recovery of $\cA$ and $q$ is carried out in several steps. We begin in Subsection \ref{subsec:first_order_linearization} by performing the first order linearization. This yields an integral identity relating the unknown coefficients $\mathcal{A}$ and $q_1$ to the DN map, which is then used to recover these coefficients. Indeed, we observe from \eqref{eq:magnetic_operator} and \eqref{eq:expansion_q} that $\mathcal{A}$ and $q_1$ represent the linear coefficients of the dynamical Schr\"odinger operator. Due to the gauge invariance of the DN map, $\mathcal{A}$ can only be determined up to the gauge transformation given in \eqref{eq:gauge_relation_A0}. Consequently, in order to obtain uniqueness of $\mathcal{A}$, it is necessary to assume that $\nabla\cdot \mathcal{A}^{(1)}= \nabla\cdot \mathcal{A}^{(2)}$  in  $Q$. Afterwards,  in Subsection \ref{sec:second_order_linearization}, we employ the second-order linearization to recover  $q_2$. Finally, in Subsection \ref{subsec:higher_order_linearization}, we combine higher order linearization with an induction argument to uniquely recover all remaining nonlinear coefficients.

On the other hand, the proof of Theorem \ref{thm:main_result_partial_data} primarily relies on a unique continuation property stated in Lemma \ref{UCP}. It allows us to reduce the partial data problem to a corresponding full data one. Once this reduction is complete, the desired uniqueness results follow directly from Theorem \ref{thm:main_result}. 

Let us  conclude this section  by discussing the main ideas to establish the well-posedness of the IBVP \eqref{eq:ibvp_nonlinear}. We begin with the analysis of the corresponding linear dynamical Schr\"odinger equation with a source term and homogeneous initial and boundary conditions, whose well-posedness is established in Proposition \ref{prop:wellposedness_linear_base_step}. We then show in Lemma \ref{Lemma:wellposedness_linear_induction} that solutions inherit higher Sobolev regularity when the magnetic and electric potentials, as well as the source term, possess higher regularity. Subsequently, we utilize the expansion \eqref{eq:expansion_q} to reformulate the nonlinear  IBVP  \eqref{eq:ibvp_nonlinear} as one for a linear equation whose source term includes nonlinearities. Combining the well-posedness theory for the linear equation with a fixed-point argument, we establish the existence, uniqueness, and regularity of solutions to \eqref{eq:ibvp_nonlinear}.

This paper is organized as follows. In Section \ref{sec:notation}, we introduce some notations that will be used throughout this paper  and recall some properties about time-dependent Sobolev spaces. Afterwards, we establish the well-posedness of the IBVP \eqref{eq:ibvp_nonlinear} with small boundary data in Section \ref{sec:well-posedness}. We then move to the construction of the geometric optics solutions in Section \ref{sec:construction_GO_solution}, followed by establishing the full data result, namely,  Theorem \ref{thm:main_result}, in Section \ref{sec:proof_full_data}. Finally, we prove the partial data result stated in Theorem \ref{thm:main_result_partial_data} in Section \ref{sec:proof_partial_data}. 

\subsection*{Acknowledgments}
The authors would like to express their gratitude to  Philipp Zimmermann for helpful discussions.  
M.K. acknowledges the support of the Prime Minister Research Fellowship (PMRF ID: 2902506) from the government of India for his research. The research of B.L. was partially supported by the  Simons Foundation Travel Support for Mathematicians (MPS-TSM-00013766). The work of M.V. was supported by the ARG-MATRICS grant from the
ANRF, Government of India (File No.  ANRF/ARGM/2025/002368/MTR). This research also received partial support under the FIST program of the Department
of Science and Technology, Government of India (Ref. No. SR/FST/MS-I/2018/22(C)). 

\section{Notation and Preliminaries}
\label{sec:notation}
In this short section, we introduce the function spaces that will be used in the subsequent analysis.
For a measure space $(\omega,\mu)$, a Hilbert space   $\left(X, \langle\cdot,\cdot\rangle_{X} \right)$, and  $1\leq p<\infty$, we define the Bochner space $L^{p}(\omega;X)$ by
\[
L^{p}(\omega;X)
:=
\left\{
f:\omega\to X:
f \text{ is strongly measurable and }
\int_{\omega}\|f(x)\|_{X}^{p} d\mu(x)<\infty
\right\},
\]
and it is equipped with the norm
\begin{equation*}
\|f\|_{L^{p}(\omega;X)}
:=
\left(
\int_{\omega}\|f(x)\|_{X}^{p} d\mu(x)
\right)^{1/p}\ \ \mbox{for}\ f\in L^p(\omega;X).
\end{equation*}
In particular, when $p=2$, the space $L^{2}(\omega;X)$ is a Hilbert space with the inner product 
\begin{equation*}
\langle f,g\rangle_{L^{2}(\omega;X)}
:=
\int_{\omega} \langle f(x),g(x)\rangle_{X} d\mu(x),\ \ \mbox{for}\  f,g\in L^{2}(\omega;X).
\end{equation*}
On the other hand, when $p=\infty$, we define the space 
\begin{equation*}
L^{\infty}(\omega;X)
:=
\left\{
f:\omega\to X:
f \text{ is strongly measurable and }
\operatorname*{ess   sup}_{x\in\omega}\|f(x)\|_{X}<\infty
\right\},
\end{equation*}
where
\begin{equation*}
\operatorname*{ess  sup}_{x\in\omega}\|f(x)\|_{X}
:=
\inf
\left\{
M>0:
\mu\left(
\{x\in\omega:\|f(x)\|_{X}>M\}
\right)=0
\right\},
\end{equation*}
and the corresponding norm is given by
\begin{equation*}
\|f\|_{L^{\infty}(\omega;X)}
:=
\operatorname*{ess sup}_{x\in\omega}\|f(x)\|_{X}.
\end{equation*}

We next briefly discuss   some  properties of Sobolev spaces that are required in this paper. Let $\omega\subset \mathbb{R}^{n}$
be an open set with boundary $\partial\omega$. For any integer $m\in\mathbb{N}\cup\{0\}$ and $1\leq p\leq \infty$, we define $W^{m,p}(\omega;X)$ as the space of all functions $u\in L^{p}(\omega;X)$ whose weak derivative $\partial^{\alpha}u \in L^{p}(\omega;X)$ exists for every multi-index $\alpha$ such that $|\alpha|\leq m$. 
The space $W^{m,p}(\omega;X)$ is a Banach space with a norm
\begin{equation*}
\|u\|_{W^{m,p}(\omega;X)}
:=
\sum_{|\alpha|\leq m}
\|\partial^{\alpha}u\|_{L^{p}(\omega;X)}.
\end{equation*}
In particular, when $m=0$, we have $W^{0,p}(\omega;X)=L^{p}(\omega;X)$. 
Also, in the special case $p=2$, we shall write
$H^{m}(\omega;X):=W^{m,2}(\omega;X)$, which is a Hilbert space with an inner product
\begin{equation*}
\langle u,v\rangle_{H^{m}(\omega;X)}
:=
\sum_{|\alpha|\leq m}
\langle \partial^{\alpha}u,\partial^{\alpha}v\rangle_{L^{2}(\omega;X)}.
\end{equation*}
Moreover, we denote by $H_{0}^{m}(\omega;X)$ the closure of $C_{c}^{\infty}(\omega;X)$ in $H^{m}(\omega;X)$, namely, 
\begin{equation}\label{eq:H0m_definition}
H_{0}^{m}(\omega;X)
:=
\overline{C_{c}^{\infty}(\omega;X)}^{ H^{m}(\omega;X)},
\end{equation}
where $C_{c}^{\infty}(\omega;X)$ is the space of compactly supported smooth functions valued in $X$.

If $X= \mathbb{R}^{n}$ or $\mathbb{C}^{n}$,  we simply denote the space $W^{m,p}(\omega;X)$ as $W^{m,p}(\omega)$. Furthermore, if $\partial\omega$ is  smooth and
$m\geq 1$, an equivalent characterization of \eqref{eq:H0m_definition} is given by 
\[
H_{0}^{m}(\omega)
=
\left\{
u\in H^{m}(\omega):
\partial^{\alpha}u|_{\partial\omega}=0
\text{ for all } |\alpha|\leq m-1
\right\},
\]
where the boundary values are understood in the sense of traces, see \cite[Theorem 5.37]{AdamsFournier2003} for more details. 

Let us now recall some properties of Sobolev spaces involving time. For the space-time  $Q=(0,T)\times \Omega$ and any pair of integers $r,s\geq 0$, we define the anisotropic Sobolev space
$H^{r,s}(Q)
:=
H^{r}(0,T;L^{2}(\Omega))
\cap
L^{2}(0,T;H^{s}(\Omega))$, 
equipped with the norm
\begin{equation}\label{eq:anisotropic_sobolev_norm}
\|u\|_{H^{r,s}(Q)}=
\|u\|_{H^{r}(0,T;L^{2}(\Omega))}
+
\|u\|_{L^{2}(0,T;H^{s}(\Omega))}.
\end{equation} 
In particular, when $r=s=m$, by the standard characterization of Sobolev spaces on product domains, we have $H^{m,m}(Q)=H^{m}(Q)$ with equivalent norms, see for instance \cite[Chapter~4, Proposition~2.1]{LionsMagenes1968Vol2}.

Finally, we define anisotropic Sobolev spaces on the lateral boundary $\Sigma=(0,T)\times\p\Omega$ as follows. For any real numbers $r,s\ge 0$, we write $H^{r,s}(\Sigma):= H^{r}\left(0,T;L^{2}(\p\Omega)\right) \cap L^{2}\left(0,T;H^{s}(\p\Omega)\right)$,
equipped with the norm
\[
\|f\|_{H^{r,s}(\Sigma)}:=\|f\|_{H^{r}(0,T;L^{2}(\p\Omega))}+\|f\|_{L^{2}(0,T;H^{s}(\p\Omega))}.
\]

\section{Well-posedness of the Forward Problem}
\label{sec:well-posedness}

The goal of this section is to show that the IBVP 
\eqref{eq:ibvp_nonlinear} is well-posed for small Dirichlet boundary data $f$. Since the  argument is lengthy and technical, we shall divide the proof into several steps.
We first study a nonhomogeneous linear dynamical Schr\"odinger equation with homogeneous initial and boundary conditions, whose well-posedness is established in Proposition \ref{prop:wellposedness_linear_base_step}. Next, under stronger regularity assumptions on the coefficients and the source term, we prove in Lemma \ref{Lemma:wellposedness_linear_induction} that the solution possesses regularity in the higher order Sobolev spaces.
We then incorporate nonhomogeneous boundary conditions and establish a corresponding well-posedness result   in Lemma \ref{Lemma:nonhom_wellposedness_linear}. Finally, by combining these linear well-posedness results with a fixed-point argument, we prove that the nonlinear IBVP \eqref{eq:ibvp_nonlinear} is well-posed for sufficiently small boundary data $f$ in Theorem \ref{thm:wellposedness}.

Our first main result of this section gives us the well-posedness of the linear dynamical Schr\"odinger equation, particularly with complex-valued electric potential $q(t,x)$. In what follows,  the notation $a \lesssim b$ stands for $a\le Cb$ for some implicit constant  $C>0$ that is independent of $a$ and $b$. The dependence of $C$ will be stated whenever relevant.

\begin{prop}
\label{prop:wellposedness_linear_base_step}
Let $\Omega \subseteq \R^n$, $n \ge 2$, be a bounded open set with smooth boundary $\p \Omega$. For any $T>0$, let $Q=(0,T) \times \Omega$ and $\Sigma = (0,T) \times \p \Omega$.
Suppose that $A \in W^{2,\infty}(Q;\mathbb{R}^{n})$, $q \in W^{1,\infty}(Q;\mathbb{C})$, and $F\in H^1(0,T;L^2(\Omega))$ satisfies the initial condition $F(0,\cdot)=0$ in $\Omega$. Then the IBVP 
\begin{equation}
\label{eq:magnetic_schrodinger_ibvp}
\begin{cases}
\mathrm{i}\partial_t u
+\Delta u
+2\mathrm{i}A \cdot\nabla u
+\mathrm{i}\nabla \cdot A  u
-|A|^{2}u
+q u
=F
& \text{ in } Q,
\\
u=0 & \text{ on }\Sigma,
\\
u(0,\cdot)=0 & \text{ in }\Omega, 
\end{cases}
\end{equation}
admits a unique solution $u\in L^{2}\left(0,T;H^{2}(\Omega)\cap H^{1}_{0}(\Omega)\right) \cap H^{1}(0,T;L^{2}(\Omega))$, 
which satisfies the  estimate
\begin{equation}
\label{eq:final_regularity_estimate}
\|u\|_{H^{1,2}(Q)} 
\lesssim
\|F\|_{H^1(0,T;L^2(\Omega))}.
\end{equation}
Here the implicit constant   depends on $\Omega$, 
$\|A\|_{W^{2,\infty}(Q)}$, $\|q\|_{W^{1,\infty}(Q)}$, and $T$, but is independent of $u$. 
\end{prop}

\begin{proof}

The proof is completed via several steps. First, we introduce an equivalent weak formulation of the IBVP \eqref{eq:magnetic_schrodinger_ibvp} and apply the Faedo-Galerkin method to this weak formulation to establish the existence and uniqueness of the solution $ u \in H^{1}\left(0,T;L^{2}(\Omega)\right) \cap L^{2}\left(0,T;H_{0}^{1}(\Omega)\right)$ to \eqref{eq:magnetic_schrodinger_ibvp}.
Next, using elliptic regularity, we improve the spatial regularity of the solution and conclude that $u \in L^{2}\left(0,T;H^{2}(\Omega)\cap H_{0}^{1}(\Omega)\right) \cap H^{1}\left(0,T;L^{2}(\Omega)\right).$

For almost every $t\in(0,T)$ and  all  $u(t,\cdot), v \in H_{0}^{1}(\Omega)$, we define the following sesquilinear form
\begin{equation}
\label{eq:sesquilinear_form}
\begin{aligned}
a(t,u,v)
&= 
\left \langle\nabla u(t,\cdot), \nabla v \right\rangle_{L^{2}(\Omega)}
+
\left\langle 
\left(-2\mathrm{i} A(t,\cdot) \cdot \nabla 
-\mathrm{i}\nabla \cdot  A(t,\cdot) + |A(t,\cdot)|^{2}-q(t,\cdot)\right)u(t,\cdot), v
\right\rangle_{L^{2}(\Omega)}
\\
&
:= a_0 (t,u,v)+ a_1 (t,u,v).
\end{aligned}
\end{equation}
Then it follows immediately from the triangle inequality that
\begin{equation}
\label{eq:first_split}
|a(t,u,u)|
\geq |a_{0}(t,u,u)| - |a_{1}(t,u,u)|.
\end{equation}
We also observe that
\[
a_{0}(t,u,u) = \|\nabla u(t,\cdot)\|_{L^{2}(\Omega)}^{2}.
\]

Let us now estimate $a_{1}(t,u,u)$. For the first order term, since $u(t,\cdot) \in H_0^1(\Omega)$, by the Cauchy-Schwarz inequality,  the inequality
\begin{equation}
\label{eq:Linf_L2_product}
\|fg\|_{L^{2}(\Omega)} \leq \|f\|_{L^{\infty}(\Omega)}\|g\|_{L^{2}(\Omega)}, \quad   f\in L^{\infty}(\Omega), \: g\in L^{2}(\Omega), 
\end{equation}
as well as the fact that $2ab\leq \frac{1}{2}a^{2}+2b^{2}$ for all $a,b\in\mathbb{R}$, we have 
\begin{equation}
\label{eq:first_order_CS}
\begin{aligned}
\left|\left\langle -2\mathrm{i} A(t,\cdot)\cdot\nabla u(t,\cdot) , u(t,\cdot) \right\rangle_{L^{2}(\Omega)}\right|
&\leq 2 \|A(t,\cdot)\|_{L^{\infty}(\Omega)}
\|\nabla u(t,\cdot)\|_{L^{2}(\Omega)}\|u(t,\cdot)\|_{L^{2}(\Omega)}
\\
&\leq
\frac{1}{2}\|\nabla u(t,\cdot)\|_{L^{2}(\Omega)}^{2}
+2\|A(t,\cdot)\|_{L^{\infty}(\Omega)}^{2}\|u(t,\cdot)\|_{L^{2}(\Omega)}^{2}.
\end{aligned}
\end{equation}
On the other hand, for the zeroth order terms, it holds that
\begin{equation}
\label{eq:zeroth_order}
\begin{aligned}
&\left|
\left\langle 
\left(-\mathrm{i}\nabla \cdot  A(t,\cdot) + |A(t,\cdot)|^2 - q(t,\cdot)\right)u(t,\cdot),
u(t,\cdot)
\right\rangle_{L^{2}(\Omega)}
\right| 
\\
& \qquad 
\leq
\left\|-\mathrm{i}\nabla \cdot  A(t,\cdot) 
+ |A(t,\cdot)|^2 - q(t,\cdot)\right\|_{L^\infty(\Omega)}
\|u(t,\cdot)\|_{L^2(\Omega)}^{2}.
\end{aligned}
\end{equation}
Thus,   we deduce from \eqref{eq:first_split}--\eqref{eq:zeroth_order} that
\begin{equation}
\label{eq:garding_final}
|a(t,u,u)|
\geq
\frac{1}{2}\|\nabla u(t,\cdot)\|_{L^2(\Omega)}^2 -
\gamma \|u(t,\cdot)\|_{L^2(\Omega)}^2
\end{equation}
for some constant $\gamma>0$ depending on $\|A(t,\cdot)\|_{W^{1,\infty}(\Omega)}$, and $\|q(t,\cdot)\|_{L^{\infty}(\Omega)}$, but  independent of $u$.

We next move on  to establish the existence of solutions to the IBVP \eqref{eq:magnetic_schrodinger_ibvp}, whose weak formulation   is to find a function $u\in H^{1}\left(0,T;L^{2}(\Omega)\right)\cap L^{2}\left(0,T;H_{0}^{1}(\Omega) \right)$ such that  
\begin{equation}
\label{eq:weak_formulation}
\langle \mathrm{i}\partial_{t}u(t,\cdot),v\rangle_{L^{2}(\Omega)}
- a(t,u,v)
= \langle F(t,\cdot),v\rangle_{L^{2}(\Omega)} \quad  \text{ and } \quad u(0,\cdot)=0
\end{equation}
for all 
$v\in H_{0}^{1}(\Omega)$ and almost every $t\in (0,T)$.
We shall adopt the Faedo-Galerkin method to construct such a function, see \cite[Section 3, Theorem 10.1] {lions1972nonhomogeneous1}. This method has been adopted in \cite{Kian_Soccorsi,Lai_Uhlmann_Yan} as well.
We start by choosing an orthonormal basis $\{w_{k}\}_{k=1}^{\infty}$  of $L^{2}(\Omega)$, which is  also
orthogonal in $H_{0}^{1}(\Omega)$. For each $m\in \N$, let us  define the set 
$E_{m}:=\mathrm{span}\{w_{1},\dots,w_{m}\}$.
We seek an approximate solution to \eqref{eq:weak_formulation} of  the form
\begin{equation}
\label{eq:form_um}
u_{m}(t,x)=\sum_{k=1}^{m} g_{k}^{(m)}(t) w_{k}(x), \quad  (t,x)\in    Q.
\end{equation}
Here the coefficient functions $g_k^{(m)}(t)$, $k=1,\dots,m$, are determined in such a way that the corresponding approximate solution $u_m$ satisfies the initial value problem
\begin{equation}
\label{eq:galerkin_system-1}
\begin{cases}
\langle\mathrm{i} \partial_{t}u_{m}(t,\cdot),w_{k}  \rangle_{L^{2}(\Omega)} - a(t,u_{m},w_{k})= \langle F(t,\cdot),w_{k}\rangle_{L^{2}(\Omega)},
\\
u_{m}(0,x)=0.
\end{cases}
\end{equation}

To find $g_j^{(m)}(t)$,  we utilize the orthonormality of $\{w_k\}_{k=1}^m$ in $L^2(\Omega)$, the condition $u_{m}(0,x)=0$ for almost every $x\in \Omega$, along with the sesquilinearity of $a$.  After these steps, we see that  $g^{(m)}_j(t)$ satisfies the following initial value problem: 
\begin{equation}
\label{eq:galerkin_ode_system}
\begin{cases}
\mathrm{i} \partial_{t} g_{j}^{(m)}(t)
-\sum_{k=1}^{m} a(t,w_{k},w_{j}) g_{k}^{(m)}(t)
= \langle F(t,\cdot),w_{j}\rangle_{L^{2}(\Omega)},
\\
g_{j}^{(m)}(0)=0,
\end{cases}
\end{equation}
To simplify notations, let us introduce the vectors  $g^{(m)}(t):=(g_{1}^{(m)}(t),\dots,g_{m}^{(m)}(t))^{T}$ and $F_{m}(t,\cdot):=(f_{1}(t,\cdot),\dots,f_{m}(t,\cdot))^{T}$, where $f_j(t,\cdot) = \n{F(t,\cdot), w_j}_{L^2(\Omega)}$, as well as the matrix  $\mathbb{A}(t,\cdot):=\left[a(t,w_{k},w_{j})\right]_{1\leq j,k\leq m}$. Then \eqref{eq:galerkin_ode_system} can be rewritten as the system of  equations
\begin{equation}
\label{eq:vector_system}
\begin{cases}
\mathrm{i} \partial_{t} g^{(m)}(t)- \mathbb{A}(t,\cdot) g^{(m)}(t) = F_{m}(t,\cdot),
\\
g^{(m)}(0)=0.
\end{cases}
\end{equation}
By integrating the first equation of \eqref{eq:vector_system} over the interval $(0,t)$ and using the initial condition
$g^{(m)}(0)=0$, we obtain  
\begin{equation}
\label{integral_Galerkin_system}
g^{(m)}(t)
= 
-\mathrm{i}\int_{0}^{t} \mathbb{A}(s,\cdot) g^{(m)}(s)ds
-\mathrm{i}\int_{0}^{t} F_{m}(s,\cdot)ds, \quad  t\in(0,T) .
\end{equation}

We now show that for every $F_{m}\in L^{2}(0,T;\mathbb{C}^{m})$ and $\mathbb{A}\in L^{\infty}(0,T;\mathbb{C}^{m^2})$, the equation \eqref{integral_Galerkin_system} admits a unique solution $g^{(m)}\in H^{1}\left(0,T;\mathbb{C}^{m}\right)$.  
To achieve this, we  introduce an operator whose fixed points are precisely the solutions of \eqref{integral_Galerkin_system} and show that this operator is a contraction on a sufficiently small time interval. The existence and uniqueness of a local solution  will follow from the Banach fixed point theorem. This solution is then extended to the entire interval $(0,T)$ by a standard continuation argument. 

Consider the operator $\Phi: H^{1}(0,T;\mathbb{C}^{m})\rightarrow   H^{1}(0,T;\mathbb{C}^{m})$ defined by the formula 
\[
\Phi(\widetilde{g}) (t)=
-\mathrm{i}\int_{0}^{t} \mathbb{A}(s) \widetilde{g}(s)ds
-\mathrm{i}\int_{0}^{t} F_{m}(s)ds, \quad t\in (0,T),
\]
where $\mathbb{A}$ and $F_{m}$ are as in the system   
\eqref{eq:vector_system}. We shall prove that  $\Phi$ is a contraction for some suitable $T_{\star}\in (0,T)$, which will be chosen later. To this end, we observe that for any $\widetilde{g}_1,\widetilde{g}_2\in H^{1}(0,T;\mathbb{C}^{m})$ such that $\widetilde{g}_1(0)=\widetilde{g}_2(0)=0$, the following identity holds for any $t\in (0,T)$:
\begin{equation}
\label{eq:phi_difference}
\Phi(\widetilde{g}_1)(t)-\Phi(\widetilde{g}_2)(t)
= -\mathrm{i}\int_{0}^{t} \mathbb{A}(s)\left(\widetilde{g}_1(s)-\widetilde{g}_2(s)\right)ds 
\end{equation}
Using H\"older's inequality, together with the boundedness of $\mathbb{A}$, and taking the supremum over $t\in(0,t^{\prime})$,  we obtain  the following  inequality:
\begin{equation}
\label{eq:phi-Linf-est}
\left\|\Phi(\widetilde{g}_1)-\Phi(\widetilde{g}_2)\right\|_{L^\infty(0,t^{\prime};\mathbb C^m)}
\le
\sqrt{t^{\prime}} \|\mathbb{A}\|_{L^\infty (0,t^{\prime};\mathbb{C}^{m^2})}
\left\|\widetilde{g}_1-\widetilde{g}_2\right\|_{L^2\left(0,t^{\prime};\mathbb C^m\right)}
\end{equation}
We next estimate the $L^{\infty}$-norm of $\frac{d}{dt}\Phi(\tilde g)=: \Phi'(\tilde g)$. To this end, by differentiating both sides of \eqref{eq:phi_difference} with respect to $t$ and arguing as above, we get
\begin{equation}
\label{eq:phi_time_derivative-Linf-est}
\left\|\Phi'(\widetilde{g}_1)-\Phi'(\widetilde{g}_2) \right\|_{L^{\infty}(0,t^{\prime};\mathbb{C}^{m})}
\leq 
\|\mathbb{A}\|_{L^{\infty}(0,t^{\prime};\mathbb{C}^{m^2})}\sqrt{t^{\prime}}
\left\|\frac{d}{dt}\left(\widetilde{g}_1-\widetilde{g}_2\right)\right\|_{L^{2}(0,t^{\prime};\mathbb{C}^{m})}.
\end{equation}
Furthermore, using the inequality
\[
\|f\|_{L^{2}\left(0,t^{\prime};\mathbb{C}^{m}\right)}
\leq
\sqrt{t^{\prime}} 
\|f\|_{L^{\infty}\left(0,t^{\prime};\mathbb{C}^{m}\right)},\quad  f\in L^{\infty}\left(0,t^{\prime};\mathbb{C}^{m}\right),
\]
along with estimates \eqref{eq:phi-Linf-est} and \eqref{eq:phi_time_derivative-Linf-est}, we obtain 
\begin{align*}
\left\|\Phi(\widetilde{g}_1)-\Phi(\widetilde{g}_2)\right\|_{H^{1}(0,t^{\prime};\mathbb{C}^{m})}
&  \leq
t^{\prime} \|\mathbb{A}\|_{L^{\infty}\left(0,T;\mathbb{C}^{m^2}\right)}
\|\widetilde{g}_1-\widetilde{g}_2\|_{H^{1}(0,t^{\prime};\mathbb{C}^{m})}.
\end{align*}
Choosing $t^{\prime}=T_{\star}>0$ sufficiently small such that $T_{\star}\|\mathbb{A}\|_{L^{\infty}\left(0,T;\mathbb{C}^{m^2}\right)}<1$, we conclude that the map $\Phi: H^{1}(0,T_{\star};\mathbb{C}^{m})\to  H^{1}(0,T_{\star};\mathbb{C}^{m})$ is a contraction.
Hence, by the Banach fixed point theorem, the integral equation \eqref{integral_Galerkin_system} admits a unique solution
$g^{(m)}\in H^{1}(0,T_{\star};\mathbb{C}^{m})$ . 

We next extend the solution $g^{(m)}$ to the interval $\left(\frac{T_{\star}}{2}, \frac{3T_{\star}}{2}\right)$, namely, we establish that the vector-valued function $g^{(m)}\in H^{1}\left( \frac{T_{\star}}{2}, \frac{3T_{\star}}{2};\mathbb{C}^{m} \right)$  satisfies the equation \eqref{integral_Galerkin_system}.
To this end, we define the operator $\widetilde{\Phi} : H^{1} \left(\frac{T_{\star}}{2},\frac{3T_{\star}}{2};\mathbb{C}^{m}\right)
\rightarrow 
H^{1} \left(\frac{T_{\star}}{2},\frac{3T_{\star}}{2};\mathbb{C}^{m}\right)$ by
\[
\widetilde{\widetilde{g}} 
\mapsto 
-\mathrm{i}\int_{\frac{T_{\star}}{2}}^{t} \mathbb{A}(s) \widetilde{\widetilde{g}}(s) ds
-\mathrm{i}\int_{\frac{T_{\star}}{2}}^{t} F_{m}(s) ds 
+ 
g^{(m)}\left(\frac{T_{\star}}{2}\right),
\]
where $\mathbb{A}$ and $F_{m}$ are the same as in the initial value problem \eqref{eq:vector_system}. By similar arguments as above,  $\widetilde{\Phi}$ admits a unique fixed point in 
$H^{1} \left(\frac{T_{\star}}{2},\frac{3T_{\star}}{2};\mathbb{C}^{m}\right)$ that solves the equation \eqref{integral_Galerkin_system} on $\left(\frac{T_{\star}}{2},\frac{3T_{\star}}{2}\right)$. Furthermore, by the uniqueness theorem, this solution must agree with the previously 
constructed $g^{(m)}$ in the overlapping interval 
$\left(\frac{T_{\star}}{2},T_{\star}\right)$. As a consequence, $g^{(m)}$ extends to a solution of 
\eqref{integral_Galerkin_system} on  
$\left(0,\frac{3T_{\star}}{2}\right)$.
Proceeding iteratively, we obtain a unique solution $g^{(m)}\in H^{1}\left(0,T;\mathbb{C}^{m} \right)$ of \eqref{integral_Galerkin_system}. Moreover, in view of \eqref{eq:form_um}, we have $u_{m}\in H^{1}\left(0,T;E_{m} \right)$.

We next  show that $g^{(m)}\in H^{2}\left(0,T;\mathbb{C}^{m}\right)$, which immediately gives us $u_{m}\in H^{2}\left(0,T;E_{m} \right)$. To this end, since $A\in W^{2,\infty}(Q)$, 
$q\in W^{1,\infty}(Q)$, and $F\in H^{1}(0,T;L^{2}(\Omega))$, we get that $\mathbb{A}\in W^{1,\infty}\left(0,T;\mathbb{C}^{m\times m}\right)$ and $F_m\in H^{1}\left(0,T;\mathbb{C}^{m}\right)$ for any $m\in \mathbb{N}$. 
As $g^{(m)}\in H^{1}(0,T;\mathbb{C}^{m})$, it follows immediately that 
$\mathbb{A}g^{(m)}\in H^{1}(0,T;\mathbb{C}^{m})$. In view of \eqref{eq:vector_system}, we have  $\partial_t g^{(m)}=-\mathrm{i}\left(\mathbb{A}g^{(m)}+F_m\right)\in H^{1}(0,T;\mathbb{C}^{m})$. Hence, we get  $g^{(m)}\in H^{2}\left(0,T;\mathbb{C}^{m}\right)$. Due to \eqref{eq:form_um}, we conclude that $u_m\in H^{2}\left(0,T;E_m\right)$.

We now show that $\p_t u_m(0,\cdot)=0$ for later purposes. Using that $g^{(m)}(0,\cdot)=0$ and 
the given initial condition $F(0,\cdot)=0$, which yields that $F_m(0,\cdot)=0$ for any $m\in \mathbb{N}$, we obtain
\[
\partial_t g^{(m)}(0,\cdot)=-\mathrm{i}\left(\mathbb{A}g^{(m)}+F_m\right)(0,\cdot)=0.
\]
This implies that $\partial_t u_m(0,\cdot)=0$ in $\Omega$ for any $m\in\mathbb{N}$.

Next, we aim  to derive the uniform estimate 
\begin{equation}
\label{eq:un_estimate}
\|u_{m}\|_{L^{\infty}(0,T;L^{2}(\Omega))}
+
\|\nabla u_{m}\|_{L^{\infty}(0,T;L^{2}(\Omega))}
+\|\partial_{t}u_{m}\|_{L^{\infty}(0,T;L^{2}(\Omega))}
\lesssim
\|F\|_{H^1(0,T;L^{2}(\Omega))},
\end{equation}
where the implicit constant depends on $T$, $\|A\|_{W^{2,\infty}(Q)}$, and $\|q\|_{W^{1,\infty}(Q)}$, but is independent of $m$. We first bound $\|u_{m}\|_{L^{\infty}(0,T;L^{2}(\Omega))}$ from above, for which we only require the source term $F\in L^{2}(Q)$. Then we derive an upper bound for the other two terms, where both the assumption $F\in H^{1}(0,T;L^{2}(\Omega))$ and the compatibility condition $F(0,\cdot)=0$ will be utilized.

Let us first establish the estimate
\begin{equation}
\label{eq:un_estimate_basic}
\|u_{m}\|_{L^{\infty}(0,T;L^{2}(\Omega))}
\lesssim
\|F\|_{L^{2}(Q)}.
\end{equation}
By multiplying the equation \eqref{eq:galerkin_system-1} by $\overline{g_{k}^{(m)}(t)}$ and summing over $k=1,2,\dots m$, we get
\begin{equation}
\label{eq:galerkin_energy_identity}
\langle \mathrm{i} \partial_{t}u_{m}(t,\cdot),u_{m}(t,\cdot) \rangle_{L^{2}(\Omega)} 
- a(t,u_{m},u_{m})
= \langle F(t,\cdot),u_{m}(t,\cdot)\rangle_{L^{2}(\Omega)}.
\end{equation}
Since $u_m(t,\cdot) \in H_0^1(\Omega)$, and $A$ is real-valued, integration by parts yields that
\[
\Im a(t,u_{m},u_{m}) = \Im \n{(|A(t,\cdot)|^2-q(t,\cdot))u_m(t,\cdot), u_m(t,\cdot)}_{L^2(\Omega)} 
= -\int_{\Omega} \Im  q(t,x) |u_{m}(t,x)|^{2}dx.
\] 
Thus, by taking the imaginary parts of \eqref{eq:galerkin_energy_identity} and applying Young's inequality, we have
\[
\frac{d}{dt}\|u_{m}(t,\cdot)\|_{L^{2}(\Omega)}^{2}  - \left(2\gamma+1\right) \|u_{m}(t,\cdot)\|^{2}_{L^{2}(\Omega)}
\leq  
\|F(t,\cdot)\|^{2}_{L^{2}(\Omega)},
\]
where  $\gamma>0$ is a constant depending on $\|A(t,\cdot)\|_{W^{1,\infty}(\Omega)}$ and $\|q(t,\cdot)\|_{L^{\infty}(\Omega)}$, but is independent of $u_m(t,\cdot)$.
From here, an application of Gronwall's lemma gives us the estimate \eqref{eq:un_estimate_basic}.

We next   derive the estimate \begin{equation}\label{eq:step2}
\|\nabla u_{m}\|_{L^{\infty}(0,T;L^{2}(\Omega))}
+\|\partial_{t}u_{m}\|_{L^{\infty}(0,T;L^{2}(\Omega))}
\;\lesssim\; \|F\|_{H^{1}(0,T;L^{2}(\Omega))} ,
\end{equation}
where the implicit constant depends on $T$, $\|A\|_{W^{2,\infty}(Q)}$, and $\|q\|_{W^{1,\infty}(Q)}$. To that end, multiplying the  equation  \eqref{eq:galerkin_system-1} by $\partial_{t}\overline{g_{k}^{(m)}(t)}$ and summing over $k=1,2,\dots m$, we have
\[
\langle\mathrm{i} \partial_{t}u_{m}(t,\cdot),\partial_{t}u_{m}(t,\cdot)  \rangle_{L^{2}(\Omega)} - a(t,u_{m},\partial_{t}u_{m})= \langle F(t,\cdot),\partial_{t}u_{m}(t,\cdot)\rangle_{L^{2}(\Omega)}.
\]
Then we take the real part of the equation above to obtain
\begin{equation}
\label{eq:time_estimate1}
- \Re a(t,u_{m},\partial_{t}u_{m})= \Re\langle F(t,\cdot),\partial_{t}u_{m}(t,\cdot)\rangle_{L^{2}(\Omega)}.
\end{equation}
Let us define a sesquilinear form  $b(t,u,v)=a(t,u,v)+\mathrm{i}\langle \Im {q}(t,\cdot)u,v\rangle_{L^2(\Omega)}$. In view of the equation \eqref{eq:sesquilinear_form}, we have 
\begin{equation}\label{defn-b}
\begin{aligned}
b(t,u,v)
= &
\langle \nabla u(t,\cdot),\nabla v\rangle_{L^2(\Omega)}
\\&
+
\left\langle
\left(
-2\mathrm{i} A(t,\cdot)\cdot\nabla
-\mathrm{i}\nabla\cdot A(t,\cdot)
+|A(t,\cdot)|^2
-\Re {q}(t,\cdot)
\right)u(t,\cdot),
v
\right\rangle_{L^2(\Omega)} .
\end{aligned}
\end{equation}
We also define the operator $b'$  by the following formula: 
\begin{equation}
\label{eq:b-prime-def}
b'(t,u,v)
:=
\left\langle
\left(
-2\mathrm{i} \partial_{t}A(t,\cdot)\cdot\nabla
-\mathrm{i}\nabla\cdot \partial_{t}A(t,\cdot)
+2A(t,\cdot)\cdot\partial_{t}A(t,\cdot)
-\partial_{t}\Re {q}(t,\cdot)
\right)u(t,\cdot),
v
\right\rangle_{L^2(\Omega)} .
\end{equation}
Since $A$ is real-valued  and $u(t,\cdot),v\in H^{1}_{0}(\Omega)$, it holds that the operator
$-2\mathrm{i}A\cdot\nabla-\mathrm{i}\nabla\cdot A$ is symmetric, and that 
$|A|^{2}-\Re q$ is real-valued. Hence,  $b(t,\cdot,\cdot)$ is Hermitian. By  similar computations as in \eqref{eq:first_order_CS}--\eqref{eq:garding_final}, we get
\begin{equation}
\label{eq:garding_b}
b(t,u,u) \ge \frac{1}{2}\|\nabla u\|_{L^{2}(\Omega)}^{2}
-\gamma_{1}\|u\|_{L^{2}(\Omega)}^{2}
\end{equation}
for some constant $\gamma_{1}>0$ depending on $\|A\|_{W^{1,\infty}(Q)}$ and
$\|\Re q\|_{L^{\infty}(Q)}$. Thus,  we get from \eqref{eq:time_estimate1} and \eqref{defn-b} that 
\[
\Re b(t,u_m,\partial_tu_m)
=
-\Re \langle F,\partial_tu_m\rangle_{L^2(\Omega)}
+
\Re \left(
\mathrm{i}\langle (\Im {q} )u_m,\partial_tu_m\rangle_{L^2(\Omega)}
\right).
\]
Hence, by differentiating $b(t,u_m,u_m)$ with respect to $t$, we have
\[
\frac{d}{dt}b(t,u_m,u_m)
=
b'(t,u_m,u_m)
-
2\Re \langle F,\partial_tu_m\rangle_{L^2(\Omega)}
+
2\Re \left(
\mathrm{i}\langle (\Im {q}) u_m,\partial_tu_m\rangle_{L^2(\Omega)}
\right).
\]
Integrating the previous equation over the interval $(0,t)$ and using $u_m(0, \cdot)=0$, we obtain
\begin{equation}
\label{eq:b-energy-integrated}
\begin{aligned}
b(t,u_m(t, \cdot),u_m(t, \cdot))
= &
\int_0^t b'(s,u_m(s, \cdot),u_m(s, \cdot)) ds
-
2\Re \int_0^t
\langle F(s, \cdot),\partial_tu_m(s, \cdot)\rangle_{L^2(\Omega)} ds
\\
& +
2\Re \int_0^t
\mathrm{i}\langle \Im {q(s, \cdot)} u_m(s, \cdot),\partial_tu_m(s, \cdot)\rangle_{L^2(\Omega)} ds
\\
:= &
I_1+I_2+I_3.
\end{aligned}
\end{equation}

Let us proceed to estimate each term   above. For $I_{1}$, it follows from \eqref{eq:b-prime-def}, the Cauchy-Schwarz inequality, along with the assumptions $A\in W^{2,\infty}(Q)$ and $q\in W^{1,\infty}(Q)$, that $|b'(s,u_{m},u_{m})|\lesssim\|u_{m}(s,\cdot)\|_{H^{1}(\Omega)}^{2}$ for almost every $s\in(0,T)$. Therefore, we obtain
\begin{equation}
\label{eq:I1-bound}
|I_1|
\lesssim
\int_{0}^{t}\|u_{m}(s,\cdot)\|^{2}_{H^{1}(\Omega)} ds
\lesssim
\int_{0}^{t}\|\nabla u_{m}(s,\cdot)\|^{2}_{L^{2}(\Omega)} ds
+\|F\|^{2}_{L^{2}(Q)},
\end{equation}
where we have used the estimate \eqref{eq:un_estimate_basic} to deduce
\[
\int_{0}^{t}\|u_{m}(s,\cdot)\|_{L^{2}(\Omega)}^{2}ds
\le T\|u_{m}\|^{2}_{L^{\infty}(0,T;L^{2}(\Omega))}\lesssim\|F\|^{2}_{L^{2}(Q)}
\]
in the last step. 

For $I_2$, in view of the assumption $F(0,\cdot)=0$ in $\Omega$, we integrate by parts to get
\[
\int_0^t
\langle F(s,\cdot),\partial_tu_m(s,\cdot)\rangle_{L^2(\Omega)} ds
=
\langle F(t,\cdot),u_m(t,\cdot)\rangle_{L^2(\Omega)}
-
\int_0^t
\langle \partial_tF(s,\cdot),u_m(s,\cdot)\rangle_{L^2(\Omega)}  ds.
\]
Using the Cauchy-Schwarz inequality, we obtain 
\begin{equation}
\label{eq:F-ut-by-parts-bound}
\begin{aligned}
&\left|
\int_0^t
\langle F(s,\cdot),\partial_tu_m(s,\cdot)\rangle_{L^2(\Omega)} ds
\right|
\\
& \qquad \leq 
\|F(t,\cdot)\|_{L^2(\Omega)}\|u_m(t,\cdot)\|_{L^2(\Omega)}
+
\int_0^t
\|\partial_tF(s,\cdot)\|_{L^2(\Omega)}\|u_m(s,\cdot)\|_{L^2(\Omega)} ds.     
\end{aligned}
\end{equation}
We next estimate each term on the right-hand side above. To this end, since  $F(0,\cdot)=0$ in $\Omega$, we recall the formula $\frac{d}{dt} \|F(t,\cdot)\|_{L^2(\Omega)}^2 = 2 \Re \n{F(t,\cdot), \p_t F(t,\cdot)}_{L^2(\Omega)}$ and utilize Young's inequality to deduce that 
\begin{align*}
\|F(t,\cdot)\|_{L^2(\Omega)}^2
&=
\|F(t,\cdot)\|_{L^2(\Omega)}^2
-
\|F(0,\cdot)\|_{L^2(\Omega)}^2 
\\
&=
2\Re\int_0^t
\left\langle
F(s,\cdot),\partial_s F(s,\cdot)
\right\rangle_{L^2(\Omega)} ds 
\\
&
\le
\int_0^t
\|F(s,\cdot)\|_{L^2(\Omega)}^2
+
\|\partial_s F(s,\cdot)\|_{L^2(\Omega)}^2
ds 
\\
&\leq
\|F\|_{H^1(0,T;L^2(\Omega))}^2. 
\end{align*}
Turning our attention to the second term, using the Cauchy-Schwarz inequality, in conjunction with the estimate \eqref{eq:un_estimate_basic}, we get
\begin{align*}
\int_0^t
\|\partial_tF(s,\cdot)\|_{L^2(\Omega)}
\|u_m(s,\cdot)\|_{L^2(\Omega)} ds
&\leq
\|\partial_tF\|_{L^2(0,T;L^2(\Omega))}
\|u_m\|_{L^2(0,T;L^2(\Omega))} 
\\
&\lesssim
\|F\|_{H^1(0,T;L^2(\Omega))}
\|u_m\|_{L^\infty(0,T;L^2(\Omega))} 
\\
&\lesssim
\|F\|_{H^1(0,T;L^2(\Omega))}^2.
\end{align*}
From here, we substitute the previous two inequalities into \eqref{eq:F-ut-by-parts-bound} and apply the estimate  \eqref{eq:un_estimate_basic} to conclude that 
\begin{equation}
\label{eq:F-term-bound}
\left|
I_2
\right|
\lesssim\|F\|_{H^1(0,T;L^2(\Omega))}^2.
\end{equation}

To estimate $I_3$, it follows from  Young's inequality and the estimate \eqref{eq:un_estimate_basic} that 
\begin{equation}
\label{eq:Imq-term-bound}
\begin{aligned}
\left|
I_3
\right|
&\lesssim
\int_0^t
\left(
\|u_m(s,\cdot)\|_{L^2(\Omega)}^2
+
\|\partial_tu_m(s,\cdot)\|_{L^2(\Omega)}^2
\right) ds\\
&\lesssim \|F\|_{L^{2}(Q)}^2+ \int_{0}^{t}\|\partial_tu_m(s,\cdot)\|_{L^2(\Omega)}^2 ds.
\end{aligned}
\end{equation}
Therefore, we conclude from  \eqref{eq:un_estimate_basic}, \eqref{eq:garding_b}--\eqref{eq:I1-bound}, \eqref{eq:F-term-bound}, and \eqref{eq:Imq-term-bound} that
\begin{equation}
\label{eq:gradient-un-estimate}
\|\nabla u_{m}(t,\cdot)\|_{L^{2}(\Omega)}^{2}
\lesssim
\|F\|_{H^1(0,T;L^2(\Omega))}^{2}
+
\int_{0}^{t}
\left(
\|\nabla u_{m}(s,\cdot)\|_{L^{2}(\Omega)}^{2}
+
\|\partial_t u_m(s,\cdot)\|_{L^2(\Omega)}^2
\right)ds.
\end{equation}

It remains to derive the corresponding bound for $\|\partial_{t}u_{m}(t,\cdot)\|_{L^{2}(\Omega)}$. To this end, we differentiate the  equation \eqref{eq:galerkin_system-1} with respect to
$t$, followed by multiplying the resulting equation by $\overline{\partial_{t}g_{k}^{(m)}(t)}$ and summing over $k=1,2,\dots m$ to obtain
\begin{equation}
\label{eq:dtv-equation}
\mathrm{i}\langle \partial^2_tu_m(t,\cdot),\partial_{t}u_{m}(t,\cdot)\rangle_{L^{2}(\Omega)}
-
a(t;\partial_{t}u_{m},\partial_{t}u_{m})
=
\langle \partial_tF(t,\cdot),\partial_{t}u_{m}(t,\cdot)\rangle_{L^2(\Omega)}
+
a'(t;u_m,\partial_{t}u_{m}),
\end{equation}
where $a$ is given by \eqref{eq:sesquilinear_form}, and $a'$ is a sesquilinear form defined as follows:
\[
a'(t,u,v)
:=
\left\langle
\left(
-2\mathrm{i}\partial_{t}A(t,\cdot)\cdot\nabla
-\mathrm{i}\nabla\cdot \partial_{t}A(t,\cdot)
+\partial_{t}|A(t,\cdot)|^{2}
-\partial_{t}q(t,\cdot)
\right)u,v
\right\rangle_{L^{2}(\Omega)}.
\]
Furthermore, as $A$ is real-valued, taking the imaginary part  of \eqref{eq:dtv-equation} gives us 
\[
\frac12\frac{d}{dt}\|\partial_{t}u_{m}(t,\cdot)\|_{L^2(\Omega)}^2
+
\int_\Omega \mathrm{Im} {q}(t,x)|\partial_{t}u_{m}(t,x)|^2 dx
=
\operatorname{Im}\langle \partial_tF,\partial_{t}u_{m}(t,\cdot)\rangle
+
\operatorname{Im}a'(t;u_m,\partial_{t}u_{m}).
\] 

Let us observe that $a'$ and the sesquilinear form $b'$ introduced in \eqref{eq:b-prime-def} are related by the equation 
\[
a'(t,u,v)=b'(t,u,v)-\mathrm{i}\left\langle \left(\partial_{t}\Im q(t,\cdot)\right)u,v\right\rangle_{L^{2}(\Omega)}.
\]
Thus,   we obtain $|a'(t;u_m,\partial_{t}u_{m})|
\lesssim \|u_m(t,\cdot)\|_{H^1(\Omega)}
\|\partial_{t}u_{m}(t,\cdot)\|_{L^2(\Omega)}$. Moreover, by the Cauchy-Schwarz inequality and Young's inequality, we get
\[
\frac{d}{dt}\|\partial_{t}u_{m}(t,\cdot)\|_{L^2(\Omega)}^2
\lesssim
\|\partial_{t}u_{m}(t,\cdot)\|_{L^2(\Omega)}^2
+
\|u_m(t,\cdot)\|_{H^1(\Omega)}^2
+
\|\partial_tF(t,\cdot)\|_{L^2(\Omega)}^2.
\]
Integrating in time and using $\partial_{t}u_{m}(0,\cdot)=0$ in $\Omega$, we get
\begin{equation}
\begin{aligned}
\label{eq:w-estimate}
\|\partial_{t}u_{m}(t,\cdot)\|_{L^2(\Omega)}^2
&\lesssim
\int_0^t
\left(
\|\partial_{t}u_{m}(s,\cdot)\|_{L^2(\Omega)}^2
+
\|u_m(s,\cdot)\|_{H^1(\Omega)}^2
+
\|\partial_tF(s,\cdot)\|_{L^2(\Omega)}^2
\right) ds\\
& \lesssim \|F\|_{H^{1}(0,T;L^2(\Omega))}^2+
\int_0^t
\left(
\|\partial_{t}u_{m}(s,\cdot)\|_{L^2(\Omega)}^2
+
\|\nabla u_m(s,\cdot)\|_{L^2(\Omega)}^2
\right) ds.
\end{aligned}
\end{equation}

We now add the estimates \eqref{eq:gradient-un-estimate} and \eqref{eq:w-estimate} to obtain
\begin{align*}
& \|\nabla u_{m}(t,\cdot)\|_{L^{2}(\Omega)}^2+ \|\partial_{t}u_{m}(t,\cdot)\|_{L^2(\Omega)}^2
\\
& \qquad \lesssim \|F\|_{H^{1}(0,T;L^2(\Omega))}^2+
\int_0^t
\left(
\|\partial_{t}u_{m}(s,\cdot)\|_{L^2(\Omega)}^2
+
\|\nabla u_m(s,\cdot)\|_{L^2(\Omega)}^2
\right) ds
\end{align*}
for almost every $t\in(0,T)$. By Gronwall's lemma, we have
\begin{align*}
\|\nabla u_{m}(t,\cdot)\|_{L^{2}(\Omega)}^2+ \|\partial_{t}u_{m}(t,\cdot)\|_{L^2(\Omega)}^2\lesssim \|F\|_{H^{1}(0,T;L^2(\Omega))}^2.
\end{align*}
Therefore, taking the supremum of the inequality above over $t\in(0,T)$ gives us the claimed estimate \eqref{eq:step2}. Finally, by adding the estimates \eqref{eq:un_estimate_basic} and \eqref{eq:step2}, we obtain the uniform estimate \eqref{eq:un_estimate}.

We now return to establishing the existence and uniqueness of solutions to the weak formulation \eqref{eq:weak_formulation}. In view of  the uniform bound  \eqref{eq:un_estimate}, the sequence
$\{u_{m}\}_{m=1}^\infty$ is bounded in $H^{1}(0,T;L^{2}(\Omega))\cap L^{2}(0,T;H_{0}^{1}(\Omega))$. Hence, by the Banach-Alaoglu theorem, there exists a subsequence
$\{u_{m_l}\}_{l=1}^{\infty}$ and a function
$u\in H^{1}(0,T;L^{2}(\Omega))\cap L^{2}(0,T;H_{0}^{1}(\Omega))$
such that $u_{m_l} \to  u$ and $\partial_{t}u_{m_l}  \to \partial_{t}u$ weakly in $L^2(Q)$, and $ 
\nabla u_{m_l} \to  \nabla u$
weakly in $L^2(0,T;(L^2(\Omega))^n)$.

For a fixed integer $m\in\mathbb{N}$, we  let $m_l\geq m$, and suppose that $w\in E_{m}$ is a test function. 
By the equation \eqref{eq:galerkin_system-1}, we obtain
\[
\langle \mathrm{i}\partial_{t} u_{m_l}(t,\cdot),w\rangle_{L^{2}(\Omega)}
- a(t,u_{m_l},w)
= \langle F(t,\cdot),w\rangle_{L^{2}(\Omega)} \quad  \text{ and } \quad u_{m_l}(0,x)=0.
\]
We next multiply the identity above by the function $\overline{\phi(t)}$, where $\phi\in C^\infty([0,T])$ satisfies $\phi(T)=0$, and integrate over the time interval $(0,T)$ to deduce that
\begin{equation}
\label{eq:weak_limit_start}
\int_{0}^{T}
\langle \mathrm{i}\partial_t u_{m_l}(t,\cdot),\phi w\rangle_{L^2(\Omega)} dt
-
\int_{0}^{T}
a(t,u_{m_l},\phi w) dt
=
\int_{0}^{T}
\langle F(t,\cdot),\phi w\rangle_{L^2(\Omega)} dt .
\end{equation}

Let us now analyze  the limit of each term as  $m_l\to\infty$. To this end, since $\partial_{t}u_{m_l} \to  \partial_{t}u$ weakly in $L^2(Q)$ and $\phi w\in L^2(Q)$, we have
\[
\lim_{m_l\to\infty}
\int_{0}^{T}
\langle \mathrm{i}\partial_t u_{m_l}(t,\cdot),\phi w\rangle_{L^2(\Omega)} dt
=
\int_{0}^{T}
\langle \mathrm{i}\partial_t u(t,\cdot),\phi w\rangle_{L^2(\Omega)} dt .
\]
For the term involving the sesquilinear form $a(t,u_{m_l},\phi w)$, since  $\nabla(\phi w) \in L^2(Q) $, $u_{m_l} \to  u$ weakly in $L^2(Q)$, and $\nabla u_{m_l} \to  \nabla u$ weakly in $L^2(0,T;(L^2(\Omega))^n)$, 
we get
\[
\lim_{m_l\to\infty}
\int_0^T a(t,u_{m_l},\phi w) dt
=
\int_0^T a(t,u,\phi w) dt .
\]
Therefore, for all functions $\phi\in C^\infty([0,T])$ satisfying $\phi(T)=0$, it holds that
\begin{equation}
\label{eq:weak_form_intermediate}
\int_{0}^{T}
\langle \mathrm{i}\partial_t u(t,\cdot),\phi w\rangle_{L^2(\Omega)} dt
-
\int_{0}^{T}
a(t,u,\phi w) dt
=
\int_{0}^{T}
\langle F(t,\cdot),\phi w\rangle_{L^2(\Omega)} dt. 
\end{equation}
In particular, this equation holds for all $\phi \in C_{c}^{\infty}(0,T)$. Therefore, by the Fundamental Lemma of the calculus of variations, see \cite[Chapter 3, Lemma 3.31]{AdamsFournier2003}, the identity
\begin{equation}
\label{eq:weak_formulation_Em}
\langle \mathrm{i}\partial_t u(t,\cdot), w\rangle_{L^2(\Omega)}
-
a(t,u, w)
=
\langle F(t,\cdot), w\rangle_{L^2(\Omega)}
\end{equation}
is valid for every $w\in E_{m}$. Moreover, since $\{w_k\}_{k=1}^\infty$ is orthogonal in $H_0^1(\Omega)$, the union 
$\bigcup_{m\in\mathbb{N}}E_{m}$
is dense in $H_{0}^{1}(\Omega)$. 
This allows us to extend  \eqref{eq:weak_formulation_Em} to all  
$w\in H_{0}^{1}(\Omega)$. 

To complete the verification that $u$ satisfies \eqref{eq:weak_formulation},  it remains to show that $u(0,x)=0$ for almost every $x\in\Omega$. To that end, since $\phi(T)=0$, we integrate by parts in the first term of \eqref{eq:weak_form_intermediate} to obtain
\begin{equation}
\label{eq:weak_form_limit}
-\langle\mathrm{i}u(0,\cdot),\phi(0)w\rangle_{L^2(\Omega)}
-
\int_{0}^{T}
\langle\mathrm{i}u(t,\cdot),\partial_{t}(\phi  w)\rangle_{L^2(\Omega)} dt
-
\int_{0}^{T}
a(t,u,\phi w) dt
=
\int_{0}^{T}
\langle F(t,\cdot),\phi w\rangle_{L^2(\Omega)} dt .
\end{equation}

On the other hand, since $u_{m_{l}}(0,\cdot)=0$ in
$\Omega$ for every $l$, we integrate by parts in \eqref{eq:weak_limit_start} and pass  to the limit $m_l\to\infty$ to get
\begin{equation}
\label{eq:weak_form_limit_ibp}
-
\int_{0}^{T}
\langle \mathrm{i}u(t,\cdot),\partial_{t} (\phi w)\rangle_{L^2(\Omega)} dt
-
\int_{0}^{T}
a(t,u,\phi w) dt
=
\int_{0}^{T}
\langle F(t,\cdot),\phi w\rangle_{L^2(\Omega)} dt .
\end{equation}
Thus, choosing $\phi$ such that $\phi(0)\neq0$, we deduce  from  \eqref{eq:weak_form_limit} and \eqref{eq:weak_form_limit_ibp} that $\langle \mathrm{i}u(0,\cdot),w\rangle_{L^{2}(\Omega)}=0$ for all $w\in H_{0}^{1}(\Omega)$. 
Hence,  $u(0,x)=0$ for almost every $x\in \Omega$. Therefore,  for any  $F\in H^{1}(0,T;L^{2}(\Omega))$ such that $F(0,\cdot)=0$ in $\Omega$, there exists a function $u\in H^{1}\left(0,T;L^{2}(\Omega) \right)\cap L^{2}\left(0,T;H^{1}_{0}(\Omega) \right)$ that satisfies the weak formulation \eqref{eq:weak_formulation}. In particular, $u$ is a solution to the IBVP \eqref{eq:magnetic_schrodinger_ibvp}.

We next establish the uniqueness of solutions to the IBVP \eqref{eq:magnetic_schrodinger_ibvp}. To this end, let us assume that $u_{1}$ and $u_{2}$ both satisfy \eqref{eq:weak_formulation} with the same source term $F$. Then for every
$v\in H^{1}_{0}(\Omega)$ and almost every $t\in(0,T)$,  the function $u := u_{1}-u_{2}$ satisfies the initial value problem
\[
\begin{cases}
\n{\mathrm{i}\p_{t}u(t,\cdot),v}_{L^{2}(\Omega)}-a(t,u,v)=0,
\\
u(0,x)=0.  
\end{cases}
\]
Choosing $v=u(t,\cdot)$ gives us
\[
\langle \mathrm{i}\partial_{t}u(t,\cdot),u(t,\cdot)\rangle_{L^{2}(\Omega)}
- a(t,u,u)
= 0 .
\]
Let us recall that the sesquilinear form $a(t,u,u)$ is given by \eqref{eq:sesquilinear_form}, and that $A$ is real-valued. Thus, by taking the imaginary part of the previous equation, we have
\[
\frac{1}{2}\frac{d}{dt}\|u(t,\cdot)\|^{2}_{L^{2}(\Omega)}
=
-\int_{\Omega}\Im q(t,x) |u(t,x)|^{2}dx
\le
\|\Im q\|_{L^{\infty}(Q)} \|u(t,\cdot)\|^{2}_{L^{2}(\Omega)}.
\]
By Gronwall's lemma, and   the initial condition $u(0,x)=0$,  we conclude that $\|u(t,\cdot)\|_{L^{2}(\Omega)}=0$, 
which implies that $u\equiv 0$ almost everywhere in $Q$. Therefore, the solution to the IBVP \eqref{eq:magnetic_schrodinger_ibvp} is unique.

In what follows, we prove that $u$ possesses regularity in the higher order Sobolev spaces, namely, we show that $u\in L^{2}\left(0,T;H^{2}(\Omega)\cap H^{1}_{0}(\Omega)\right)$, and it still satisfies the estimate \eqref{eq:final_regularity_estimate}.
To this end, as $u \in   H^{1}(0,T;L^{2}(\Omega))\cap L^{2}(0,T;H_{0}^{1}(\Omega))$ satisfies the problem \eqref{eq:magnetic_schrodinger_ibvp},   for almost every fixed $t\in(0,T)$, the function $u(t,\cdot)$ satisfies
the elliptic boundary value problem
\begin{equation}
\label{eq:elliptic_problem_for_u}
\begin{cases}
\Delta u(t,\cdot)=\mathfrak{F}(t,\cdot)
& \text{in }\Omega,\\
u(t,\cdot)=0
& \text{on }\partial\Omega,
\end{cases}
\end{equation}
where $\mathfrak{F}:=F-\mathrm{i}\partial_tu-2\mathrm{i}A\cdot\nabla u-\mathrm{i}(\nabla\cdot A)u +|A|^{2}u-qu$.
Let us observe from the regularity of $u$ established previously that $\mathrm{i}\partial_{t}u \in L^2(Q)$ and $\nabla u\in L^{2}(Q)$. Together with $A\in W^{1,\infty}(Q)$ and $q\in L^{\infty}(Q)$, we see that $\mathfrak{F}\in L^2(Q)$. Thus, it follows from elliptic regularity that $u\in L^{2}\left(0,T;H^{2}(\Omega)\cap H^{1}_{0}(\Omega)\right)$,  see \cite[Section 6.3, Theorem 4]{evans_pde}. Furthermore, using the triangle inequality, we get
\begin{equation}
\label{eq:H2_pre_estimate}
\begin{aligned}
\|\mathfrak{F}\|_{L^{2}(Q)}
\leq &
\|2A\|_{L^{\infty}(Q)}
\|\nabla u\|_{L^2(Q)}
+
\|F\|_{L^2(Q)}
\\
&
+
\left\|\mathrm{i}\nabla\cdot A-|A|^{2}+q\right\|_{L^{\infty}(Q)}
\|u\|_{L^2(Q)}
+
\|\partial_{t}u\|_{L^2(Q)}.
\end{aligned}
\end{equation}

Let us estimate the terms appearing on the right-hand side using the
lower semicontinuity of the norm with respect to weak convergence.
Since $u_{m_l} \to  u $, $\partial_{t}u_{m_l} \to  \partial_{t}u $, and $\nabla u_{m_l} \to  \nabla u$ weakly as $m_l \to \infty$, using the inequality $\|u\|_{L^{2}(Q)}\le\sqrt{T}\|u\|_{L^{\infty}(0,T;L^{2}(\Omega))}$ and the uniform estimate \eqref{eq:un_estimate}, we deduce that
\begin{equation}
\label{eq:u-limit-estimate}
\|u\|_{L^{2}(Q)}
+
\|\nabla u\|_{L^{2}(Q)}
+
\|\partial_{t}u\|_{L^{2}(Q)}
\lesssim
\|F\|_{H^{1}(0,T;L^{2}(\Omega))}.
\end{equation}
Therefore, we conclude from \eqref{eq:H2_pre_estimate} and \eqref{eq:u-limit-estimate} that
\begin{equation}
\label{eq:H2_regularity_estimatE_mew}
\|\mathfrak{F}\|_{L^{2}(Q)}
\lesssim
\|F\|_{H^{1}(0,T;L^{2}(\Omega))},
\end{equation}
where the implicit constant depends on $\|A\|_{W^{2,\infty}(Q)}$,
$\|q\|_{W^{1,\infty}(Q)}$ and $T$, but is independent  of $u$ and $F$. 

Since $\p \Omega$ is  smooth, elliptic regularity yields that 
\begin{equation}
\label{eq:frakF_L2_estimate}
\|u(t,\cdot)\|_{H^{2}(\Omega)}
\lesssim
\|\mathfrak{F}(t,\cdot)\|_{L^{2}(\Omega)}
\end{equation}
for almost every $t\in(0,T)$. Here we have absorbed the lower-order term $\|u(t,\cdot)\|_{L^{2}(\Omega)}$ appearing on the right-hand side of the estimate in \cite[Section 6.3, Theorem 4]{evans_pde}. Indeed, testing \eqref{eq:elliptic_problem_for_u} with $u(t,\cdot)\in H^{1}_{0}(\Omega)$ and using the Poincar\'e inequality gives us
\[
\|\nabla u(t,\cdot)\|_{L^{2}(\Omega)}^{2}
=
-\left\langle \mathfrak{F}(t,\cdot),u(t,\cdot)\right\rangle_{L^{2}(\Omega)}
\lesssim
\|\mathfrak{F}(t,\cdot)\|_{L^{2}(\Omega)}\|\nabla u(t,\cdot)\|_{L^{2}(\Omega)}.
\]
Hence, we get
\[
\|u(t,\cdot)\|_{L^{2}(\Omega)}\lesssim\|\nabla u(t,\cdot)\|_{L^{2}(\Omega)}\lesssim\|\mathfrak{F}(t,\cdot)\|_{L^{2}(\Omega)}.
\]

Squaring the estimate \eqref{eq:frakF_L2_estimate} and integrating over $(0,T)$ with respect to $t$ and invoking the estimate \eqref{eq:H2_regularity_estimatE_mew}, we get
\begin{equation}
\label{eq:u_L2H2_estimate}
\|u\|_{L^{2}(0,T;H^{2}(\Omega))}
\lesssim\|\mathfrak{F}\|_{L^{2}(Q)}\lesssim \|F\|_{H^{1}(0,T;L^{2}(\Omega))}.
\end{equation}
Finally, we obtain the estimate \eqref{eq:final_regularity_estimate} by adding \eqref{eq:u-limit-estimate} and \eqref{eq:u_L2H2_estimate}. This completes the proof of Proposition
\ref{prop:wellposedness_linear_base_step}.
\end{proof}

We next state and prove several auxiliary results that will be crucial for the subsequent analysis. 

\begin{lem}
\label{Lemma:Winfty_Hm_product_estimate}
Let $\Omega, Q$, and $\Sigma$ be the same as in Proposition \ref{prop:wellposedness_linear_base_step}. Suppose that
$m\in \mathbb{N}\cup\{0 \}$. 
If $f\in W^{m+1,\infty}(Q)$ and $g\in H^{m+1,m}(Q)$, then $fg\in H^{m+1,m}(Q)$ and
\begin{equation}
\label{eq:Winfty_Hm_product_estimate}
\|fg\|_{H^{m+1,m}(Q)}
\lesssim
\|f\|_{W^{m+1,\infty}(Q)}
\|g\|_{H^{m+1,m}(Q)},
\end{equation}
where the implicit constant depends on $m$ and $n$, but is independent of $f$ and $g$. 
\end{lem}

\begin{proof}
By  \eqref{eq:anisotropic_sobolev_norm}, we have
\begin{equation}
\label{eq:product_norm_split}
\|fg\|_{H^{m+1,m}(Q)}=\|fg\|_{H^{m+1}\left(0,T;L^2(\Omega)\right)}+\|fg\|_{L^{2}\left(0,T;H^{m}(\Omega)\right)}.
\end{equation}
Applying the Leibniz rule in the time variable, together with the triangle inequality and the inequality \eqref{eq:Linf_L2_product}, we obtain
\begin{align*}
\|fg\|_{H^{m+1}\left(0,T;L^2(\Omega)\right)}&= \sum_{k\leq m+1}\left\|\p_{t}^{k}(fg)\right\|_{L^2(0,T;L^2(\Omega))}\\
&\leq \sum_{k\leq m+1}\left(\sum_{j\leq k}\binom{k}{j}\left\|\p_{t}^{j}f \p_{t}^{k-j}g\right\|_{L^2(0,T;L^2(\Omega))}\right)\\
&\leq \sum_{k\leq m+1}\left(\sum_{j\leq k}\binom{k}{j}\left\|\p_{t}^{j}f\right\|_{L^{\infty}(Q)} \left\|\p_{t}^{k-j}g\right\|_{L^2(Q)}\right).
\end{align*}
Using the monotonicity of the Sobolev norm, along with the identity $\sum_{\substack{j\leq k}}\binom{k}{j}
= 2^{k}$,
we have
\begin{equation}
\label{eq:time_part_final}
\begin{aligned}
\|fg\|_{H^{m+1}\left(0,T;L^2(\Omega)\right)}&\leq\sum_{{k}\leq m+1}\sum_{j\leq k}\binom{k}{j}\left(\|f\|_{W^{m+1,\infty}(Q)} \|g\|_{H^{m+1}\left(0,T;L^{2}(\Omega)\right)}\right)
\\
& \leq  \left(\sum_{k\leq m+1}2^{k}\right) \left\|f\right\|_{W^{m+1,\infty}(Q)} \|g\|_{H^{m+1,m}(Q)}
\\
& =  \left(2^{m+2}-1\right)\left\|f\right\|_{W^{m+1,\infty}(Q)} \|g\|_{H^{m+1,m}(Q)}.
\end{aligned}
\end{equation}

On the other hand, for the second term of \eqref{eq:product_norm_split}, we apply the Leibniz rule in the spatial variable to get
\begin{equation}
\label{eq:space_part_expansion}
\begin{aligned}
\|fg\|^2_{L^{2}\left(0,T;H^{m}(\Omega)\right)}&= \int_{0}^{T}\left|\sum_{|\alpha|\leq m}\left\|\partial^{\alpha}_{x}\left(fg\right)(t,\cdot)\right\|_{L^{2}(\Omega)} \right|^2 dt\\
&\leq \int_{0}^{T} \left|\sum_{|\alpha|\leq m}\sum_{\beta\leq \alpha}\binom{\alpha}{\beta}\left\|\partial^{\beta}_{x}f(t,\cdot) \p^{\alpha-\beta}_{x}g(t,\cdot)\right\|_{L^{2}(\Omega)} \right|^2 dt\\
&\leq \int_{0}^{T} \left|\sum_{|\alpha|\leq m}\sum_{\beta\leq \alpha}\binom{\alpha}{\beta}\left\|\partial^{\beta}_{x}f(t,\cdot)\right\|_{L^{\infty}(\Omega)}\left\|\p^{\alpha-\beta}_{x}g(t,\cdot)\right\|_{L^{2}(\Omega)} \right|^2 dt.
\end{aligned}
\end{equation}
By the monotonicity of the Sobolev norm and the identity $\sum_{\substack{\beta\leq \alpha}}\binom{\alpha}{\beta}
= 2^{|\alpha|}$,
we deduce that
\begin{equation}
\label{eq:space_part_squared}
\begin{aligned}
\|fg\|^2_{L^{2}\left(0,T;H^{m}(\Omega)\right)}&\leq \int_{0}^{T} \left|\left(\sum_{|\alpha|\leq m}2^{|\alpha|}\right)\left\|f(t,\cdot)\right\|_{W^{m,\infty}(\Omega)} \|g(t,\cdot)\|_{H^{m}(\Omega)}\right|^2 dt\\
& \leq \left[ \left(\sum_{|\alpha|\leq m} 1\right)2^{m}\left\|f\right\|_{W^{m,\infty}(Q)} \|g\|_{L^{2}\left(0,T;H^{m}(\Omega)\right)}\right]^2.
\end{aligned}
\end{equation}
Since the multi-indices $\alpha$ in \eqref{eq:space_part_expansion} are taken in $\mathbb{N}^{n}$, it holds that
\begin{equation}
\label{eq:sum_of_1}
\sum_{|\alpha|\leq m}1
=
\sum_{k=0}^{m}\binom{n+k-1}{k}
=
\binom{n+m}{m}.
\end{equation}
Taking square roots on both sides of \eqref{eq:space_part_squared} and using \eqref{eq:sum_of_1}, we obtain
\begin{equation}
\label{eq:space_part_final}
\|fg\|_{L^{2}\left(0,T;H^{m}(\Omega)\right)}\lesssim
\left\|f\right\|_{W^{m,\infty}(Q)} \|g\|_{L^{2}\left(0,T;H^{m}(\Omega)\right)},
\end{equation}
where the implicit constant depends only on $m$ and $n$. 

Finally, since $\|f\|_{W^{m,\infty}(Q)}\le\|f\|_{W^{m+1,\infty}(Q)}$ and
$\|g\|_{L^{2}(0,T;H^{m}(\Omega))}\le\|g\|_{H^{m+1,m}(Q)}$, we substitute
\eqref{eq:time_part_final} and \eqref{eq:space_part_final} into
\eqref{eq:product_norm_split} to  get the estimate
\eqref{eq:Winfty_Hm_product_estimate}. This completes the proof of Lemma
\ref{Lemma:Winfty_Hm_product_estimate}.
\end{proof}

Our next result shows that the Sobolev space $H^{m}(Q)$ is closed under pointwise multiplication with a suitable regularity assumption. In general, it is not necessarily true that $f,g \in H^m(Q)$ implies $fg \in H^m(Q)$. However, if $m$ is sufficiently large,  $H^{m}(Q)$ becomes a Banach algebra, as stated in the following lemma. This result was originally established in  \cite[Theorem 4.39]{AdamsFournier2003}. 

\begin{lem}
\label{Lemma:banach_algebra_Hm}
Let $\Omega$, $Q$, and $\Sigma$ be the same as in Proposition \ref{prop:wellposedness_linear_base_step}. Suppose
$m\in \mathbb{N}\cup\{0 \}$ is such that $ m>n+1$. Then we have $fg\in H^m(Q)$ for all $f,g\in H^{m}(Q)$, and 
\begin{equation}
\label{eq:banach_algebra_estimate}
\|fg\|_{H^{m}(Q)}
\lesssim
\|f\|_{H^{m}(Q)}\|g\|_{H^{m}(Q)},
\end{equation}
where the implicit constant depends on $m,n,\Omega$ and $T$, but is independent of $f$ and $g$.
\end{lem}

\begin{proof}
Applying Leibniz's formula, we obtain
\begin{align*}
\|fg\|_{H^{m}(Q)}
&=
\sum_{|\alpha|\leq m}
\|\partial^{\alpha}(fg)\|_{L^{2}(Q)}
\le 
\sum_{|\alpha|\leq m}
\sum_{\beta\leq \alpha}
\binom{\alpha}{\beta}
\left\|
\partial^{\beta}f \partial^{\alpha-\beta}g
\right\|_{L^{2}(Q)}.
\end{align*}
We split the inner sum into two cases.

First, suppose that $|\beta|\leq \frac{m}{2}$. Then it follows immediately that $ m-|\beta|\geq \frac{m}{2}$. Since $m>n+1$, by the Sobolev embedding theorem, 
it holds that $H^{m-|\beta|}(Q)\hookrightarrow L^{\infty}(Q)$. Thus, we deduce that  
\[
\|\partial^{\beta}f\|_{L^{\infty}(Q)}
\lesssim
\|\partial^{\beta}f\|_{H^{m-|\beta|}(Q)}
\lesssim 
\|f\|_{H^{m}(Q)}.
\]
Consequently, using the inequality \eqref{eq:Linf_L2_product}, we obtain
\[
\left\|
\partial^{\beta}f \partial^{\alpha-\beta}g
\right\|_{L^{2}(Q)}
\leq
\|\partial^{\beta}f\|_{L^{\infty}(Q)}
\|\partial^{\alpha-\beta}g\|_{L^{2}(Q)}
\lesssim
\|f\|_{H^{m}(Q)}
\|g\|_{H^{m}(Q)}.
\]

Next, suppose that $|\beta|>\frac{m}{2}$. Since $\beta\leq \alpha$ and $|\alpha|\leq m$, we have $|\alpha-\beta|<\frac{m}{2}$.
Applying  the Sobolev embedding theorem again,
we get  
\[
\|\partial^{\alpha-\beta}g\|_{L^{\infty}(Q)}
\lesssim
\|\partial^{\alpha-\beta}g\|_{H^{m-|\alpha-\beta|}(Q)}
\lesssim
\|g\|_{H^{m}(Q)}.
\]
Thus, an application of the inequality \eqref{eq:Linf_L2_product} gives us 
\[
\left\|
\partial^{\beta}f \partial^{\alpha-\beta}g
\right\|_{L^{2}(Q)}
\leq
\|\partial^{\beta}f\|_{L^{2}(Q)}
\|\partial^{\alpha-\beta}g\|_{L^{\infty}(Q)}
\lesssim
\|f\|_{H^{m}(Q)}
\|g\|_{H^{m}(Q)}.
\]

Combining the previous two  cases, we have
\[
\left\|
\partial^{\beta}f \partial^{\alpha-\beta}g
\right\|_{L^{2}(Q)}
\lesssim
\|f\|_{H^{m}(Q)}
\|g\|_{H^{m}(Q)}.
\]
Hence, we obtain the estimate \eqref{eq:banach_algebra_estimate}by summing over the $\binom{n+m+1}{m}$ space-time multi-indices $\alpha$ with $|\alpha|\le m$ and using $\sum_{\beta\le\alpha}\binom{\alpha}{\beta}=2^{|\alpha|}\le2^{m}$. This completes the proof of Lemma \ref{Lemma:banach_algebra_Hm}.
\end{proof}

We now return to studying the well-posedness of the linear dynamical Schr\"odinger operator. In the following lemma, we assume that the coefficients $A$ and $q$, along with the source term $F$ in the IBVP \eqref{eq:magnetic_schrodinger_ibvp}, are of higher regularity, and  show that the corresponding solution $u$ also possesses higher regularity. 

\begin{lem}
\label{Lemma:wellposedness_linear_induction}
Let $\Omega, Q$, and $\Sigma$ be the same as in Proposition \ref{prop:wellposedness_linear_base_step}. Let  $A \in W^{ {m+2},\infty}(Q;\R^n)$, $q \in  W^{{m+1},\infty}(Q;\mathbb{C})$, and $F\in H^{m+1,m}(Q)$ for some  $m\in\mathbb{N}\cup\{0\}$.  Assume that $\partial_{t}^{k}F(0,\cdot)=0$ almost everywhere in $\Omega$ and for all $k=0,1,\dots ,m$.
Then the IBVP \eqref{eq:magnetic_schrodinger_ibvp} admits a unique solution 
$u\in L^{2}\left(0,T;H^{m+2}(\Omega)\cap H_{0}^{1}(\Omega)\right) \cap H^{m+1}\left(0,T;L^{2}(\Omega)\right)$, which satisfies the compatibility condition $\partial_{t}^{k}u(0,x)=0$ for almost every $x\in \Omega$ and all $k=0,1,\dots, m$. Furthermore, $u$ satisfies the following estimate:
\[
\|u\|_{H^{m+1,m+2}(Q)}
\lesssim
\|F\|_{H^{m+1,m}(Q)},
\]
where the implicit constant depends on 
$\|A\|_{W^{m+2,\infty}(Q)}$, $\|q\|_{W^{m+1,\infty}(Q)}$, and $T$, but is independent of $u$.
\end{lem}

\begin{proof}
We prove by induction on $m$. When $m=0$, the claims follow immediately from Proposition \ref{prop:wellposedness_linear_base_step}.

We now assume that the claims hold for some integer $m\geq 0$. That is, if $A\in W^{m+2, \infty}(Q)$, $q\in W^{m+1,\infty}(Q)$, and $F\in H^{m+1,m}(Q)$, the IBVP \eqref{eq:magnetic_schrodinger_ibvp} admits a unique solution   
\[
u\in L^{2}\left(0,T;H^{m+2}(\Omega)\cap H_{0}^{1}(\Omega) \right) \cap  H^{m+1}(0,T;L^{2}(\Omega)),
\]
which satisfies the estimate 
\begin{equation}
\label{eq:induction_estimates}
\|u\|_{H^{m+1,m+2}(Q)}
\lesssim
\|F\|_{H^{m+1,m}(Q)},
\end{equation}
where the implicit constant depends  on 
$\|A\|_{W^{m+2,\infty}(Q)}$, $\|q\|_{W^{m+1,\infty}(Q)}$, and $T$, but is independent of $u$. Moreover, if $\partial_{t}^{k}F(0,x)=0$ for almost every $x\in \Omega$ and all $k=0,1,\dots ,m$, we have the compatibility condition $\partial_{t}^{k}u(0,x)=0$ for almost every $x\in \Omega$ and all $k=0,1,\dots ,m$.

Our goal is to show that, when $A\in W^{m+3,\infty}(Q)$, $q\in W^{m+2,\infty}(Q)$,  $F\in H^{m+2,m+1}(Q)$, and $\partial_{t}^{k}F(0,x)=0$ for all $k=0,1,\dots,m+1$, the IBVP \eqref{eq:magnetic_schrodinger_ibvp} has a unique solution $u\in L^{2}\left(0,T;H^{m+3}(\Omega)\cap H_{0}^{1}(\Omega)\right) \cap  H^{m+2}(0,T;L^{2}(\Omega))$ satisfying the compatibility condition  $\partial_{t}^{k}u(0,x)=0$ for almost every $x\in\Omega$ and all $k=0,1,\dots,m+1$. Furthermore, $u$  satisfies the estimate
\begin{equation}
\label{eq:induction_step_estimate}
\|u\|_{H^{m+2,m+3}(Q)} \lesssim
\|F\|_{H^{m+2,m+1}(Q)},
\end{equation}
where the implicit constant depends on 
$\|A\|_{W^{m+3,\infty}(Q)}$, $\|q\|_{W^{m+2,\infty}(Q)}$, and $T$, but is independent of $u$.

Let us first verify the compatibility condition. From the induction hypothesis, we have $\partial_{t}^{k}u(0,x)=0$ for almost every $x\in \Omega$ and all $k=0,1\dots m$. 
On the other hand, we use the Leibniz rule to deduce that 
\begin{equation}
\label{eq:equation_m_time_detivative}
\left(\mathrm{i}\partial_t 
+\Delta 
+2\mathrm{i}A(t,x)\cdot\nabla 
+\mathrm{i}\nabla \cdot A(t,x) 
-|A(t,x)|^{2}
+q(t,x)\right)
\partial_t^{m}u  = F_0(t,x) \quad \text{ in } Q,
\end{equation}
where
\begin{align}
\label{def:F_0}
F_{0}(t,x)
=&
\partial_{t}^{m} F(t,x)-2\mathrm{i} \left(\partial_{t}^{m}A(t,x)\right)\cdot\nabla u
-
\mathrm{i} (\nabla \cdot \partial_{t}^{m}A(t,x)) u
+\left(\partial_{t}^{m} \left(|A(t,x)|^{2}\right)\right)u
\nonumber\\
&-\left(\partial_{t}^{m}q(t,x)\right) u
- \sum_{k=1}^{m-1} \binom{m}{k} \left(
2\mathrm{i} \partial_{t}^{k}
A(t,x)\cdot \partial_t^{m-k} (\nabla u)
+\mathrm{i}\nabla \cdot \partial_{t}^{k} A(t,x) \partial_{t}^{m-k}u \right.\nonumber
\\
&- \left.	
\partial_{t}^{k}|A(t,x)|^{2} \partial_{t}^{m-k}u
+\partial_{t}^{k}q(t,x) \partial_{t}^{m-k}u
\right) .
\end{align}
Therefore, using  \eqref{eq:equation_m_time_detivative}, \eqref{def:F_0}, along with the assumption $\partial_{t}^{m}F(0,x)=0$ for almost every $x\in \Omega$, together with $\partial^{k}_{t}u(0,\cdot)=0$ in $\Omega$, we have $\partial_{t}^{m+1}u(0,x)=0$ for almost every $x\in\Omega$.

We next establish the higher regularity of $u$ in the time variable $t$. More precisely, we claim that $u \in H^{m+2}(0,T;L^{2}(\Omega))$ and satisfies the estimate
\begin{equation}
\label{eq:CmH1_est}
\|u\|_{H^{m+2}\left(0,T;L^{2}(\Omega)\right)} 
\lesssim
\|F\|_{H^{m+2,m+1}(Q)},
\end{equation}
where the implicit constant   is independent of $u$.
To achieve this, by the Leibniz rule and the compatibility conditions, it holds that $\partial_t^{m+1}u$ satisfies the IBVP
\[
\begin{cases}
\left(\mathrm{i}\partial_t 
+\Delta 
+2\mathrm{i}A\cdot\nabla 
+\mathrm{i}\nabla \cdot A 
-|A|^{2}
+q\right)
\partial_t^{m+1}u  = F_1 & \text{ in } Q,\\
\partial^{m+1}_{t} u= 0 &\text{ on } \Sigma,\\
\partial^{m+1}_{t}u(0,\cdot) = 0 & \text{ in } \Omega.
\end{cases}
\]
Here 
\begin{equation}
\label{def:F_1}
F_{1}
= 
\partial_{t}^{m+1} F-2\mathrm{i} \left(\partial_{t}^{m+1}A\right)\cdot\nabla u
-
\mathrm{i} (\nabla \cdot \partial_{t}^{m+1}A) u
+\left(\partial_{t}^{m+1} \left(|A|^{2}\right)\right)u -\left(\partial_{t}^{m+1}q\right) u
- \mathcal{B}_{A,q}u,
\end{equation}
where 
\[
\mathcal{B}_{A,q}u := \sum_{k=1}^{m} \binom{m+1}{k} \left(
2\mathrm{i} \partial_{t}^{k}
A\cdot\nabla \partial_t^{m+1-k}u
+\mathrm{i}\nabla \cdot \partial_{t}^{k} A \partial_{t}^{m+1-k}u-
\partial_{t}^{k}|A|^{2} \partial_{t}^{m+1-k}u
+\partial_{t}^{k}q \partial_{t}^{m+1-k}u
\right) .
\]

In order to establish \eqref{eq:CmH1_est}, we shall   apply the Galerkin method again. For any $\ell\in \mathbb{N}$, we define $u_{\ell}$ by
\[
u_{\ell}(t,x)=\sum_{k=1}^{\ell} g_{k}^{(\ell)}(t) w_{k}(x), \quad  (t,x)\in Q,
\]
and denote $u_{\ell}^{(j)}:=\partial_{t}^{j}u_{\ell}$ for $j\in\mathbb{N}\cup\{0\}$.

Let us recall that $\mathbb{A}(t,\cdot):=\left[a(t,w_{k},w_{j})\right]_{1\leq j,k\leq \ell}$, where $a$ is the sesquilinear form defined by \eqref{eq:sesquilinear_form}. Since $A\in W^{m+3,\infty}(Q)$, $q\in W^{m+2,\infty}(Q)$, and $F\in H^{m+2}\left(0,T;L^{2}(\Omega)\right)$, we have $\mathbb{A}\in W^{m+2,\infty}(0,T;\mathbb{C}^{\ell^2})$ and $F_{\ell}\in H^{m+2}\left(0,T;\mathbb{C}^{\ell}\right)$. Hence, differentiating the identity $\partial_{t}g^{(\ell)}=-\mathrm{i}\left(\mathbb{A}g^{(\ell)}+F_{\ell}\right)$ repeatedly with respect to $t$ and arguing as in the proof of Proposition \ref{prop:wellposedness_linear_base_step}, we obtain $g^{(\ell)}\in H^{m+3}\left(0,T;\mathbb{C}^{\ell}\right)$, which yields that $u_{\ell}\in H^{m+3}\left(0,T;E_{\ell}\right)$. In particular, when $t=0$, it follows immediately from \eqref{integral_Galerkin_system} that $g^{(\ell)}(0)=0$. Also, due to the hypothesis $\partial_{t}^{k}F(0,\cdot)=0$, we have   $\partial_{t}^{k}F_{\ell}(0,\cdot)=0$ for all $k=0,1,\dots,m+1$. Therefore, we get 
\begin{equation}
\label{eq:galerkin_vanishing}
\partial_{t}^{k}g^{(\ell)}(0)=0,
\quad\text{ and }\quad
u_{\ell}^{(k)}(0,\cdot)=0  \text{ in }\Omega,
\quad k=0,1,\dots,m+2 .
\end{equation}

From the induction hypothesis, we have the following uniform estimate for every $\ell\in\mathbb{N}$:
\begin{equation}\label{eq:galerkin_hyp}
\sum_{j=0}^{m}
\left(
\left\|u_{\ell}^{(j)}\right\|_{L^{\infty}(0,T;L^{2}(\Omega))}
+\left\|\nabla u_{\ell}^{(j)}\right\|_{L^{\infty}(0,T;L^{2}(\Omega))}
+\left\|u_{\ell}^{(j+1)}\right\|_{L^{\infty}(0,T;L^{2}(\Omega))}
\right)
\lesssim
\|F\|_{H^{m+1}\left(0,T;L^{2}(\Omega)\right)},
\end{equation}
where the implicit constant being independent of $\ell$ and $u$. Thus, it suffices to prove   that
\begin{equation}
\label{eq:galerkin_step}
\left\|u_{\ell}^{(m+1)}\right\|_{L^{\infty}(0,T;L^{2}(\Omega))}
+\left\|\nabla u_{\ell}^{(m+1)}\right\|_{L^{\infty}(0,T;L^{2}(\Omega))}
+\left\|u_{\ell}^{(m+2)}\right\|_{L^{\infty}(0,T;L^{2}(\Omega))}
\lesssim
\|F\|_{H^{m+2}\left(0,T;L^{2}(\Omega)\right)} .
\end{equation}
Differentiating the system \eqref{eq:galerkin_system-1} $m+1$ times with respect to $t$ and using \eqref{eq:galerkin_vanishing},  we obtain 
\[
\begin{cases}
\left\langle\mathrm{i} \partial_{t}u_{\ell}^{(m+1)}(t,\cdot),w_{k}  \right\rangle_{L^{2}(\Omega)} - a\left(t,u_{\ell}^{(m+1)},w_{k}\right)= \left\langle F_{1,\ell}(t,\cdot),w_{k}\right\rangle_{L^{2}(\Omega)}
\\
u_{\ell}^{(m+1)}(0,x)=0,
\end{cases}
\]
for all $k=1,\dots,\ell$, where $F_{1,\ell}$ is given by \eqref{def:F_1} with $u$ replaced by $u_{\ell}$. Furthermore, \eqref{eq:galerkin_vanishing} and the compatibility condition  $\partial_{t}^{m+1}F(0,\cdot)=0$ yield $F_{1,\ell}(0,\cdot)=0$. Consequently, by repeating the proof of Proposition \ref{prop:wellposedness_linear_base_step}, we have
\[
\left\|u_{\ell}^{(m+1)}\right\|_{L^{\infty}(0,T;L^{2}(\Omega))}
+\left\|\nabla u_{\ell}^{(m+1)}\right\|_{L^{\infty}(0,T;L^{2}(\Omega))}
+\left\|u_{\ell}^{(m+2)}\right\|_{L^{\infty}(0,T;L^{2}(\Omega))}
\lesssim
\|F_{1,\ell}\|_{H^{1}\left(0,T;L^{2}(\Omega)\right)} .
\]

To estimate $F_{1,\ell}$, we deduce from \eqref{def:F_1} with $u=u_{\ell}$ and the induction hypothesis \eqref{eq:galerkin_hyp} that
\[
\left\|F_{1,\ell}\right\|_{L^{2}(Q)}
\lesssim
\left\|\partial_{t}^{m+1}F\right\|_{L^{2}(Q)}
+
\sum_{j=0}^{m}\left\|u_{\ell}^{(j)}\right\|_{L^{\infty}(0,T;H^{1}(\Omega))}
\lesssim
\|F\|_{H^{m+1}\left(0,T;L^{2}(\Omega)\right)} .
\]
On the other hand, by differentiating \eqref{def:F_1} with respect to $t$, we obtain
\begin{align*}
\label{eq:dt_R_j}
\partial_{t} F_{1,\ell}
= & 
\partial_{t}^{m+2} F
+ \partial_{t}\left[
-2\mathrm{i}\big(\partial_{t}^{m+1}A\big)\cdot\nabla u_{\ell}
- \mathrm{i}\big(\nabla\cdot\partial_{t}^{m+1}A\big)u_{\ell}
+ \big(\partial_{t}^{m+1}|A|^{2}\big)u_{\ell}
- \big(\partial_{t}^{m+1}q\big)u_{\ell}
\right]
\\
&- \partial_{t}\big(\mathcal{B}_{A,q}u_{\ell}\big). 
\end{align*}
Due to \eqref{eq:galerkin_hyp}, every term on the right-hand side above except for the term involving $\nabla u_{\ell}^{(m+1)}$ is bounded in $L^{2}(Q)$ by $\|F\|_{H^{m+2}\left(0,T;L^{2}(\Omega)\right)}$. This exception term arises from $\partial_{t}(\mathcal{B}_{A,q}u_{\ell})$ from the summand with $k=1$ and has the form $-2\mathrm{i}(m+1)\partial_{t}A\cdot\nabla u_{\ell}^{(m+1)}$.

Following the derivation of \eqref{eq:un_estimate}, after the integration by parts, which can be applied since $F_{1,\ell}(0,\cdot)=0$ in $\Omega$,  the quantity $\partial_{t}F_{1,\ell}$ arises  in the estimate of $I_{2}$ in \eqref{eq:b-energy-integrated}. Therefore, using the boundedness of $\p_t A$, the Cauchy-Schwarz inequality, and Young's inequality,  the additional contribution to $I_2$ is
\[
\left|
\int_{0}^{t}
\left\langle \left(\partial_{t}A(s,\cdot)\right)\cdot
\nabla u_{\ell}^{(m+1)}(s,\cdot),u_{\ell}^{(m+1)}(s,\cdot)
\right\rangle_{L^{2}(\Omega)}
ds
\right|
\lesssim
\int_{0}^{t}
\left\|u_{\ell}^{(m+1)}(s,\cdot)\right\|_{H^{1}(\Omega)}^{2} ds.
\]
Since the right-hand side is of the same form as the term $I_{1}$ estimated in \eqref{eq:I1-bound}, it is handled by the same  application of Gronwall's lemma. This establishes \eqref{eq:galerkin_step}, and consequently proves the estimate \eqref{eq:galerkin_hyp} with $m$ replaced by $m+1$.

Moreover, in view of \eqref{eq:galerkin_hyp} and \eqref{eq:galerkin_step}, the sequences  $\left\{\partial_{t}^{j}u_{\ell}\right\}_{\ell=1}^{\infty}$, $j=0,1,\dots,m+1$, and $\left\{\partial_{t}^{m+2}u_{\ell}\right\}_{\ell=1}^{\infty}$ are uniformly bounded in $L^{\infty}\left(0,T;H_{0}^{1}(\Omega)\right)$ and  $L^{\infty}\left(0,T;L^{2}(\Omega)\right)$, respectively. Hence, arguing as in the proof of Proposition \ref{prop:wellposedness_linear_base_step}, we conclude that
\[
\sum_{j=0}^{m+1}\left\|\partial_{t}^{j}u\right\|_{L^{\infty}(0,T;H^{1}(\Omega))}
+
\left\|\partial_{t}^{m+2}u\right\|_{L^{\infty}(0,T;L^{2}(\Omega))}
\lesssim
\|F\|_{H^{m+2}\left(0,T;L^{2}(\Omega)\right)}
\leq
\|F\|_{H^{m+2,m+1}(Q)} ,
\]
which immediately gives us the estimate \eqref{eq:CmH1_est}.

We now turn to establishing the regularity of $u$ in the spatial variable $x$. To this end, let us recall  that the function $u(t,\cdot)$ solves the elliptic boundary value problem  \eqref{eq:elliptic_problem_for_u} for almost every
$t\in(0,T)$. Since $\partial\Omega$ is smooth, it follows from \cite[Theorem 8.13] {gilbarg_trudinger_1983} that the following estimate holds for every $k\in\N\cup\{0\}$ such that $\mathfrak{F}(t,\cdot)\in H^{k}(\Omega)$ and almost every $t\in(0,T)$:
\begin{equation}
\label{eq:elliptic_shift}
\|u(t,\cdot)\|_{H^{k+2}(\Omega)}\lesssim\|\mathfrak{F}(t,\cdot)\|_{H^{k}(\Omega)} .
\end{equation}
We apply the  estimate above with $k=m+1$. By Lemma
\ref{Lemma:Winfty_Hm_product_estimate} and the monotonicity of Sobolev norms, we get 
\[
\|\mathfrak{F}\|_{L^{2}(0,T;H^{m+1}(\Omega))}
\lesssim
\|F\|_{L^{2}(0,T;H^{m+1}(\Omega))}
+\|\p_{t}u\|_{L^{2}(0,T;H^{m+1}(\Omega))}
+\|u\|_{L^{2}(0,T;H^{m+2}(\Omega))},
\]
where the implicit constant depends on $\|A\|_{W^{m+3,\infty}(Q)}$ and
$\|q\|_{W^{m+2,\infty}(Q)}$.

We next estimate the second and  third terms above. By the induction hypothesis \eqref{eq:induction_estimates}, the monotonicity of Sobolev norms, and the estimate \eqref{eq:elliptic_shift} with $k=m$, we have
\[
\|u\|_{L^{2}(0,T;H^{m+2}(\Omega))}\lesssim \|F\|_{H^{m+1,m}(Q)}\leq \|F\|_{H^{m+2,m+1}(Q)}. 
\]
To bound the second term, we have already established that $u\in H^{m+2}(0,T;L^{2}(\Omega))$. By the induction
hypothesis, it holds that $u\in L^{2}(0,T;H^{m+2}(\Omega))$. Hence, it follows from \cite[Chapter 4, Proposition 2.1]{LionsMagenes1968Vol2} that $u\in H^{1}(0,T;H^{m+1}(\Omega))$. Moreover, applying the interpolation estimate given in  \cite[Chapter 1, Proposition 2.3]{lions1972nonhomogeneous1}, we obtain
\[
\|\p_{t}u\|_{L^{2}(0,T;H^{m+1}(\Omega))}
\le
\|u\|_{H^{1}(0,T;H^{m+1}(\Omega))}
\lesssim
\|u\|_{H^{m+2}(0,T;L^{2}(\Omega))}^{\frac{1}{m+2}}
\|u\|_{L^{2}(0,T;H^{m+2}(\Omega))}^{\frac{m+1}{m+2}} .
\]
Thus, by the estimates \eqref{eq:induction_estimates} and \eqref{eq:CmH1_est}, it holds that
\begin{equation}
\label{eq:interp_dtu}
\|\p_{t}u\|_{L^{2}(0,T;H^{m+1}(\Omega))} \lesssim \|F\|_{H^{m+2,m+1}(Q)}. 
\end{equation}
By integrating the
square of \eqref{eq:elliptic_shift} with $k=m+1$ over the interval  $(0,T)$, and using the estimates \eqref{eq:induction_estimates}, 
\eqref{eq:CmH1_est}, and \eqref{eq:elliptic_shift}--\eqref{eq:interp_dtu}, we get
$u\in L^{2}\left(0,T;H^{m+3}(\Omega)\cap H^{1}_{0}(\Omega)\right)$, as well as the estimate 
\begin{equation}
\label{eq:spatial_final}
\|u\|_{L^{2}(0,T;H^{m+3}(\Omega))}
\lesssim
\|F\|_{H^{m+2,m+1}(Q)} .
\end{equation} 
From here, we obtain \eqref{eq:induction_step_estimate} by adding the estimates \eqref{eq:CmH1_est} and \eqref{eq:spatial_final}. This completes the induction step. 

Lastly, since the difference of any two solutions belongs to $H^{1}\left(0,T;L^{2}(\Omega) \right)\cap L^{2}(0,T;H_{0}^{1}(\Omega))$ and satisfies the homogeneous problem, uniqueness follows immediately from  Proposition \ref{prop:wellposedness_linear_base_step}. This completes the proof of
Lemma \ref{Lemma:wellposedness_linear_induction}.
\end{proof}

We next consider the case of nonhomogeneous Dirichlet boundary conditions. To establish the well-posedness, which is stated below, we reduce the problem to one with homogeneous Dirichlet boundary conditions and a non-zero source term. 

\begin{lem}
\label{Lemma:nonhom_wellposedness_linear}
Let $\Omega, Q$, and $\Sigma$ be the same as in Proposition \ref{prop:wellposedness_linear_base_step}. Let  $A \in W^{m+2,\infty}(Q;\mathbb{R}^n)$, $q \in  W^{m+1,\infty}(Q;\mathbb{C})$, $F\in H^{m+1,m}(Q)$, and $f\in H^{m+\frac52, m+\frac52}(\Sigma)$ for some  $m\in\mathbb{N}\cup\{0\}$.  Assume that  $\partial_{t}^{k}F(0,x)=0$ for almost every $x\in\Omega$ and all $k=0,\dots,m$, and that $\partial^{k}_{t}f(0,\cdot)|_{\partial\Omega}=0$ for $k=0,\dots,m+1$. Then the IBVP 
\begin{equation}
\label{eq:magnetic_schrodinger_problem}
\begin{cases}
\mathrm{i}\partial_t u
+\Delta u
+2\mathrm{i}A \cdot\nabla u
+\mathrm{i} (\nabla \cdot A)  u
-|A|^{2}u
+q u
=F
& \text{ in } Q,\\
u=f &  \text{ on } \Sigma,\\
u(0,\cdot)=0 &   \text{ in } \Omega, 
\end{cases}
\end{equation}
admits a unique solution $u\in H^{m+1,m+2}(Q)$, 
which satisfies the compatibility condition $\partial_{t}^{k}u(0,x)=0$ for almost every  $x\in \Omega$ and all $k=0,1,\dots ,m$. Moreover, we have the  estimate
\begin{equation}
\label{eq:u_final_estimate}
\|u\|_{H^{m+1,m+2}(Q)}
\lesssim
\|F\|_{H^{m+1,m}(Q)}
+ \|f\|_{H^{m+\frac52, m+\frac52}(\Sigma)}.
\end{equation}
Here the implicit constant depends on $\Omega$, $T$, $m$, $\|A\|_{W^{m+2,\infty}(Q)}$, $\|q\|_{W^{m+1,\infty}(Q)}$, but is independent of $u$. 
\end{lem}

\begin{proof}
By \cite[Chapter~4, Theorem~2.3]{LionsMagenes1968Vol2}, there exists a function $w \in H^{m+3,m+3}(Q)$ such that
\begin{equation}
\label{eq:w_conditions}
w|_{\Sigma} = f, \quad \text{ and } \quad 
\partial_t^{k} w(0,x)=0 \quad \text{for almost every  } x\in\Omega \text{ and all } k=0,1,\dots,m+2,
\end{equation}
and satisfies the estimate 
\begin{equation}
\label{eq:w_estimate}
\|w\|_{H^{m+3,m+3}(Q)} 
\lesssim \|f\|_{H^{m+\frac52, m+\frac52}(\Sigma)},
\end{equation}
where the implicit constant depends on $m$, $\Omega$, and $T$, but is independent of $w$.

By Lemma \ref{Lemma:wellposedness_linear_induction}, the IBVP 
\begin{equation}
\label{eq:homogeneous_boundary_problem_v}
\begin{cases}
\mathrm{i}\partial_t v
+\Delta v
+2\mathrm{i}A\cdot\nabla v
+\mathrm{i}(\nabla\cdot A)v
-|A|^{2}v
+qv
= F - \widetilde{F}_{1} & \text{ in } Q,\\
v=0 & \text{ on } \Sigma,\\
v(0,\cdot)=0 & \text{ in } \Omega,
\end{cases}
\end{equation}
where
$\widetilde{F}_{1} = \mathrm{i}\partial_t w +\Delta w +2\mathrm{i}A\cdot\nabla w +\mathrm{i}(\nabla\cdot A)w -|A|^{2}w +qw\in H^{m+1,m}(Q)$,
has a unique solution $v\in H^{m+1,m+2}(Q)$ satisfying the following estimate:
\begin{align*}
\|v\|_{H^{m+1,m+2}(Q)}
\lesssim &
\|F\|_{H^{m+1,m}(Q)}+\|\partial_{t}w\|_{H^{m+1,m}(Q)}+ \|\Delta w\|_{H^{m+1,m}(Q)} 
+ \left\|A\cdot\nabla w\right\|_{H^{m+1,m}(Q)}
\\ & 
+ \|(\nabla\cdot A) w\|_{H^{m+1,m}(Q)} 
+ \||A|^2 w\|_{H^{m+1,m}(Q)} + \|q w\|_{H^{m+1,m}(Q)}.
\end{align*} 
We next estimate the quantities involving the coefficients $A$ and $q$. By Lemma \ref{Lemma:Winfty_Hm_product_estimate}, we obtain
\[
\|v\|_{H^{m+1,m+2}(Q)}
\lesssim  
\|F\|_{H^{m+1,m}(Q)}+\|\partial_{t}w\|_{H^{m+1,m}(Q)} + \|\Delta w\|_{H^{m+1,m}(Q)}
+ \| w\|_{H^{m+2}(Q)} .
\]
Hence, it holds that 
\begin{align*}
\|v\|_{H^{m+1,m+2}(Q)}
\lesssim
\left\|F\right\|_{H^{m+1,m}(Q)}
+ \left\| w\right\|_{H^{m+3,m+3}(Q)},
\end{align*}
where the implicit constant depends on $m$, $n$, $\|A\|_{W^{m+2,\infty}(Q)}$, and $\|q\|_{W^{m+1,\infty}(Q)}$.  Moreover, the estimate \eqref{eq:w_estimate} and the monotonicity of Sobolev norms yield that 
\begin{equation}
\label{eq:v_estimate_compact}
\|v\|_{H^{m+1,m+2}(Q)}
\lesssim  
\|F\|_{H^{m+1,m}(Q)}
+  
\|f\|_{H^{m+\frac52, m+\frac52}(\Sigma)}.
\end{equation}

We now let $u:=v+w \in H^{m+1, m+2}(Q)$. Then it is straightforward to see that $u$ satisfies the problem \eqref{eq:magnetic_schrodinger_problem}. Furthermore, we utilize the estimates \eqref{eq:w_estimate}  and \eqref{eq:v_estimate_compact}, as well as the triangle inequality, to obtain the claimed estimate \eqref{eq:u_final_estimate}. 

Let us next verify the compatibility condition  $\partial_{t}^{k}u(0,x)=0$ for almost every $x\in\Omega$ and all $k=0,1,\dots,m$. Since $v$ solves the IBVP \eqref{eq:homogeneous_boundary_problem_v}, by Lemma \ref{Lemma:wellposedness_linear_induction}, we get $\partial_{t}^{k}v(0,x)=0$ for almost every $x\in \Omega$ and all $k=0,1,\dots,m$. As $u=v+w$, the compatibility condition for $u$ follows immediately from \eqref{eq:w_conditions}.

Finally, by the same arguments as in the proof of Proposition \ref{prop:wellposedness_linear_base_step}, we obtain the uniqueness of the solution to the IBVP \eqref{eq:magnetic_schrodinger_problem}. 
This completes the proof of Lemma \ref{Lemma:nonhom_wellposedness_linear}.
\end{proof}

We are now ready to study the well-posedness of the nonlinear  IBVP \eqref{eq:ibvp_nonlinear}, provided that the boundary data $f$ is sufficiently small in a suitable sense. Let us begin by introducing some notations used to describe the nonlinear  electric potential $q(t,x,u)$. Let $\Psi$ be a holomorphic function in a neighborhood of $z=0$, where $z\in\mathbb{C}$. Then $\Psi$ admits the following power series expansion around $z=0$:
\[
\Psi(z)
= \Psi(0) + \Psi'(0)z + z^2 \widetilde{\Psi}(z),
\]
where
\[
\widetilde{\Psi}(z)
:= \sum_{k =2}^\infty \frac{\Psi^{(k)}(0)}{k!} z^{k-2}.
\]

In view of the expansion \eqref{eq:expansion_q}, the nonlinear electric potential $q$ can be written as
\[
q(t,x,u)
=
q_{1}(t,x) u
+\widetilde{q}(t,x,u) u^2,
\]
where 
\[
\widetilde{q}(t,x,u):=\sum_{k =2}^\infty \frac{q_{k}(t,x)}{k!} u^{k-2}.
\]
Let us define the corresponding linearized operator $\mathcal{L}_{\mathcal{A},q_1}$ by
\[
\mathcal{L}_{\mathcal{A},q_1}u
:= \mathrm{i}\partial_t u
- \left(-\mathrm{i}\nabla + \mathcal{A}(t,x)\right)^2 u
+ q_{1}(t,x) u.
\]
Then the nonlinear operator $\mathcal{L}_{\mathcal{A},q}$ defined in \eqref{eq:magnetic_operator} can be rewritten as
\begin{align*}
\mathcal{L}_{\mathcal{A},q}u
= \mathcal{L}_{\mathcal{A},q_1}u - \mathcal{R}(u),
\end{align*}
where the nonlinear remainder $\mathcal{R}(u)= - \widetilde{q}(t,x,u) u^2$. 
Therefore, the nonlinear IBVP \eqref{eq:ibvp_nonlinear} can be reformulated as
\begin{equation}\label{eq:nonlinear_ibvp}
\begin{cases}
\mathcal{L}_{\mathcal{A},q_1}u = \mathcal{R}(u) & \text{in } Q, \\
u = f & \text{on } \Sigma, \\
u(0,\cdot) = 0 & \text{in } \Omega.
\end{cases}
\end{equation}

We are now ready to state and prove the well-posedness of the nonlinear IBVP \eqref{eq:ibvp_nonlinear}  with small Dirichlet data $f$.

\begin{thm}
\label{thm:wellposedness}
Let $\Omega$, $Q$, and $\Sigma$ be the same as in Proposition \ref{prop:wellposedness_linear_base_step}. Assume that  $m\in\mathbb{N}$ is such that $m>n+1$, and  $\mathcal{A}\in W^{m+2,\infty}(Q;\mathbb{R}^{n})$.
Suppose that $q(t,x,z)$ satisfies Assumptions \ref{assump2}-\ref{assump3}. Let $\varepsilon_0>0$ be sufficiently small. 
Then  for all $f\in H^{m+\frac52, m+\frac52}(\Sigma)$   such that $\partial^{k}_{t}f(0,\cdot)|_{\partial\Omega}=0$ for all integers $0\leq k \leq m+1$  and $\|f\|_{ H^{m+\frac52, m+\frac52}(\Sigma)} \le \varepsilon_0$,  the nonlinear IBVP \eqref{eq:ibvp_nonlinear}  admits a unique solution $u\in H^{m+1,m+2}(Q)$ that satisfies the compatibility conditions 
$\partial_{t}^{k} u(0,x)=0$ for almost every  $x\in\Omega$ and  all $ k=0,1,\dots, m$. Moreover, the  following estimate holds: 
\begin{equation}
\label{eq:est_solution_nonlinear}
\|u\|_{H^{m+1,m+2}(Q)} 
\lesssim \|f\|_{ H^{m+\frac52, m+\frac52}(\Sigma)},
\end{equation}
where the implicit constant depends on $m$, $\|\mathcal{A}\|_{W^{m+2,\infty}(Q)}$, $\|q\|_{W^{m+1,\infty}(Q)}$, but is independent of $u$ and $f$. 
\end{thm}

\begin{proof}
We will use the standard fixed point argument. For this purpose, we shall work with the IBVP \eqref{eq:nonlinear_ibvp}. This allows us to separate the linear part from the nonlinear remainder, which is well-suited to apply the fixed point argument. Then we introduce an appropriate map on a suitable Banach space. We first show that this map  is well-defined, followed by establishing that it  is a contraction. Therefore, from the Banach fixed point theorem, we get a unique fixed point for this map, which is indeed a solution to the IBVP \eqref{eq:nonlinear_ibvp}.

Motivated by the discussion above, let us define $\mathcal K^{[0,R]}$, which is a closed subset of the Sobolev space $H^{m+1,m+2}(Q)$, as follows:
\begin{align*}
\mathcal{K}^{[0,R]} := &\left \{ g \in H^{m+1,m+2}(Q) :
\|g\|_{H^{m+1,m+2}(Q)} \leq R, \right.
\\
& \left. \qquad 
\partial_t^{k} g(0,x) = 0 \text{ for almost every  }x\in\Omega \text{ and all } k=0,\dots,m \right\},
\end{align*}
where $R>0$ will be chosen later. We also equip $\mathcal{K}^{[0,R]}$  with the $\|\cdot\|_{H^{m+1,m+2}(Q)}$-norm .

Let $v \in \mathcal{K}^{[0,R]}$. Consider the linear IBVP
\begin{equation}
\label{eq:linear_fixed_point}
\begin{cases}
\mathcal{L}_{\mathcal{A},q_1}u = \mathcal{R}(v) & \text{in } Q,\\[0.3em]
u = f & \text{on } \Sigma,\\[0.3em]
u(0,\cdot) = 0 & \text{in } \Omega,
\end{cases}
\end{equation}
where the nonlinear remainder $	\mathcal{R}(v)$ is given by
\begin{equation}
\label{eq:nonlinear_remainder}
\begin{aligned}
\mathcal{R}(v)
:= - \widetilde{q}(t,x,v) v^2.
\end{aligned}
\end{equation}

Since $v \in \mathcal{K}^{[0,R]}$, the coefficients $\mathcal{A} \in W^{m+2,\infty}(Q)$, and $\widetilde{q}(\cdot,\cdot,v)  \in W^{m+1,\infty}(Q)$, it follows from Lemma \ref{Lemma:banach_algebra_Hm} and the definition of $\mathcal{K}^{[0,R]}$ that $\mathcal{R}(v) \in H^{m+1,m}(Q)$
and  $\partial_{t}^{k}\mathcal{R}(v)(0,x)=0$ for almost every $x\in\Omega$ and all $k=0,1,\dots,m$. By Lemma \ref{Lemma:nonhom_wellposedness_linear}, the IBVP  \eqref{eq:linear_fixed_point} admits a unique solution
$u \in H^{m+1,m+2}(Q)$ 
satisfying the estimate
\[
\|u\|_{H^{m+1,m+2}(Q)}
\lesssim 
\|\mathcal{R}(v)\|_{H^{m+1,m}(Q)}
+
\|f\|_{H^{m+\frac52, m+\frac52}(\Sigma)}.
\]
Thus,  using \eqref{eq:nonlinear_remainder}, as well as the estimate \eqref{eq:Winfty_Hm_product_estimate}, we conclude that
\[
\|u\|_{H^{m+1,m+2}(Q)}  
\lesssim  
\|\widetilde q(v)\|_{W^{m+1,\infty}(Q)}\|v^{2}\|_{H^{m+1}(Q)}
+
\|f\|_{H^{m+\frac52, m+\frac52}(\Sigma)}.
\]
Since $v\in \mathcal{K}^{[0,R]}$, we infer from the estimate \eqref{eq:banach_algebra_estimate} that $\|v^2\|_{H^{m+1}(Q)}
\lesssim 
R^2$.
Let us  now assume that  $0<R\leq 1$. Then we immediately get
\begin{align*}
\|u\|_{H^{m+1,m+2}(Q)}
\lesssim
R^2
+
\|f\|_{ H^{m+\frac52, m+\frac52}(\Sigma)},
\end{align*}
where the implicit constant depends on $m$, $n$ and $\|\widetilde q\|_{L^{\infty}\left(\mathbb{C};W^{m+1,\infty}(Q)\right)}$.
Therefore, by choosing $\varepsilon_{0}>0$ and $R>0$ sufficiently small,
we obtain the estimate
\begin{equation}
\label{eq:est_solution_linearized}
\|u\|_{H^{m+1,m+2}(Q)}
\le R.
\end{equation}

Consider  the operator $\mathcal{L}^{[0,R]} :
\mathcal{K}^{[0,R]} \to \mathcal{K}^{[0,R]}$ defined by 
\begin{equation}
\label{eq:fixed_point_operator}
\mathcal{L}^{[0,R]}(v) = u,
\end{equation}
where $u$ is the solution to the linear IBVP \eqref{eq:linear_fixed_point}  associated with
$v \in \mathcal{K}^{[0,R]}$. We first note that the estimate \eqref{eq:est_solution_linearized} yields that $\mathcal{L}^{[0,R]}$ is well-defined.

Next, we show that $\mathcal{L}^{[0,R]}$ is a contraction for a sufficiently small $R>0$ that will be specified later.
Let $v_j\in \mathcal{K}^{[0,R]}$, $j=1,2$,  and let
$u_j = \mathcal{L}^{[0,R]}(v_j)$. Then $w := u_1 - u_2$ satisfies the following IBVP: 
\[
\begin{cases}
\mathcal{L}_{\mathcal{A},q_{1}} w
= \mathcal{R}(v_1,v_2) & \text{in } Q,\\[0.3em]
w= 0 & \text{on } \Sigma,\\[0.3em]
w(0,\cdot) = 0 & \text{in } \Omega,
\end{cases}
\]
where
\[
\mathcal{R}(v_{1},v_{2})
:= \mathcal{R}(v_{1})-\mathcal{R}(v_{2})=-\widetilde q(v_1)(v_1-v_2)(v_1+v_2) +\left(\widetilde q(v_2)-\widetilde q(v_1)\right)v_2^2.
\]
Therefore, by utilizing the estimates  \eqref{eq:Winfty_Hm_product_estimate} and \eqref{eq:banach_algebra_estimate}, as well as the triangle inequality, we deduce that
\begin{equation}
\label{eq:R_estimate_int}
\begin{aligned}
&\|\mathcal{R}(v_1,v_2)\|_{H^{m+1,m}(Q)}
\\
&\lesssim
\|\widetilde{q}(v_2)-\widetilde{q}(v_1)\|_{H^{m+1}(Q)}
\|v_2^2\|_{H^{m+1}(Q)}
+
\|\widetilde{q}(v_1)\|_{W^{m+1,\infty}(Q)}
\|(v_1-v_2)(v_1+v_2)\|_{H^{m+1}(Q)}. 
\end{aligned}
\end{equation}

We now derive an upper bound for the quantities appearing on the right-hand side of \eqref{eq:R_estimate_int}. To estimate the second term, it follows from the estimate  \eqref{eq:banach_algebra_estimate}, together with the fact that $v_{j}\in \mathcal{K}^{[0,R]} $, $j=1,2$, that
\begin{equation}
\label{eq:product_estimates_contraction}
\begin{aligned}
& \|v_2^2\|_{H^{m+1}(Q)}
\lesssim  R^2 \quad\text{ and }\quad 	\|(v_1-v_2)(v_1+v_2)\|_{H^{m+1}(Q)}
\lesssim R \|v_1-v_2\|_{H^{m+1}(Q)}.
\end{aligned}
\end{equation}
For the first term of \eqref{eq:R_estimate_int}, by the fundamental theorem of calculus, we get
\[
\widetilde{q}(v_{1})-\widetilde{q}(v_{2})=\int_{0}^{1}
\frac{d}{ds}\widetilde{q}(sv_{1}+(1-s)v_{2}) ds
=
\int_{0}^{1}
D\widetilde{q}\bigl(sv_{1}+(1-s)v_{2}\bigr)(v_{1}-v_{2})
ds.
\]
Here $D\widetilde{q}(v)$ denotes the  derivative of  $\widetilde{q}$ evaluated at.
Therefore, we utilize the estimate \eqref{eq:banach_algebra_estimate} to deduce  that
\begin{equation}
\label{eq:q_tilde_difference_estimate}
\begin{aligned}
\left\|\widetilde{q}(v_{1})-\widetilde{q}(v_{2})\right\|_{H^{m+1}(Q)}
&\le
\int_{0}^{1}
\left\|
D\widetilde{q}\bigl(sv_{1}+(1-s)v_{2}\bigr)(v_{1}-v_{2})
\right\|_{H^{m+1}(Q)} ds  
\\
&\lesssim
\int_{0}^{1}
\left\|
D\widetilde{q}\bigl(sv_{1}+(1-s)v_{2}\bigr)
\right\|_{W^{m+1,\infty}(Q)}
\|v_{1}-v_{2}\|_{H^{m+1}(Q)} ds 
\\
&\lesssim
\left\|\widetilde{q}\right\|_{W^{1,\infty}(\mathbb{C};W^{m+1,\infty}(Q))}
\|v_{1}-v_{2}\|_{H^{m+1}(Q)}.
\end{aligned}
\end{equation}

By substituting the estimates  \eqref{eq:product_estimates_contraction} and \eqref{eq:q_tilde_difference_estimate}  into \eqref{eq:R_estimate_int}, 
and incorporating the assumption $0<R\le 1$, 
we obtain

\begin{equation}
\label{eq:R_difference_estimate}
\|\mathcal{R}(v_1,v_2)\|_{H^{m+1,m}(Q)}
\lesssim
R
\|v_{1}-v_{2}\|_{H^{m+1}(Q)},
\end{equation}
where the implicit constant  depends only on $m$, $n$ and $\left\|\widetilde{q}\right\|_{W^{1,\infty}\left(\mathbb{C};W^{m+1,\infty}(Q)\right)}$. 
Thus, an application of Lemma \ref{Lemma:nonhom_wellposedness_linear} yields that $w\in\mathcal{K}^{[0,R]} $ and satisfies the following estimate:
\begin{align*}
\left\|w\right\|_{H^{m+1,m+2}(Q)}
\lesssim
\left\|\mathcal{R}(v_{1},v_{2}) \right\|_{H^{m+1,m}(Q)}\lesssim R \|v_{1}-v_{2}\|_{H^{m+1}(Q)}
\lesssim R \|v_{1}-v_{2}\|_{H^{m+1,m+2}(Q)}.
\end{align*}
Let us now  choose $R_0>0$ such that the  product of the implicit constant above and $R_0$ is less than $1$ and fix  
$R\in(0,R_0]$. Hence, the map $\mathcal{L}^{[0,R]}$ defined by \eqref{eq:fixed_point_operator} is a contraction on
$\mathcal K^{[0,R]}$. By the Banach fixed point theorem, $\mathcal{L}^{[0,R]}$ admits a unique fixed point
$u\in\mathcal K^{[0,R]}$. Furthermore, due to the fact that $\cR(0)=0$, we conclude that $u$ satisfies the following estimate  for any
$f\in H^{m+\frac52, m+\frac52}(\Sigma)$:
\begin{align*}
\|u\|_{H^{m+1,m+2}(Q)}
&\lesssim 
\|\mathcal R(u,0)\|_{H^{m+1,m}(Q)}
+  \|f\|_{ H^{m+\frac52, m+\frac52}(\Sigma)}
\\
&\lesssim R\|u\|_{H^{m+1,m+2}(Q)}
+  \|f\|_{ H^{m+\frac52, m+\frac52}(\Sigma)},
\end{align*}
where the last step follows from the estimate \eqref{eq:R_difference_estimate}.
From here, we obtain
the estimate \eqref{eq:est_solution_nonlinear} immediately. 
This completes the proof of Theorem \ref{thm:wellposedness}.
\end{proof}

\section{Construction of Geometric Optics Solutions}
\label{sec:construction_GO_solution}

In this section, we construct geometric optics solutions by  following the methods developed in \cite{Kian_Soccorsi, Lai_Lu_Zhou}. More precisely, we first construct solutions for the linearized operator $\mathcal{L}_{\mathcal{A},q_{1}}$ that vanish at $t=0$. These solutions depend on a large parameter $\lambda>0$ and contain a principal part and a remainder term. 
On the other hand, the remainder term will vanish in a suitable norm as $\lambda \to \infty$.

The subtlety and the difficulty of this construction are as follows. To obtain the uniqueness of coefficients $\mathcal{A}$ and $q_{1}$, it suffices to prove an $H^{1,1}(Q)$-decay for the remainder term, which is stated in Proposition \ref{prop: geom opt sol 1}. However, to achieve the uniqueness of the higher order coefficients $q_{i}$, $i\geq 2$, 
we need a stronger decay estimate. Thus, we also prove the decay of the remainder term in $H^{m+1}(Q)$ for sufficiently large $m$ such that $H^{m+1}(Q)$ is a Banach algebra  in Corollary \ref{Corollary: geom opt sol 1}.

Finally, in Lemma \ref{Lemma: geom opt sol_adjoint} and Corollary \ref{Corollary:adjoint geom opt sol 1}, we construct  solutions for the adjoint operator $\mathcal{L}_{\mathcal{A},q_{1}}^\ast$ that vanish at $t=T$. The construction follows the same strategy as in the previous case, combined with a time-reversal argument.

Our first result in this section provides the existence of a smooth cutoff function that is constant on prescribed intervals, and its derivatives satisfy suitable decay estimates.

\begin{lem}
\label{Lemma:time_cutoff}
For every $0<\varepsilon<\frac{T}{4}$, there exists a function
$\chi_{\varepsilon}\in C_{c}^{\infty}(0,T)$ satisfying the following properties:
\begin{enumerate}
\item  $0 \leq \chi_\varepsilon \leq 1$.

\item $\chi_\varepsilon(t) = 0$ on $[0,\varepsilon] \cup [T-\varepsilon, T]$.

\item $\chi_\varepsilon(t) = 1$ on $[2\varepsilon, T-2\varepsilon]$.

\item For every integer $k \geq 0$, there exists a constant $C_k>0$, which is independent of $\varepsilon$, such that
\begin{equation}
\label{eq:cutoff_derivative_estimate}
\left|\frac{d^{k}}{dt^{k}} \chi_\varepsilon(t) \right|
\le C_k \varepsilon^{-k} 
\quad \text{for all } t \in \mathbb{R}.
\end{equation}
\end{enumerate}
\end{lem}
\begin{proof}
Define the function $ f : \mathbb{R} \to \mathbb{R} $ by
\[
f(t) :=
\begin{cases}
e^{-1/t}, & t>0,\\
0, & t \leq 0.
\end{cases}
\]
Observe that
$f\in C^\infty(\mathbb{R})$.
We then define another function
$ h : \mathbb{R} \to \mathbb{R} $ by
$h(t) := \frac{f(t)}{f(t) + f(1-t)}$. 
Then it follows immediately that $h\in C^\infty(\R)$ and
\begin{equation}
\label{eq:properties_h_cutoff}
h(t) =
\begin{cases}
0, & t \leq 0,\\
0< h(t) < 1, & 0 < t < 1,\\
1, & t \geq 1.
\end{cases}
\end{equation}

For any fixed  $0<\varepsilon<\frac{T}{4}$, let us define the function $\chi_\varepsilon:(0,T)\to[0,1]$
by
\[
\chi_\varepsilon(t)
:=
\begin{cases}
0, & t\in [0,\varepsilon],
\\
h\left(\frac{t-\varepsilon}{\varepsilon}\right),
& t \in (\varepsilon,  2\varepsilon),
\\ 
1, & t \in [2\varepsilon,  T-2\varepsilon],
\\ 
h\left(\frac{T-\varepsilon-t}{\varepsilon}\right),
& t \in (T-2\varepsilon,  T-\varepsilon),
\\ 
0, & t\in [  T-\varepsilon, T].
\end{cases}
\]
By  \eqref{eq:properties_h_cutoff}, we have $\chi_\varepsilon\in C^\infty(0,T)$ and $0\leq \chi_\varepsilon\leq 1$. This proves (1). 

Let us next verify (2) and (3). To this end, we note that $\chi_\varepsilon(t)=0$
on $[0,\varepsilon]\cup[T-\varepsilon,T]$. This yields that 
$\supp (\chi_\varepsilon)\subseteq[\varepsilon,T-\varepsilon]$ and $\chi_\varepsilon(t) = 1$ on $[2\varepsilon, T-2\varepsilon]$. 
Thus, we conclude that  $\chi_\varepsilon\in C_c^\infty(0,T)$. 

It remains to prove the  estimate \eqref{eq:cutoff_derivative_estimate}. 
When $k=0$, we obtain the estimate immediately from the bound $0\leq \chi_\varepsilon \leq 1$. In particular, we take $C_0=1$ in this case.

When $k\geq 1$, we first observe that all derivatives vanish on the intervals where
$\chi_\varepsilon$ is constant. On
$(\varepsilon,2\varepsilon)$, we compute that
\[
\frac{d^k}{dt^k}\chi_\varepsilon(t)
=
\varepsilon^{-k}
h^{(k)}\left(\frac{t-\varepsilon}{\varepsilon}\right),
\]
where $h^{(k)}$ denotes the $k^{\mathrm{th}}$ derivative of $h$. Similarly, on $(T-2\varepsilon,T-\varepsilon)$, we have
\[
\frac{d^k}{dt^k}\chi_\varepsilon(t)
=
(-1)^k\varepsilon^{-k}
h^{(k)}\left(\frac{T-\varepsilon-t}{\varepsilon}\right).
\]
Since $h^{(k)}$ is continuous on the compact interval $[0,1]$, it holds that
\[
\left|\frac{d^{k}}{dt^{k}}\chi_\varepsilon(t)\right|
\leq  \max_{s\in[0,1]} |h^{(k)}(s)| \varepsilon^{-k},
\qquad t\in\mathbb{R},
\]
which yields the estimate  \eqref{eq:cutoff_derivative_estimate} with $C_k:=\max_{s\in[0,1]} |h^{(k)}(s)|$. This completes the proof of Lemma \ref{Lemma:time_cutoff}.
\end{proof}

Our next result states the existence of amplitude terms in the geometric optics solutions, which satisfy certain transport equations. Throughout this section, $\mathcal{L}_{A,q}$ denotes the linear dynamical Schr\"odinger operator
\begin{equation}
\label{eq:def_linear}
\mathcal{L}_{A,q}  :=\mathrm{i}\partial_t 
+\Delta 
+2\mathrm{i}A(t,x)\cdot\nabla 
+\mathrm{i}\nabla \cdot A(t,x) 
-|A(t,x)|^{2}
+q(t,x).
\end{equation}
\begin{lem}
\label{lem:transport}
Let $Q$ be the same as in Proposition \ref{prop:wellposedness_linear_base_step}, and let $A\in C^\infty(Q, \R^n)$, $q\in C^{\infty}( \overline{Q};\C)$.
For any integer $N\geq 0$ and any vector $\alpha\in\R^n\setminus\{0\}$, there exist functions
$m_{j}\in C^{\infty}({\overline{Q}})$, $0\leq j\leq N$, that satisfy
the following system of transport equations:
\begin{equation}
\label{eqn: Transport}
\begin{cases}
\alpha\cdot\left(\nabla+\mathrm{i}A \right)m_{0}=0,\\
-2\mathrm{i}\alpha\cdot\left(\nabla+\mathrm{i}A\right) m_{1}=\mathcal L_{A,q}m_{0},\\
-2\mathrm{i}\alpha\cdot\left(\nabla+\mathrm{i}A\right) m_{2}=\mathcal L_{A,q}m_{1},\\
\hspace{2.5cm}   \vdots    \\
-2\mathrm{i}\alpha\cdot\left(\nabla+\mathrm{i}A\right) m_{N}=\mathcal  L_{A,q}m_{N-1}.
\end{cases}
\end{equation}
In particular, we have $
\partial_{t}^{k} m_j(0,x)=\partial_{t}^{k}m_j(T,x)=0$ for  $x\in\Omega$ and $k\in\mathbb{N}\cup\{0\}$.

\end{lem}
\begin{proof}
Our strategy is to construct explicit smooth solutions to  \eqref{eqn: Transport}. Let us write any $x\in\mathbb{R}^{n}$  uniquely as
\begin{equation}
\label{eq:x_decomposition_alpha}
x = x' + s\alpha,\quad   x'\in\alpha^{\perp},  \:   s\in\mathbb{R}.
\end{equation}
We first construct $m_0$. To this end, we define
\begin{equation}
\label{eq:theta_definition}
\Theta(t,x',s)
:=
\int_s^\infty \alpha\cdot \widetilde{A}(t,x'+p\alpha) dp,
\end{equation}  
where $\widetilde A\in C^\infty((0,T)\times \mathbb R^n)$ is an extension of $A$  that is compactly supported in the spatial variable. Let us note that such extension can be  obtained by the Sobolev extension theorem, see for instance
\cite[Section 5.4, Theorem 1]{evans_pde}. 

Let $\chi_\varepsilon\in C_c^\infty(0,T)$ be a cutoff function as in Lemma \ref{Lemma:time_cutoff}, and suppose
that $\eta\in C^\infty(\alpha^\perp)$. We then define
\begin{equation}
\label{eq:m0_definition}
m_0(t,x'+s\alpha)
:=
\chi_\varepsilon(t)\eta(x')
e^{\mathrm{i}\Theta(t,x',s)}.
\end{equation}
Then we have $m_0\in C^\infty(\overline{Q})$. To verify that $m_0$ satisfies the first equation in \eqref{eqn: Transport}, it follows from direct computations that 
\begin{align*}
\p_s m_0(t,x'+s\alpha)
&=
-\mathrm{i}\alpha\cdot \widetilde{A}(t,x'+s\alpha)m_0(t,x'+s\alpha),
\end{align*}
where we have used that $\tilde A$ is compactly supported. 
On the other hand, we compute that 
$$
\p_s m_0(t,x'+s\alpha)
=
\alpha\cdot\nabla m_0(t,x'+s\alpha).
$$
Furthermore, as $\widetilde{A}=A$ in $\overline Q$, it follows that
$$
\alpha\cdot\left(\nabla+\mathrm{i}A\right)m_0=0
\quad \text{in } Q.
$$
Moreover, since $\chi_\varepsilon\in C_c^\infty(0,T)$, we have $m_0(0,x)=m_0(T,x)=0$ in $\Omega$.

We now proceed to construct $m_j$, $1\leq j\leq N$, inductively. 
Suppose that the amplitude term 
$m_{j-1}\in C^\infty (\overline Q)$ satisfies the $(j-1)$-th transport equation in \eqref{eqn: Transport} and  $m_{j-1}(0,x)=m_{j-1}(T,x)=0$ for $x\in\Omega$. 
Let us recall that the $j$-th transport equation reads
\begin{equation}
\label{eq:jth_transport_equation}
-2\mathrm{i}\alpha\cdot(\nabla+\mathrm{i}A)m_j
=
\mathcal{L}_{A,q}m_{j-1}.
\end{equation}
and set 
\begin{equation}
\label{eq:mj_transport_solution_def}
m_j(t,x'+s\alpha)
:=
\frac{\mathrm{i}}{2}
e^{\mathrm{i}\Theta(t,x',s)}
\int_0^s
e^{-\mathrm{i}\Theta(t,x',\ell)}
\mathcal{L}_{A,q}m_{j-1}(t,x'+\ell\alpha) d\ell
\end{equation}
and observe that $m_{j}\in C^\infty(\overline{Q})$. 

To show that $m_{j}$ satisfies  \eqref{eq:jth_transport_equation},
we rewrite  \eqref{eq:mj_transport_solution_def} as
\begin{equation}\label{eq:mj-transport}
m_j(t,x'+s\alpha)e^{-\mathrm{i}\Theta(t,x',s)}
=
\frac{\mathrm{i}}{2}
\int_0^s
e^{-\mathrm{i}\Theta(t,x',\ell)}
\mathcal{L}_{A,q}m_{j-1}(t,x'+\ell\alpha)  d\ell .
\end{equation}
Differentiating both sides with respect to $s$, we obtain
\begin{equation}
\label{eq:derivative_mj_integrating_factor}
\p_s 
\left[
m_j(t,x'+s\alpha)e^{-\mathrm{i}\Theta(t,x',s)}
\right]
=
\frac{\mathrm{i}}{2}
e^{-\mathrm{i}\Theta(t,x',s)}
\mathcal{L}_{A,q}m_{j-1}(t,x'+s\alpha).
\end{equation}
Using \eqref{eq:theta_definition}, 
the left-hand side of equation \eqref{eq:derivative_mj_integrating_factor} becomes
\[
e^{-\mathrm{i}\Theta(t,x',s)}
\left[
\p_s m_j(t,x'+s\alpha)
+
\mathrm{i}\alpha\cdot \widetilde{A}(t,x'+s\alpha)m_j(t,x'+s\alpha)
\right].
\]
Hence, we get from the previous computations that 
\[
\p_s m_j(t,x'+s\alpha)
+
\mathrm{i}\alpha\cdot \widetilde{A}(t,x'+s\alpha)m_j(t,x'+s\alpha)
=
\frac{\mathrm{i}}{2}
\mathcal{L}_{A,q}m_{j-1}(t,x'+s\alpha).
\]
Moreover, since
\[
\p_s m_j(t,x'+s\alpha)
=
\alpha\cdot\nabla m_j(t,x'+s\alpha),
\]
and $\widetilde{A}$ is a compactly supported extension of $A$, we have
\[
\alpha\cdot(\nabla+\mathrm{i}A)m_j
=
\frac{\mathrm{i}}{2}\mathcal{L}_{A,q}m_{j-1} \quad \text{ in }Q.
\]
From here, we immediately obtain the $j^{\mathrm{th}}$ transport equation \eqref{eq:jth_transport_equation}.

It remains to verify  $
\partial_{t}^{k} m_j(0,x)=\partial_{t}^{k}m_j(T,x)=0$ for  $x\in\Omega$ and $k\in\mathbb{N}\cup\{0\}$. Since
$\chi_{\varepsilon}\in C_c^\infty(0,T)$, the function $m_0$ and its time derivative vanish in a neighborhood
of $t=0$ and $t=T$. Then it follows that  $\partial^{k}_{t}\mathcal{L}_{A,q}m_0=0$ at 
$t=0$ and $t=T$ for any $k\in \mathbb{N}\cup\{0\}$. Therefore, using \eqref{eq:mj-transport} with $j=1$ gives us $\partial^{k}_{t}m_1(0,x)=\partial^{k}_{t}m_1(T,x)=0$ for any $k\in \mathbb{N}\cup\{0\}$ and $x\in\Omega$. 
Repeating this argument inductively, we obtain $\partial^{k}_{t}m_j(0,x)=\partial^{k}_{t}m_j(T,x)=0$ for any $k\in \mathbb{N}\cup\{0\}$, $x\in\Omega$ and $0\le j\le N$. 
This completes the proof of Lemma \ref{lem:transport}.
\end{proof}

We are now ready to state and prove the existence of geometric optics solutions to the linear operator \eqref{eq:def_linear}.

\begin{prop}
\label{prop: geom opt sol 1}
Let $\Omega, Q$, and $\Sigma$ be as in Proposition \ref{prop:wellposedness_linear_base_step}.
Let $\mathcal{L}_{A,q}$ be the linear dynamical Schr\"odinger operator given by \eqref{eq:def_linear} with  $A \in C^{\infty}(\overline{Q};\R^n)$ and  $q \in C^{\infty}(\overline{Q};\C)$. Then there exists a constant $\lambda_0 \gg 1$ such that for all $\lambda\geq \lambda_0$, there exists a solution $v_{\lambda}\in H^{1,2}(Q)$ to the initial value problem 
\begin{equation}
\label{eq:ibvp_GO_solution}
\begin{cases}
\mathcal{L}_{A,q} v
=0
& \text{ in } Q,
\\
v(0,\cdot)=0 & \text{ in }\Omega, 
\end{cases}
\end{equation}
having the form
\[
v_{\lambda}(t,x)= e^{\mathrm{i}(\lambda x\cdot\alpha-\lambda^2\lvert\alpha\rvert^2t)}\bigg(m_{0}(t,x)+\frac{m_{1}(t,x)}{\lambda}+\frac{m_{2}(t,x)}{\lambda^2}+\cdots+\frac{m_{N}(t,x)}{\lambda^N}\bigg)+ R_{\lambda}(t,x),
\]
where  $\alpha \in \mathbb{R}^n\setminus\{0\}$, $N\geq 3$, and the amplitudes $m_j$ satisfy the transport equation  \eqref{eqn: Transport} for $0\leq j\leq N$. Moreover, the remainder term $R_{\lambda}\in H^{1,2}(Q)$ satisfies $R_{\lambda}(0,x)=0$ for $x\in\Omega$, as well as  the estimate
\begin{equation}
\label{eq:R_lambda_H12_estimate}
\|R_{\lambda}\|_{H^{1,2}(Q)}
\lesssim \frac{1}{\lambda} \left\| \mathcal{L}_{A,q}m_{N} \right\|_{H^{1}\left(0,T;L^{2}(\Omega)\right)}.
\end{equation}
Here the implicit constant depends only  on $\|A\|_{W^{2,\infty}(Q)}$, $\|q\|_{W^{1,\infty}(Q)}$, $\alpha$, and $Q$.
\end{prop}

\begin{proof}
The existence of smooth amplitude $m_j$, $0 \le j \le N$, was established in Lemma \ref{lem:transport}. Hence, it suffices to construct a  remainder term $R_{\lambda}\in H^{1,2}( Q)$ satisfying the IBVP
\begin{equation}\label{eq:R_lambda_ibvp}
\begin{cases}
\mathcal{L}_{A,q}R_{\lambda}=-\mathcal{L}_{A,q}\left[e^{\mathrm{i}(\lambda x\cdot\alpha-\lambda^2\lvert\alpha\rvert^2t)}\left(m_{0}(t,x)+\frac{m_{1}(t,x)}{\lambda}+\frac{m_{2}(t,x)}{\lambda^2}+\cdots+\frac{m_{N}(t,x)}{\lambda^N}\right)  \right] 
& \text{ in } Q,
\\
R_{\lambda}(0,\cdot)=0 & \text{ in }\Omega, \\
R_{\lambda}=0 & \text{ on }\Sigma,
\end{cases}
\end{equation}
and the estimate \eqref{eq:R_lambda_H12_estimate}. 
Since the amplitudes $m_j$, $0\le j\le N$, satisfies the respective transport equations in \eqref{eqn: Transport},  we get from direct computations that 
\begin{equation*}
\mathcal{L}_{A,q}\left(
e^{\mathrm{i}(\lambda x\cdot\alpha-\lambda^2|\alpha|^2t)}\left(
m_0+\frac{m_1}{\lambda}
+\frac{m_2}{\lambda^2}
+\cdots
+\frac{m_N}{\lambda^N}
\right)
\right)
=
\frac{1}{\lambda^N}
e^{\mathrm{i}(\lambda x\cdot\alpha-\lambda^2|\alpha|^2t)}
\mathcal{L}_{A,q}m_N.
\end{equation*}
Hence, the problem \eqref{eq:R_lambda_ibvp} reads
\begin{equation}
\label{eq:R_lambda_ibvp_simplified}
\begin{cases}
\mathcal{L}_{A,q}R_\lambda
=
-\frac{1}{\lambda^N}
e^{\mathrm{i}(\lambda x\cdot\alpha-\lambda^2|\alpha|^2t)}
\mathcal{L}_{A,q}m_N
& \text{ in } Q,\\
R_\lambda(0,\cdot)=0
& \text{ in } \Omega, \\
R_{\lambda}=0 &\text{ on }\Sigma.
\end{cases}
\end{equation}

Since $-\frac{1}{\lambda^N}
e^{\mathrm{i}(\lambda x\cdot\alpha-\lambda^2|\alpha|^2t)}
\mathcal{L}_{A,q}m_N \in H^{1}\left(0,T;L^{2}(\Omega)\right)$ and  $\mathcal{L}_{A,q}m_N(0,x)=0$ for $x\in \Omega$,
an application of  Proposition \ref{prop:wellposedness_linear_base_step} yields that the problem \eqref{eq:R_lambda_ibvp_simplified} admits a unique solution  $R_{\lambda}\in H^{1,2}(Q)$.  
Finally, we apply  the estimate \eqref{eq:final_regularity_estimate} to the problem above to obtain the estimate 
\begin{align*}
\|R_{\lambda}\|_{H^{1,2}(Q)}
&\lesssim  \frac{1}{\lambda^{N}}\left\| e^{\mathrm{i}(\lambda x\cdot\alpha-\lambda^2|\alpha|^2t)}\mathcal{L}_{A,q}m_{N} \right\|_{H^{1}\left(0,T;L^{2}(\Omega)\right)}\\
&\lesssim \frac{1}{\lambda^{N-2}}\left\| \mathcal{L}_{A,q}m_{N} \right\|_{H^{1}\left(0,T;L^{2}(\Omega)\right)}\\
&\lesssim \frac{1}{\lambda} \left\| \mathcal{L}_{A,q}m_{N} \right\|_{H^{1}\left(0,T;L^{2}(\Omega)\right)},
\end{align*}
for $N\geq 3$ and $\lambda\geq \lambda_{0}>1$. This completes the proof of Proposition \ref{prop: geom opt sol 1}.
\end{proof}

Our next result gives us the existence of  geometric optics solutions whose remainders decay in the Sobolev space $H^m(Q)$, where $m$ is sufficiently large. Such a decay rate is necessary to recover higher order nonlinearities.  
\begin{cor}
\label{Corollary: geom opt sol 1}
Let $\Omega, Q$, and $\Sigma$ be as in Proposition \ref{prop:wellposedness_linear_base_step}.
Let $\mathcal{L}_{A,q}$ be the linear dynamical Schr\"odinger operator given by \eqref{eq:def_linear} with  $A \in C^{\infty}(\overline{Q};\R^n)$ and  $q \in C^{\infty}(\overline{Q};\C)$. Let $m>n+1$.  Then there exists a constant $\lambda_0 \gg 1$ such that for all $\lambda\geq \lambda_0$, the IBVP \eqref{eq:ibvp_GO_solution} admits a solution $v_{\lambda}\in H^{m+1}(Q)$ of the form
\[
v_{\lambda}(t,x)= e^{\mathrm{i}(\lambda x\cdot\alpha-\lambda^2\lvert\alpha\rvert^2t)}\bigg(m_{0}(t,x)+\frac{m_{1}(t,x)}{\lambda}+\frac{m_{2}(t,x)}{\lambda^2}+\cdots+\frac{m_{N}(t,x)}{\lambda^N}\bigg)+ R_{\lambda}(t,x),
\]
where  $\alpha \in \mathbb{R}^n\setminus\{0\}$, $N\geq 2m+3$, and the amplitudes $m_j$ satisfy the transport equation  \eqref{eqn: Transport} for $0\leq j\leq N$. Moreover, the remainder term $R_{\lambda}\in H^{m+1}(Q)$ satisfies $R_{\lambda}(0,x)=0$ for $x\in\Omega$, and it satisfies the estimate 
\begin{equation}
\label{eq:est_remainder_higher_order}
\|R_{\lambda}\|_{H^{m+1}(Q)}
\lesssim
\frac{1}{\lambda}
\left\|\mathcal{L}_{A,q}m_N\right\|_{H^{m+1,m}(Q)}.
\end{equation}
Here the implicit constant depends only  on $\|A\|_{W^{m+2,\infty}(Q)}$, $\|q\|_{W^{m+1,\infty}(Q)}$, $\alpha$, and $Q$.
\end{cor}

\begin{proof}
We only need to seek the remainder $R_\lambda$ that satisfies the IBVP \eqref{eq:R_lambda_ibvp_simplified}. By the Sobolev embedding theorem, we have $ 
\bigcap_{k=0}^{\infty} H^k(Q)
= C^\infty(\overline{Q})
$.
As $-\frac{1}{\lambda^N}
e^{\mathrm{i}(\lambda x\cdot\alpha-\lambda^2|\alpha|^2t)}
\mathcal{L}_{A,q}m_N \in C^{\infty}(\overline{Q})$, it follows immediately that $-\frac{1}{\lambda^N}
e^{\mathrm{i}(\lambda x\cdot\alpha-\lambda^2|\alpha|^2t)}
\mathcal{L}_{A,q}m_N \in H^{m+1,m}(Q)$ for every $m\in\N\cup \{0\}$. Hence, by Lemma \ref{Lemma:wellposedness_linear_induction}, we get $R_{\lambda}\in H^{m+1}(Q)$ and 
\begin{align*}
\|R_{\lambda}\|_{H^{m+1}(Q)}
&\lesssim
\frac{1}{\lambda^N} \left\|
e^{\mathrm{i}(\lambda x\cdot\alpha-\lambda^2|\alpha|^2t)}
\mathcal{L}_{A,q}m_N\right\|_{H^{m+1,m}(Q)}.
\end{align*}

Moreover, since $m>n+1$, in view of Lemma \ref{Lemma:banach_algebra_Hm}, we have
\begin{align*}
\|R_{\lambda}\|_{H^{m+1}(Q)}
&\lesssim  \frac{1}{\lambda^N}  \left\|
e^{\mathrm{i}(\lambda x\cdot\alpha-\lambda^2|\alpha|^2t)}\right\|_{H^{m+1,m}(Q)}
\left\|\mathcal{L}_{A,q}m_N\right\|_{H^{m+1,m}(Q)}
\\
&\lesssim  \frac{1}{\lambda^{N-2m-2}}\left\|
\mathcal{L}_{A,q}m_N\right\|_{H^{m+1,m}(Q)}.
\end{align*}
Finally, the claimed estimate \eqref{eq:est_remainder_higher_order} follows immediately from above, as well as the fact that  $\lambda>1$ and $N-2m-2\geq 1$.
This completes the proof of Corollary \ref{Corollary: geom opt sol 1}.
\end{proof} 

We now turn   to the geometric optics solutions for  the adjoint operator $\mathcal{L}_{A,q}^\ast$ that vanish at $t=T$, where  $\mathcal{L}_{A,q}^\ast$ is given by
\begin{equation}
\label{eq:adjoint-operator}
\mathcal{L}_{A,q}^\ast = \mathrm{i}\partial_t v
+\Delta v
+2\mathrm{i}A(t,x)\cdot\nabla v
+\mathrm{i}\nabla \cdot A(t,x) v
-\left|A(t,x)\right|^{2}v
+\overline{q}(t,x)v.
\end{equation}
In the next two results, we give the existence of geometric optics solutions for $\mathcal{L}_{A,q}^\ast$ vanishing at $t=T$, whose remainder terms also vanish in a suitable sense as $\lambda \to \infty$. The proofs are very similar to those of Proposition \ref{prop: geom opt sol 1} and Corollary \ref{Corollary: geom opt sol 1}, and incorporate a time-reversal argument, which was also utilized in \cite{BhardwajKumarVashisth2026}. Thus, we shall omit the details. 

\begin{lem}
\label{Lemma: geom opt sol_adjoint}
Let $\Omega, Q$, and $\Sigma$ be as in Proposition \ref{prop:wellposedness_linear_base_step}.
Let $\mathcal{L}_{A,q}^\ast$ be the linear dynamical Schr\"odinger adjoint operator given by \eqref{eq:adjoint-operator}, where $A \in C^{\infty}(\overline{Q};\R^n)$ and  $q \in C^{\infty}(\overline{Q};\C)$.  Then there exists a constant $\lambda_0 \gg 1$ such that for all $\lambda\geq \lambda_0$, the problem 
\begin{equation}
\label{eq:adjoint_fbvp}
\begin{cases}
\mathcal{L}_{A,q}^\ast v
=0
& \text{ in } Q,
\\
v(T,\cdot)=0 & \text{ in }\Omega, 
\end{cases}
\end{equation}
admits a solution $u_{\lambda}\in H^{1,2}(Q)$ of  the form
\begin{equation*}
\label{eqn: Ansatz for v_adjoint}
u_{\lambda}(t,x)= e^{\mathrm{i}(\lambda x\cdot\alpha-\lambda^2\lvert\alpha\rvert^2t)}\left(\widetilde{m}_{0}(t,x)+\frac{\widetilde{m}_{1}(t,x)}{\lambda}+\frac{\widetilde{m}_{2}(t,x)}{\lambda^2}+\cdots+\frac{\widetilde{m}_{N}(t,x)}{\lambda^N}\right)+ \tilde{R}_{\lambda}(t,x),
\end{equation*}
where  $\alpha \in \mathbb{R}^n\setminus\{0\}$ and $N\geq3 $. Here the smooth amplitudes $\widetilde{m}_j$ satisfy  $
\widetilde{m}_j(0,x)=\widetilde{m}_j(T,x)=0$ for $x\in \Omega$, $0\leq j\leq N$, as well as the following system of transport equations:  
\begin{equation}\label{eqn: Transport_Adjoint}
\begin{cases}
\alpha\cdot\left(\nabla+\mathrm{i}A\right) \widetilde{m}_{0}=0,\\
-2\mathrm{i}\alpha\cdot\left(\nabla+\mathrm{i}A\right) \widetilde{m}_{1}=\mathcal{L}_{A,\overline{q}}\widetilde m_{0},\\
-2\mathrm{i}\alpha\cdot\left(\nabla+\mathrm{i}A\right) \widetilde{m}_{2}=\mathcal{L}_{A,\overline{q}}\widetilde m_{1},\\
\hspace{2.5cm}   \vdots    \\
-2\mathrm{i}\alpha\cdot\left(\nabla+\mathrm{i}A\right) \widetilde{m}_{N}=\mathcal{L}_{A,\overline{q}}\widetilde m_{N-1}.\\
\end{cases}
\end{equation}
Moreover,  the remainder term $\widetilde{R}_{\lambda}\in H^{1,2}(Q)$ satisfies $\widetilde{R}_{\lambda}(T,x)=0$ for $x\in\Omega$ and the estimate  
\begin{equation}\label{eq:R_adjoint_lambda_H12_estimate}
\|\tilde{R}_{\lambda}\|_{H^{1,2}(Q)}
\lesssim 
\frac{1}{\lambda}\left\| \mathcal{L}_{A,q} ^\ast \widetilde m_{N} \right\|_{H^{1}\left(0,T;L^{2}(\Omega)\right)}.
\end{equation}
\end{lem}

\begin{cor}
\label{Corollary:adjoint geom opt sol 1}
Let $\Omega, Q$, and $\Sigma$ be as in Proposition \ref{prop:wellposedness_linear_base_step}.
Let $\mathcal{L}_{A,q}^\ast$ be the linear dynamical Schr\"odinger adjoint operator given by \eqref{eq:adjoint-operator}, where   $A \in C^{\infty}(\overline{Q};\R^n)$ and  $q \in C^{\infty}(\overline{Q};\C)$. Let $m>n+1$.  Then there exists a constant $\lambda_0 \gg 1$ such that for all $\lambda\geq \lambda_0$, the problem \eqref{eq:adjoint_fbvp} has a solution $u_{\lambda}\in H^{m+1}(Q)$ of the form
\begin{equation*}
\label{eqn: Ansatz for v_copy}
u_{\lambda}(t,x)= e^{\mathrm{i}(\lambda x\cdot\alpha-\lambda^2\lvert\alpha\rvert^2t)}\left(\widetilde{m}_{0}(t,x)+\frac{\widetilde{m}_{1}(t,x)}{\lambda}+\frac{\widetilde{m}_{2}(t,x)}{\lambda^2}+\cdots+\frac{\widetilde{m}_{N}(t,x)}{\lambda^N}\right)+ \widetilde{R}_{\lambda}(t,x)
\end{equation*}
with  $\alpha \in \mathbb{R}^n\setminus\{0\}$ and  $N\geq 2m+3 $. Here the smooth amplitudes $\widetilde{m}_j$ satisfy the transport equation  \eqref{eqn: Transport_Adjoint} for $0\leq j\leq N$, and the remainder term $\widetilde{R}_{\lambda}\in H^{m+1}(Q)$ satisfies $\widetilde{R}_{\lambda}(T,x)=0$ for $x\in\Omega$, as well as the estimate
\begin{equation}\label{eq:R_lambda_Hm_estimate_adjoint}
\|\widetilde{R}_{\lambda}\|_{H^{m+1}(Q)}
\lesssim
\frac{1}{\lambda}
\left\|\mathcal{L}_{A,q}^\ast \widetilde m_{N}\right\|_{H^{m+1,m}(Q)}.
\end{equation}
\end{cor}

\section{Proof of Theorem \ref{thm:main_result}}
\label{sec:proof_full_data}

In this section, we establish the uniqueness of the linear potential $\cA$ and nonlinear potential $q$ given by \eqref{eq:expansion_q} from the full DN map \eqref{eq:def_DN_map}. The proof is split into several steps. First, in Subsection \ref{subsec:first_order_linearization}, we utilize the first order linearization, as well as the geometric optics solutions from Proposition \ref{prop: geom opt sol 1} and Lemma \ref{Lemma: geom opt sol_adjoint}, to recover the coefficients $\mathcal{A}$ and $q_1$. Due to the gauge invariance, we need the assumption $\nabla\cdot \mathcal{A}^{(1)}= \nabla\cdot \mathcal{A}^{(2)}$ in $Q$ to achieve the uniqueness of $\mathcal{A}$. Subsequently, we apply the second order linearization, in conjunction with the geometric   optics  solutions given in Corollaries \ref{Corollary: geom opt sol 1} and \ref{Corollary:adjoint geom opt sol 1} to obtain the uniqueness of  $q_2$ in Subsection \ref{sec:second_order_linearization}. Finally,  the recovery of higher order nonlinearities is established via the higher order linearization in Subsection \ref{subsec:higher_order_linearization}. 

\subsection{Recovery of $\mathcal{A}$ and $q_1$}
\label{subsec:first_order_linearization}
Let us start by establishing the uniqueness of the vector potential $\mathcal{A}$ and the scalar potential $q_{1}$ from the DN map. We shall use the linearization method, which has been extensively studied in the literature, see for instance \cite{BhardwajKumarVashisth2026,LassasLiimatainenLinSalo2021} and the references therein.

Let $\left(\mathcal{A}^{(j)},q^{(j)} \right)\in C^{\infty}(\overline Q;\R^n)\times C^{\infty}(\overline Q\times \C;\C)$, $j=1,2$, and let $f\in {H^{m+\frac52, m+\frac52}(\Sigma)}$  be such that $\partial^{k}_{t}f(0,\cdot)|_{\partial\Omega}=0$ for all $0\leq k \leq m+1$. Suppose that $\varepsilon_{0}>0$ is a constant as in Theorem \ref{thm:wellposedness}. Let $\varepsilon>0$ be   sufficiently small such that $\left\|\varepsilon f\right\|_{H^{m+\frac52, m+\frac52}(\Sigma)} \leq \varepsilon_{0}$, and let the function  $u_j (t,x; \varepsilon)$ be a solution to the nonlinear  IBVP
\[
\begin{cases}
\mathcal{L}_{\mathcal{A}^{(j)},q^{(j)}} u_j=0 & \text{ in } Q ,
\\
u_j=\varepsilon f& \text{ on } \Sigma,
\\
u_j(0, \cdot) =0 &\text{ in } \Omega,
\end{cases}
\]
where the operator $\mathcal{L}_{\mathcal{A}^{(j)},q^{(j)}}$ is given by \eqref{eq:magnetic_operator} with potentials $\mathcal{A}^{(j)}$ and $q^{(j)}$. Then it follows from Theorem  \ref{thm:wellposedness} that  $u_j (\cdot,\cdot; \varepsilon)\in H^{m+1,m+2}(Q)$ and satisfies the estimate
\[
\|u_j\|_{H^{m+1,m+2}(Q)}  
\lesssim
\|\varepsilon f\|_{H^{m+\frac52, m+\frac52}(\Sigma)}.
\]
Let us note that $u_{j}(t,x; 0)=0$. 

Next, we apply the first order linearization to the problem above. Using that $u_{j}(t,x; 0)=0$, we see that all terms containing positive powers of $u_j$ or $\nabla u_j$ vanish when $\varepsilon=0$. Therefore,  the function $v_{j}:= \p_{\varepsilon} u_j(t,x;\varepsilon)|_{\varepsilon=0}$ satisfies the   IBVP
\[
\begin{cases}
\mathcal{L}_{\mathcal{A}^{(j)},q_{1}^{(j)}} v_j=0 & \text{ in } Q ,
\\
v_j= f& \text{ on } \Sigma,
\\
v_j(0, \cdot) =0 &\text{ in } \Omega,
\end{cases}
\]
where $\mathcal{L}_{\mathcal{A}^{(j)},q_{1}^{(j)}}$ is a linear operator given by \eqref{eq:def_linear} with potentials $\mathcal{A}^{(j)}$ and $q_{1}^{(j)}$. 

Let us define the DN map $\mathcal{N}_{\mathcal{A}^{(j)},q_{1}^{(j)}}$ for the operator $\mathcal{L}_{\mathcal{A}^{(j)},q_{1}^{(j)}}$ by the formula
\[
\mathcal{N}_{\mathcal{A}^{(j)},q_{1}^{(j)}}(f) = (\p_\nu + \mathrm{i}\mathcal{A}^{(j)}\cdot\nu)v_j\big|_{\Sigma}. 
\]
Due to the assumption  $\Lambda_{\mathcal{A}^{(1)},q^{(1)}}(\varepsilon f) = \Lambda_{\mathcal{A}^{(2)},q^{(2)}}(\varepsilon f)$, applying $\p_\varepsilon|_{\varepsilon=0}$ on both sides gives us $\mathcal{N}_{\mathcal{A}^{(1)},q_{1}^{(1)}} = \mathcal{N}_{\mathcal{A}^{(2)},q_{1}^{(2)}}$.

In what follows, we shall denote $\mathcal{A}:= \mathcal{A}^{(2)} -\mathcal{A}^{(1)}$, and $q_{1}:= q_{1}^{(2)}-q_{1}^{(1)}$. Then it follows from direct computations that the function $v:= v_1-v_2$ satisfies the following  IBVP:
\begin{equation}
\label{eq:linearized_difference_v}
\begin{cases}
\mathcal{L}_{\mathcal{A}^{(1)},q_{1}^{(1)}} v
=
\left(
2\mathrm{i}\mathcal{A}\cdot \nabla
+\mathrm{i}\nabla \cdot \mathcal{A}
-|\mathcal{A}^{(2)}|^2
+|\mathcal{A}^{(1)}|^2
+q_{1}
\right)v_{2}
& \text{ in } Q,
\\
v=0 & \text{ on } \Sigma,
\\
v(0,\cdot)=0 & \text{ in } \Omega.
\end{cases}
\end{equation}
Let $v_{1}$ be a solution to the problem
\begin{equation}
\label{eq:adjoint_problem_first_linearization}
\begin{cases}
\mathcal{L}_{\mathcal{A}^{(1)},\overline{q_1^{(1)}}}v_1= 0 & \text{ in } Q ,
\\
v_{1}(T, \cdot) =0 &\text{ in } \Omega.
\end{cases}
\end{equation}
Multiplying the first equation of \eqref{eq:linearized_difference_v} by $\overline{v}_{1}$ and integrating over $Q$, we get
\begin{equation}
\label{eq:integral_identity_before_parts}
\int_{Q}\mathcal{L}_{\mathcal{A}^{(1)},q_{1}^{(1)}} v \overline{v}_{1} dx dt =  \int_{Q}\left(  2\mathrm{i}\mathcal{A}\cdot \nabla + \mathrm{i}\nabla \cdot \mathcal{A} -|\mathcal{A}^{(2)}|^2+|\mathcal{A}^{(1)}|^2 + q_{1} \right)v_{2} \overline{v}_{1} dx dt.
\end{equation}
In view of the homogeneous initial and boundary conditions of \eqref{eq:linearized_difference_v}, the problem \eqref{eq:adjoint_problem_first_linearization}, and that $\mathcal{N}_{\mathcal{A}^{(1)},q_{1}^{(1)}} = \mathcal{N}_{\mathcal{A}^{(2)},q_{1}^{(2)}}$,
we integrate by parts on the left-hand side of \eqref{eq:integral_identity_before_parts}
to get
\begin{equation}
\label{eq:integral_identity_A0_q1}
\int_Q
\left(
2\mathrm{i}\mathcal{A}\cdot \nabla v_2
+
\mathrm{i}(\nabla\cdot \mathcal{A})v_2
-
\big(|\mathcal{A}^{(2)}|^2-|\mathcal{A}^{(1)}|^2\big)v_2
+
q_1v_2
\right)
\overline{v}_1 dx dt
=0 .
\end{equation}

From Proposition \ref{prop: geom opt sol 1} and Lemma \ref{Lemma: geom opt sol_adjoint}, we get geometric optics solutions 
\begin{equation}
\label{eq:GO_solution_v_lambda}
v_2(t,x)
=
e^{\mathrm{i}(\lambda x\cdot \alpha-\lambda^2|\alpha|^2t)}
\left(
m_0(t,x)+\frac{m_1(t,x)}{\lambda}
+\frac{m_2(t,x)}{\lambda^2}
+\cdots
+\frac{m_N(t,x)}{\lambda^N}\right)
+R_\lambda(t,x),
\end{equation}
and
\begin{equation}
\label{eq:GO_solution_u_lambda_conjugate}
v_1(t,x)
=
e^{\mathrm{i}(\lambda x\cdot \alpha-\lambda^2|\alpha|^2t)}
\left(
\widetilde m_0(t,x)
+\frac{\widetilde m_1(t,x)}{\lambda}
+\frac{\widetilde m_2(t,x)}{\lambda^2}
+\cdots
+\frac{\widetilde m_N(t,x)}{\lambda^N}\right)
+\widetilde R_\lambda(t,x).
\end{equation}
Here the smooth amplitudes $m_j$ and $\tilde m_j$, $0 \le j\le N$, satisfy the transport equations with $A=\mathcal{A}^{(2)}$, $q=q_{1}^{(2)}$ in \eqref{eqn: Transport}, and with $A=\mathcal{A}^{(1)}$, $q=q_{1}^{(1)}$ in \eqref{eqn: Transport_Adjoint}, respectively, and the remainders $R_\lambda, \tilde R_\lambda \in H^{1,2}(Q)$ satisfy the estimates \eqref{eq:R_lambda_H12_estimate} and \eqref{eq:R_adjoint_lambda_H12_estimate}, respectively.

To simplify notation, in what follows, let us denote the phase function $\Phi_\lambda(t,x):=\lambda x\cdot\alpha-\lambda^2|\alpha|^2t$, 
and the sum of smooth amplitudes as 
\[
M_\lambda(t,x)
:=
\sum_{k=0}^{N}\frac{m_k(t,x)}{\lambda^k}
\quad\text{ and }\quad
\tilde M_\lambda(t,x)
:=
\sum_{k=0}^{N}\frac{\tilde m_k(t,x)}{\lambda^k}.
\]
Then the geometric optics solutions \eqref{eq:GO_solution_v_lambda} and \eqref{eq:GO_solution_u_lambda_conjugate} can be written as
\[
v_2=e^{\mathrm{i}\Phi_\lambda}M_\lambda+R_\lambda
\qquad \text{ and } \qquad
v_1
=
e^{\mathrm{i}\Phi_\lambda}\tilde M_\lambda
+
\widetilde R_\lambda.
\]
Furthermore,  we set
\[
W_{\mathrm{rem}}(t,x)
:=
\mathrm{i}\nabla\cdot \mathcal{A}(t,x)
-
|\mathcal{A}^{(2)}(t,x)|^2
+
|\mathcal{A}^{(1)}(t,x)|^2
+
q_1(t,x).
\]
Then the integral identity \eqref{eq:integral_identity_A0_q1} becomes
\begin{equation}
\label{eq:int_id_A0q1_new_notation}
\int_Q
\left(
2\mathrm{i}\mathcal{A}\cdot \nabla v_2
+
W_{\mathrm{rem}}v_2
\right)
\overline{v}_1 dx dt=0.
\end{equation}
On the other hand,  applying the product rule, we have
\[
\nabla v_2
=
e^{\mathrm{i} \Phi_\lambda(t,x)}
\left[
\mathrm{i}\lambda \alpha
M_\lambda (t,x)
+
\nabla
M_\lambda (t,x)
\right]
+
\nabla R_\lambda(t,x).
\]
Therefore, by substituting the previous expressions for $v_1,$ $v_2$, and $\nabla v_2$ into \eqref{eq:int_id_A0q1_new_notation}, we obtain
\begin{equation}
\label{eq:expanded_GO_identity}
\begin{aligned}
0
= &
-2\lambda
\int_Q
\mathcal{A}\cdot\alpha 
M_\lambda\overline{\tilde M_\lambda}
dx dt
+
2\mathrm{i}
\int_Q
\mathcal{A}\cdot\nabla M_\lambda 
\overline{\tilde M_\lambda}
dx dt
-2\lambda
\int_Q
e^{\mathrm{i}\Phi_\lambda}
\mathcal{A}\cdot\alpha 
M_\lambda\overline{\widetilde R_\lambda}
dx dt
\\
&+
2\mathrm{i}
\int_Q
e^{\mathrm{i}\Phi_\lambda}
\mathcal{A}\cdot \nabla M_\lambda 
\overline{\widetilde R_\lambda}
dx dt  
+
2\mathrm{i}
\int_Q
e^{-\mathrm{i}\Phi_\lambda}
\mathcal{A}\cdot \overline{\tilde M_\lambda} \nabla R_\lambda 
dx dt
+
2\mathrm{i}
\int_Q
\mathcal{A}\cdot  \nabla R_\lambda 
\overline{\widetilde R_\lambda}
dx dt
\\
&+
\int_Q
W_{\mathrm{rem}} M_\lambda\overline{\tilde M_\lambda}
dx dt
+
\int_Q
e^{\mathrm{i}\Phi_\lambda}
W_{\mathrm{rem}} M_\lambda\overline{\widetilde R_\lambda}
dx dt  
+
\int_Q
e^{-\mathrm{i}\Phi_\lambda}
W_{\mathrm{rem}} R_\lambda\overline{\tilde M_\lambda}
dx dt
\\
& +
\int_Q
W_{\mathrm{rem}} R_\lambda\overline{\widetilde R_\lambda}
dx dt .
\end{aligned}
\end{equation}

We now isolate the leading term in $M_\lambda\overline{\tilde M_\lambda}$. To this end, since $M_\lambda\overline{\tilde M_\lambda}
=
m_0\overline{\widetilde m_0}
+
\frac{1}{\lambda}G_\lambda,$ where
\[
G_\lambda
:=
m_1\overline{\widetilde m_0}
+
m_0\overline{\widetilde m_1}
+
\sum_{\substack{0\leq j,k\leq N,   j+k\geq 2}}
\frac{m_j\overline{\widetilde m_k}}{\lambda^{j+k-1}},
\]
we get
\[
-2\lambda
\int_Q
\mathcal{A}\cdot\alpha 
M_\lambda\overline{\tilde M_\lambda}
dx dt
=
-2\lambda
\int_Q
\mathcal{A}\cdot\alpha 
m_0\overline{\widetilde m_0}
dx dt
-
2
\int_Q
\mathcal{A}\cdot\alpha 
G_\lambda
dx dt .
\]
Therefore, \eqref{eq:expanded_GO_identity} can be written as
\begin{equation}\label{eq:principal_plus_error}
-2\lambda
\int_Q
\mathcal{A}\cdot\alpha 
m_0\overline{\widetilde m_0}
dx dt
+
\mathcal E_\lambda=0 ,
\end{equation}
where
\begin{align*}
\mathcal E_\lambda
:= &
-2
\int_Q
\mathcal{A}\cdot\alpha 
G_\lambda
dx dt
+
2\mathrm{i}
\int_Q
\mathcal{A}\cdot\nabla M_\lambda 
\overline{\tilde M_\lambda}
dx dt
-2\lambda
\int_Q
e^{\mathrm{i}\Phi_\lambda}
\mathcal{A}\cdot\alpha 
M_\lambda\overline{\widetilde R_\lambda}
dx dt
\\
&+
2\mathrm{i}
\int_Q
e^{\mathrm{i}\Phi_\lambda}
\mathcal{A}\cdot\nabla M_\lambda 
\overline{\widetilde R_\lambda}
dx dt  
+
2\mathrm{i}
\int_Q
e^{-\mathrm{i}\Phi_\lambda}
\mathcal{A}\cdot\nabla R_\lambda 
\overline{\tilde M_\lambda}
dx dt
+
2\mathrm{i}
\int_Q
\mathcal{A}\cdot\nabla R_\lambda 
\overline{\widetilde R_\lambda}
dx dt
\\
&
+
\int_Q
W_{\mathrm{rem}} M_\lambda\overline{\tilde M_\lambda}
dx dt
+
\int_Q
e^{\mathrm{i}\Phi_\lambda}
W_{\mathrm{rem}} M_\lambda\overline{\widetilde R_\lambda}
dx dt
+
\int_Q
e^{-\mathrm{i}\Phi_\lambda}
W_{\mathrm{rem}} R_\lambda\overline{\tilde M_\lambda}
dx dt
\\
&	+
\int_Q
W_{\mathrm{rem}} R_\lambda\overline{\widetilde R_\lambda}
dx dt .
\end{align*}

We next show that there exists a constant $\lambda_0$ such that 
\begin{equation}
\label{eq:E_lambda_uniform_bound}
|\mathcal E_\lambda|\leq C
\end{equation}
for all $\lambda\geq \lambda_0$, where $C>0$  is independent of $\lambda$. This requires us to estimate each term in $\mathcal{E}_\lambda$. To that end, we deduce from the inequality \eqref{eq:Linf_L2_product} and the Cauchy-Schwarz inequality that 
\begin{align*}
2
\left|
\int_Q
\mathcal{A}\cdot\alpha 
G_\lambda
dx dt
\right|
\le  & 
2|\alpha| 
\|\mathcal{A}\|_{L^2(Q)} |Q|^{1/2}
\left[
\|m_1\|_{L^\infty(Q)}
\|\widetilde m_0\|_{L^\infty(Q)}
+
\|m_0\|_{L^\infty(Q)}
\|\widetilde m_1\|_{L^\infty(Q)}
\right. 
\\
& \left. 
+
\sum_{\substack{0\leq j,k\leq N,    j+k\geq 2}}
\frac{
\|m_j\|_{L^\infty(Q)}
\|\widetilde m_k\|_{L^\infty(Q)}
}{\lambda^{j+k-1}}
\right].
\end{align*}
Since $\lambda\geq 1$,  there exists a constant $C>0$, which is independent of $\lambda$,  such that 
\begin{equation}
\label{eq:E_lambda_estimate_1}
2
\left|
\int_Q
\mathcal{A}\cdot\alpha 
G_\lambda
dx dt
\right|
\leq \frac{C}{10}.
\end{equation}
Similarly, we obtain the following two estimates:
\begin{equation}
\label{eq:E_lambda_estimate_2}
\begin{aligned}
2
\left|
\int_Q
\mathcal{A}\cdot\nabla M_\lambda 
\overline{\tilde M_\lambda}
dx dt
\right|
&\le 
2
\|\mathcal{A}\|_{L^2(Q)}
\sum_{k=0}^{N}\sum_{\ell=0}^{N}
\frac{
\|\nabla m_k \overline{\widetilde m_\ell}\|_{L^2(Q)}
}{\lambda^{k+\ell}}
\\
& \le 
2|Q|^{1/2}
\|\mathcal{A}\|_{L^2(Q)}
\sum_{k=0}^{N}\sum_{\ell=0}^{N}
\frac{
\|\nabla m_k\|_{L^\infty(Q)}
\|\widetilde m_\ell\|_{L^\infty(Q)}
}{\lambda^{k+\ell}}\leq \frac{C}{10},
\end{aligned}
\end{equation}
and
\begin{equation}
\label{eq:E_lambda_estimate_3}
\left|
\int_Q
W_{\mathrm{rem}}
M_\lambda\overline{\tilde M_\lambda}
dx dt
\right|
\le
|Q|^{1/2}
\|W_{\mathrm{rem}}\|_{L^2(Q)}
\left(
\sum_{k=0}^{N}
\frac{\|m_k\|_{L^\infty(Q)}}{\lambda^k}
\right)
\left(
\sum_{\ell=0}^{N}
\frac{\|\widetilde m_\ell\|_{L^\infty(Q)}}{\lambda^\ell}
\right) \leq \frac{C}{10}.
\end{equation}

We next estimate the term containing the   remainder $\widetilde R_\lambda$. To achieve this, we utilize the Cauchy-Schwarz inequality, the facts that $|e^{\mathrm{i}\Phi_\lambda}|=1$ and $\lambda \ge 1$,  along with the estimates \eqref{eq:Linf_L2_product} and  \eqref{eq:R_adjoint_lambda_H12_estimate}, to get
\begin{equation}
\label{eq:adjoint_remainder_cross_term_estimate}
\begin{aligned}
2\lambda
\left|
\int_Q
e^{\mathrm{i}\Phi_\lambda}
\mathcal{A}\cdot\alpha 
M_\lambda\overline{\widetilde R_\lambda}
dx dt
\right|
&\leq
2\lambda |\alpha|
\|\mathcal{A}\|_{L^2(Q)}
\left(
\sum_{k=0}^{N}\frac{\|m_k\|_{L^\infty(Q)}}{\lambda^k}
\right)
\|\widetilde R_\lambda\|_{L^2(Q)}
\\
&\lesssim |\alpha|
\|\mathcal{A}\|_{L^2(Q)}
\left(
\sum_{k=0}^{N}\frac{\|m_k\|_{L^\infty(Q)}}{\lambda^k}
\right)
\|\mathcal L_{\mathcal{A}^{(1)},q^{(1)}}^\ast \widetilde m_N\|_{L^2(Q)}
\lambda^{1-N} \leq \frac{C}{10}.
\end{aligned}
\end{equation}
The same arguments also yield that
\begin{equation}
\label{eq:E_lambda_estimate_5}
\begin{aligned}
2
\left|
\int_Q
e^{\mathrm{i}\Phi_\lambda}
\mathcal{A}\cdot\nabla M_\lambda 
\overline{\widetilde R_\lambda}
dx dt
\right|
&\lesssim
\|\mathcal{A}\|_{L^2(Q)}
\left(
\sum_{k=0}^{N}
\frac{\|\nabla m_k\|_{L^\infty(Q)}}{\lambda^k}
\right)
\|\mathcal L_{\mathcal{A}^{(1)},q^{(1)}}^\ast \widetilde m_N\|_{L^2(Q)}
\lambda^{-N}\leq \frac{C}{10}.
\end{aligned}
\end{equation}

Turning our attention to the terms involving \(\nabla R_\lambda\), it again  follows from the Cauchy-Schwarz inequality, the fact that $|e^{-\mathrm{i}\Phi_\lambda}|=1$, $\lambda \ge 1$, along with the estimates \eqref{eq:Linf_L2_product} and \eqref{eq:R_lambda_H12_estimate}, that
\begin{equation}
\label{eq:E_lambda_estimate_6}
\begin{aligned}
2
\left|
\int_Q
e^{-\mathrm{i}\Phi_\lambda}
\mathcal{A}\cdot\nabla R_\lambda 
\overline{\tilde M_\lambda}
dx dt
\right|
& \leq
2
\|\mathcal{A}\|_{L^2(Q)}
\left(
\sum_{\ell=0}^{N}
\frac{\|\widetilde m_\ell\|_{L^\infty(Q)}}{\lambda^\ell}
\right)
\|\nabla R_\lambda\|_{L^2(Q)}
\\
& \lesssim
\|\mathcal{A}\|_{L^2(Q)}
\left(
\sum_{\ell=0}^{N}
\frac{\|\widetilde m_\ell\|_{L^\infty(Q)}}{\lambda^\ell}
\right)
\|\mathcal L_{\mathcal{A}^{(2)},q^{(2)}}m_N\|_{L^2(Q)}
\lambda^{-N}\leq \frac{C}{10}.
\end{aligned}
\end{equation}
The same reasoning also gives us 
\begin{equation}
\label{eq:E_lambda_estimate_7}
2
\left|
\int_Q
\mathcal{A}\cdot\nabla R_\lambda 
\overline{\widetilde R_\lambda}
dx dt
\right|
\lesssim
\|\mathcal{A}\|_{L^\infty(Q)}
\|\mathcal L_{\mathcal{A}^{(2)},q^{(2)}}m_N\|_{L^2(Q)}
\|\mathcal L_{\mathcal{A}^{(1)},q^{(1)}}^\ast \widetilde m_N\|_{L^2(Q)}
\lambda^{-2N}\leq \frac{C}{10}.
\end{equation}

Finally, for the remaining three terms containing $W_{\mathrm{rem}}$, we utilize analogous computations as above to deduce that
\begin{equation}
\label{eq:E_lambda_estimate_8}
\left|
\int_Q
e^{\mathrm{i}\Phi_\lambda}
W_{\mathrm{rem}}
M_\lambda\overline{\widetilde R_\lambda}
dx dt
\right|
\lesssim
\|W_{\mathrm{rem}}\|_{L^2(Q)}
\left(
\sum_{k=0}^{N}
\frac{\|m_k\|_{L^\infty(Q)}}{\lambda^k}
\right)
\|\mathcal L_{\mathcal{A}^{(1)},q^{(1)}}^\ast \widetilde m_N\|_{L^2(Q)}
\lambda^{-N}\leq \frac{C}{10},
\end{equation}
\begin{equation}
\label{eq:E_lambda_estimate_9}
\left|
\int_Q
e^{-\mathrm{i}\Phi_\lambda}
W_{\mathrm{rem}}
R_\lambda\overline{\tilde M_\lambda}
dx dt
\right|
\lesssim
\|W_{\mathrm{rem}}\|_{L^2(Q)}
\left(
\sum_{\ell=0}^{N}
\frac{\|\widetilde m_\ell\|_{L^\infty(Q)}}{\lambda^\ell}
\right)
\|\mathcal L_{\mathcal{A}^{(2)},q^{(2)}}m_N\|_{L^2(Q)}
\lambda^{-N}\leq \frac{C}{10},
\end{equation}
and
\begin{equation}\label{eq:E_lambda_estimate_10}
\left|
\int_Q
W_{\mathrm{rem}}
R_\lambda\overline{\widetilde R_\lambda}
dx dt
\right|
\lesssim
\|W_{\mathrm{rem}}\|_{L^\infty(Q)}
\|\mathcal L_{\mathcal{A}^{(2)},q^{(2)}}m_N\|_{L^2(Q)}
\|\mathcal L_{\mathcal{A}^{1},q^{(1)}}^\ast \widetilde m_N\|_{L^2(Q)}
\lambda^{-2N}\leq \frac{C}{10}.
\end{equation}
Since $N\geq 1$ and $\lambda\geq \lambda_0>1$, the claimed estimate  \eqref{eq:E_lambda_uniform_bound} follows immediately from  \eqref{eq:E_lambda_estimate_1}--\eqref{eq:E_lambda_estimate_10}.

We now multiply the identity \eqref{eq:principal_plus_error} by $\frac{-i}{\lambda}$ on both sides. Due to \eqref{eq:E_lambda_uniform_bound}, it follows immediately that  $\frac{\mathcal E_\lambda}{\lambda}\to 0$ as $\lambda \to \infty$. Therefore, we obtain 
\[
\int_Q
2\mathrm{i} \mathcal{A}(t,x)\cdot \alpha 
m_0(t,x)\overline{\widetilde m_0(t,x)}
dx dt
=0.
\]
Using the expressions for $m_0$ and $\widetilde m_0$ from  \eqref{eq:m0_definition} with $A=\mathcal{A}^{(2)}$ and $A=\mathcal{A}^{(1)}$ respectively, we get
\begin{equation}
\label{eq:identity_with_amplitudes}
\int_Q
2\mathrm{i} \mathcal{A}(t,x)\cdot \alpha 
\chi_{\varepsilon}^2(t)\eta^2(x')
\exp\left(
\mathrm{i}\int_s^\infty
\alpha\cdot  \mathcal{A}(t,x'+p\alpha) dp
\right)
dx dt
=0.
\end{equation}
We next extend $\mathcal{A}$ by zero outside $Q$ without enlarging the spatial domain. This is possible because the assumption $\mathcal{A}^{(1)}=\mathcal{A}^{(2)}$  on the lateral boundary $\Sigma$ implies that $\mathcal{A}=0$ on $\Sigma$.
Furthermore, in view of  the decomposition \eqref{eq:x_decomposition_alpha},  we use   Fubini's theorem to rewrite  \eqref{eq:identity_with_amplitudes} as
\begin{equation}
\label{eq:identity_line_coordinates}
\int_0^T
\int_{\alpha^\perp}
\int_{\mathbb R}
2\mathrm{i} \mathcal{A}(t,x'+s\alpha)\cdot \alpha 
\chi_{\varepsilon}^2(t)\eta^2(x')
\exp\left(
\mathrm{i}\int_s^\infty
\alpha\cdot  \mathcal{A}(t,x'+p\alpha) dp
\right)
ds dx' dt
=0.
\end{equation}
Let us note that
\begin{align*}
2\mathrm{i}\alpha\cdot  \mathcal{A}(t,x'+s\alpha)
\exp\left(
\mathrm{i}\int_s^\infty
\alpha\cdot  \mathcal{A}(t,x'+p\alpha) dp
\right)
=
-2\frac{d}{ds}
\exp\left(
\mathrm{i}\int_s^\infty
\alpha\cdot  \mathcal{A}(t,x'+p\alpha) dp
\right).
\end{align*}
Therefore, we get from \eqref{eq:identity_line_coordinates} that
\[
\int_0^T
\int_{\alpha^\perp}
\chi_{\varepsilon}^2(t)\eta^2(x')
\int_{\mathbb R}
\frac{d}{ds}
\left[
\exp\left(
\mathrm{i}\int_s^\infty
\alpha\cdot \mathcal{A}(t,x'+p\alpha) dp
\right)
\right]
ds dx' dt
=0.
\]
By the fundamental theorem of calculus, we obtain
\[
\int_0^T
\int_{\alpha^\perp}
\chi_{\varepsilon}^2(t)\eta^2(x')
\left[
1-
\exp\left(
\mathrm{i}\int_{-\infty}^{\infty}
\alpha\cdot \mathcal{A}(t,x'+p\alpha) dp
\right)
\right]
dx' dt
=0.
\]
Since $\chi_\varepsilon^2$ and $\eta^2$ can be  arbitrarily chosen with supports contained in $(0,T)$ and $\alpha^\perp$, respectively, and $\mathcal{A}$ is smooth, the fundamental lemma of the calculus of variations implies that 
\begin{equation}\label{eq:exponential_identity_A0}
\exp\left(
\mathrm{i}\int_{-\infty}^{\infty}
\alpha\cdot  \mathcal{A}(t,x'+p\alpha) dp
\right)
=1
\end{equation}
for every $(t,x')\in (0,T)\times \alpha^\perp$.

We next show that
\begin{equation}
\label{eq:A0_int_over_R}
\int_{-\infty}^{\infty}
\alpha\cdot  \mathcal{A}(t,x'+p\alpha) dp=0
\end{equation}
by following the technique from \cite[Section 4]{liu2025h}. To this end, for fixed $t\in(0,T)$, we set
\[
F_t(x')
:=
\int_{-\infty}^{\infty}
\alpha\cdot \mathcal{A}(t,x'+p\alpha) dp,
\quad x'\in\alpha^\perp .
\]
Since  $\mathcal{A}$ is smooth, it holds that $F_t$ is smooth on $\alpha^\perp$.  From
\eqref{eq:exponential_identity_A0}, we have $F_t(x')\in 2\pi\mathbb Z$ for $x'\in\alpha^\perp$. Since $\alpha^\perp$ is connected, and $F_t$ is continuous, it holds that $F_t(\alpha^\perp)$ is also connected. On the other hand, since 
$F_t(\alpha^\perp)\subset 2\pi\mathbb Z$, and the only connected subsets of the discrete set $2\pi\mathbb Z$ are singletons, $F_t$ must be constant on $\alpha^\perp$.
Moreover, since $\mathcal{A}(t,\cdot)$ is compactly supported, we have
$F_t(x')=0$ for $|x'|$ sufficiently large. From here, we get \eqref{eq:A0_int_over_R} immediately.

By \cite[Lemma~6.1]{Sahoo_Vashith}, for each $t\in (0,T)$, there exists
a function $\Psi(t,\cdot) \in C_c^\infty(\Omega)$ such that
\begin{equation}
\label{eq:gauge_relation_A0}
\mathcal{A}^{(2)}(t,x)=\mathcal{A}^{(1)}(t,x)+\nabla\Psi(t,x),
\quad (t,x)\in Q.
\end{equation}
In view of the assumption that $\nabla\cdot \mathcal{A}^{(1)}=\nabla\cdot \mathcal{A}^{(2)}$ for each $t\in (0,T)$, we get
\[
\begin{cases}
\Delta_x \Psi(t,\cdot)=0 & \text{ in } \Omega,\\
\Psi(t,\cdot)=0 & \text{ on }\partial\Omega.
\end{cases}
\]
By \cite[Theorem 8.12]{gilbarg_trudinger_1983}, we have $\Psi(t,\cdot)|_{\Omega}=0$  for each $t\in(0,T)$. Hence, we conclude from \eqref{eq:gauge_relation_A0} that $\mathcal{A}^{(2)}=\mathcal{A}^{(1)}$ in $Q$.

We now proceed to show that $q_1=0$ in $Q$. By substituting $\mathcal{A}^{(1)}=\mathcal{A}^{(2)}$ into the integral identity \eqref{eq:integral_identity_A0_q1}, we obtain 
\[
\int_{Q}
q_{1} v_{2} \overline{v}_{1} dx dt
=0.
\]
Substituting the geometric optics solutions $v_{1}(t,x)$ and $v_{2}(t,x)$ from \eqref{eq:GO_solution_v_lambda} and \eqref{eq:GO_solution_u_lambda_conjugate}, we obtain
\[
\int_{Q}
q_{1}
\left(
\sum_{k=0}^{N}\frac{m_{k}}{\lambda^{k}}
+
R_{\lambda}
\right)
\left(
\sum_{k=0}^{N}\frac{\overline{\widetilde m_{k}}}{\lambda^{k}}
+
\overline{\widetilde R_{\lambda}}
\right)
dx dt
=0.
\]
Arguing similarly to  the proof for the magnetic potential $\mathcal{A}$, we have
\[
\int_Q
q_1(t,x) 
m_0(t,x)\overline{\widetilde m_0(t,x)}
dx dt
=0.
\]
We next choose
\[
m_0= \chi(t)e^{-\mathrm{i}x\cdot\xi}e^{\mathrm{i}\int_s^\infty \alpha\cdot \mathcal{A}(t,x'+p\alpha) dp} \quad \text{ and } \widetilde{m}_0= \chi(t)e^{\mathrm{i}\int_s^\infty \alpha\cdot \mathcal{A}(t,x'+p\alpha) dp}
\]
for any $\chi\in C_c^\infty(0,T)$ and $\xi\in\alpha^\perp$, which are valid according to \eqref{eq:m0_definition}. Therefore, we obtain
\[
m_0(t,x)\overline{\widetilde m_0(t,x)}
=
\chi^2(t)e^{-\mathrm{i}x\cdot\xi}.
\]
Hence, it follows that 
\[
\int_Q
q_1(t,x)\chi^2(t)e^{-\mathrm{i}x\cdot\xi}
dx dt
=0
\]
for every $\chi\in C_c^\infty(0,T)$ and every $\xi\in\alpha^\perp$. Thus, an application of Fubini's theorem yields that 
\[
\int_0^T
\chi^2(t)
\left(
\int_\Omega
q_1(t,x)e^{-\mathrm{i}x\cdot\xi}
dx
\right)
dt
=0.
\]
By the fundamental lemma of the calculus of variations, it holds that
\[
\int_\Omega
q_1(t,x)e^{-\mathrm{i}x\cdot\xi}
dx
=0
\]
for every $t\in(0,T)$  and every $\xi\in\alpha^\perp$.
Since the direction $\alpha\in\mathbb R^n\setminus\{0\}$ is arbitrary for every $\xi\in\mathbb R^n$, we can choose $\alpha\in\mathbb R^n\setminus\{0\}$ such that $\xi\in\alpha^\perp$. Therefore, we have
\[
\int_\Omega
q_1(t,x)e^{-\mathrm{i}x\cdot\xi}
dx
=0, \quad \text{ for all } (t,\xi)\in(0,T)\times\mathbb R^n.
\]

For each fixed $t\in(0,T)$, we extend $q_1(t,\cdot)$  by zero outside $\Omega$ and denote this extension by the same letter. Then we have $\widehat{q_1}(t,\xi)=0$ for all $(t,\xi)\in(0,T)\times\mathbb R^n$, where the Fourier transform is taken with respect to the spatial variable. 
By the injectivity of the Fourier transform, we conclude that $q_1=0$ in $Q$.  

\subsection{Recovery of  $q_2$}
\label{sec:second_order_linearization}
The goal of this subsection is to establish the unique determination of  the scalar potential $q_{2}$ from the DN map defined in \eqref{eq:def_DN_map}. The proof follows the same general strategy as in Subsection \ref{subsec:first_order_linearization}.  Nevertheless, we need a stronger decay estimate for the remainder term $R_{\lambda}$, which is given in Corollary \ref{Corollary: geom opt sol 1}. 

Let us begin with the second order linearization. To this end, let $\left(\mathcal{A}^{(j)},q^{(j)} \right)\in C^{\infty}(\overline Q;\R^n)\times C^{\infty}(\overline Q\times \C;\C) $, and let $\varepsilon = (\varepsilon_1,\varepsilon_{2})\in \C^2$. 
Let $u_j (t,x; \varepsilon)$, $j=1,2$, be solutions to the     IBVP
\begin{equation}
\label{eq:ibvp_2nd_linearization}
\begin{cases}
\mathcal{L}_{\mathcal{A}^{(j)},q^{(j)}} u_j=0 & \text{ in } Q ,
\\
u_j=\varepsilon_1 f_1  + \varepsilon_{2} f_{2} & \text{ on } \Sigma,
\\
u_j(0, \cdot) =0 &\text{ in } \Omega,
\end{cases}
\end{equation}
where $\left\|\varepsilon_1 f_1+\varepsilon_2 f_2 \right\|_{H^{m+\frac52, m+\frac52}(\Sigma)} \leq \varepsilon_{0}$, and $f_p\in  H^{m+\frac52, m+\frac52}(\Sigma)$, $p=1,2$, satisfies the condition  $\partial^{k}_{t}f_{p}(0,\cdot)|_{\partial\Omega}=0$ for all $0\leq k \leq m$. 
From Theorem \ref{thm:wellposedness}, we have $u_j (\cdot,\cdot; \varepsilon)\in H^{m+1,m+2}(Q)$. 
By applying $\p_{\varepsilon_{p}}|_{\varepsilon=0}$ to \eqref{eq:ibvp_2nd_linearization}  and using that $u_{j}(t,x; 0)=0$, $j=1,2$, we see that the function $v_{j,p}:= \p_{\varepsilon_{p}} u_j(t,x;\varepsilon)|_{\varepsilon=0}$ satisfies the following IBVP:
\begin{equation*}
\begin{cases}
\mathcal{L}_{\mathcal{A}^{(j)},q_{1}^{(j)}} v_{j,p}=0 & \text{ in } Q ,
\\
v_{j,p}= f_{p}& \text{ on } \Sigma,
\\
v_{j,p}(0, \cdot) =0 &\text{ in } \Omega.
\end{cases}
\end{equation*}
Here the linear operator $\mathcal{L}_{\mathcal{A}^{(j)},q_{1}^{(j)}}$ is  given by \eqref{eq:def_linear} with potentials $\mathcal{A}^{(j)}$ and $q_{1}^{(j)}$. We have already established in Subsection \ref{subsec:first_order_linearization} that $\mathcal{A}:= \mathcal{A}^{(1)}=\mathcal{A}^{(2)} $ and $q_1:= q_1^{(1)}=q_1^{(2)} $.
Thus, by the uniqueness for the linear IBVP, we obtain $v_p:= v_{1,p}=v_{2,p}$, $p=1,2$. 

We next apply the operator  $\partial^2_{\varepsilon_{1}\varepsilon_{2}}|_{\varepsilon=0}$ to the problem \eqref{eq:ibvp_2nd_linearization} to get that the function $w_j
:=
\partial_{\varepsilon_1}\partial_{\varepsilon_2}
u_j|_{\varepsilon=0}$ satisfies the IBVP
\begin{equation}
\label{eq:second_linearized_equation_j}
\begin{cases}
\mathcal L_{\mathcal{A},q_1}w_j
+
2q_2^{(j)}v_1v_2
=0
& \text{ in } Q,
\\
w_j=0
& \text{ on } \Sigma,
\\
w_j(0,\cdot)=0
& \text{ in } \Omega.
\end{cases}
\end{equation}
Let us denote $w:=w_1-w_2$ and $q_2:=q_2^{(1)}-q_2^{(2)}$.  Then we immediately obtain from 
\eqref{eq:second_linearized_equation_j} that
\begin{equation}
\label{eq:second_linearized_difference}
\begin{cases}
\mathcal L_{\mathcal{A},q_1}w
+
2q_2v_1v_2
=0
& \text{ in } Q,
\\
w=0
& \text{ on } \Sigma,
\\
w(0,\cdot)=0
& \text{ in } \Omega.
\end{cases}
\end{equation}
Let $w_0$ be a solution to the  problem
\begin{equation}
\label{eq:adjoint_problem_w0}
\begin{cases}
\mathcal L_{\mathcal{A},q_1}^{*}w_0=0
& \text{ in } Q,
\\
w_0(T,\cdot)=0
& \text{ in } \Omega.
\end{cases}
\end{equation}
Multiplying the first equation in
\eqref{eq:second_linearized_difference} by $\overline{w_0}$ and integrating over $Q$,  and arguing as in Subsection \ref{subsec:first_order_linearization}, we have the integral identity
\begin{equation}
\label{eq:second_linearized_integral_identity}
\int_Q
2q_2v_1v_2
\overline{w_0} dx dt
=0.
\end{equation}

We now choose the vectors $\alpha_1,\alpha_2,\alpha_3 \in \R^n \setminus \{0\}$ such that
\begin{equation}
\label{eq:alpha_condition}
\alpha_1+\alpha_2=\alpha_3
\quad \text{ and }\quad 
|\alpha_1|^2+|\alpha_2|^2=|\alpha_3|^2.
\end{equation}
By Corollaries \ref{Corollary: geom opt sol 1} and \ref{Corollary:adjoint geom opt sol 1}, we have geometric optics solutions given by
\begin{equation}
\label{eq:GO_solution_v1}
v_1(t,x)
=
e^{\mathrm{i}(\lambda x\cdot\alpha_1-\lambda^2|\alpha_1|^2t)}
\left(
\sum_{k=0}^{N}\frac{m_k^1(t,x)}{\lambda^k}\right)
+
R_\lambda^1(t,x),
\end{equation}
\begin{equation}
\label{eq:GO_solution_v2}
v_2(t,x)
=
e^{\mathrm{i}(\lambda x\cdot\alpha_2-\lambda^2|\alpha_2|^2t)}
\left(
\sum_{k=0}^{N}\frac{m_k^2(t,x)}{\lambda^k}\right)
+
R_\lambda^2(t,x),
\end{equation}
and
\begin{equation}
\label{eq:GO_solution_w0}
w_0(t,x)
=
e^{\mathrm{i}(\lambda x\cdot\alpha_3-\lambda^2|\alpha_3|^2t)}
\left(
\sum_{k=0}^{N}\frac{m_k^3(t,x)}{\lambda^k}\right)
+
R_\lambda^3(t,x).
\end{equation}
Here the smooth amplitude $m_k^j$, $j=1,2$, solves the transport equation \eqref{eqn: Transport} with $\alpha = \alpha_j$ and $A=\mathcal{A}$, while $m_k^3$ is a solution to the same equation with $\alpha = \alpha_3$ and $A=\mathcal{A}$. Moreover, the remainder terms $R_\lambda^1$ and $R_\lambda^2$ satisfy the estimate \eqref{eq:est_remainder_higher_order}, whereas $R_\lambda^3$ satisfies the estimate \eqref{eq:R_lambda_Hm_estimate_adjoint}.
For $j=1,2,3$, let us write $\Phi_j(t,x):=\lambda x\cdot\alpha_j-\lambda^2|\alpha_j|^2t$, and  denote
\begin{equation}
\label{eq:a_j_definition}
a_j(t,x):=
\sum_{k=0}^{N}\frac{m_k^j(t,x)}{\lambda^k}
+
R_\lambda^j(t,x).
\end{equation}
Using these notations, the solutions \eqref{eq:GO_solution_v1}--\eqref{eq:GO_solution_w0} read $v_j=e^{\mathrm{i}\Phi_j}a_j$, $j=1,2,3$. 

Let us observe from \eqref{eq:alpha_condition} that 
\begin{equation}
\label{eq:phase_cancellation_second_linearization}
\Phi_1+\Phi_2-\Phi_3=0.
\end{equation} 
Also, we utilize  \eqref{eq:phase_cancellation_second_linearization} to compute that
\[
\nabla(v_1v_2)
=
e^{\mathrm{i}\Phi_3}
\left[
\mathrm{i}\lambda\alpha_3 a_1a_2
+
\nabla(a_1a_2)
\right].
\]
Thus, by substituting the geometric solutions into the integral identity \eqref{eq:second_linearized_integral_identity}, we obtain
\begin{align*}
\int_Q
2q_2
a_1a_2\overline{a_3} dx dt
=0.
\end{align*}

We now analyze the integral above  as $\lambda \to \infty$. To this end, by  Corollaries \ref{Corollary: geom opt sol 1} and \ref{Corollary:adjoint geom opt sol 1}, we get from \eqref{eq:a_j_definition} that $a_j \to m_0^j$ in $H^{m+1}(Q)$ as $ \lambda\to\infty$, $j=1,2,3$. 
Furthermore, since $H^{m+1}(Q)$ is a Banach algebra, it follows that $a_1a_2\overline{a_3}
\to m_0^1m_0^2\overline{m_0^3}$ in $H^{m+1}(Q)$ as $\lambda\to\infty$. 
Therefore, we obtain 
\[
\int_Q\ q_2
a_1a_2\overline{a_3} dx dt
\longrightarrow
\int_Q
q_2 m_0^1(t,x)m_0^2(t,x)\overline{m_0^3(t,x)}
dx dt,
\quad\lambda\to\infty .
\]
Hence, we conclude from the previous analysis that 
\begin{equation}
\label{eq:A1_alpha_leading_limit_zero}
0
=
\int_Q
2q_2
m_0^1(t,x)m_0^2(t,x)\overline{m_0^3(t,x)}
dx dt .
\end{equation}

We now choose the smooth amplitudes
$m_0^1$, $m_0^2$, and $m_0^3$ such that
\[
m_0^1(t,x)m_0^2(t,x)\overline{m_0^3(t,x)}
=
\chi^3(t)e^{-\mathrm{i}x\cdot \xi},
\]
where $\chi\in C_c^\infty(0,T)$  and $\xi\in\alpha^\perp$. Then the identity  \eqref{eq:A1_alpha_leading_limit_zero} gives us
\[
\int_Q
q_{2}(t,x)
\chi^{3}(t)
e^{-\mathrm{i}x\cdot\xi}
dx dt
=0.
\]
Proceeding as in the proof of uniqueness of the scalar potential $q_1$, we obtain $\widehat{q_2}(t,\xi)=0$ for all 
$(t,\xi)\in(0,T)\times\mathbb R^n$. 
By the injectivity of the Fourier transform, we conclude that $q_2=0$ in $Q$.

\subsection{Recovery of $q_{i+1}$, $i\ge 2$}
\label{subsec:higher_order_linearization}

In this subsection, we prove the uniqueness result for the higher order scalar potentials appearing in the power series expansion \eqref{eq:expansion_q}
by applying higher order linearization and using an induction argument.  
Assume that  $q_{\ell}^{(1)}=q_{\ell}^{(2)}$ for $1\leq \ell\leq r-1$ in $Q$ for some integer $r\geq 2$.
Our goal is to show that  $q_{r}^{(1)}=q_{r}^{(2)}$ in $Q$.

Let $q^{(j)} \in C^{\infty}(\overline Q\times \C;\C)$, $j=1,2$, be the nonlinear potential given by the expansion  \eqref{eq:expansion_q}. Let $f_p \in  H^{m+\frac52, m+\frac52}(\Sigma)$, $p=1, \dots, r$, be such that $\partial^{k}_{t}f_{p}(0,\cdot)|_{\partial\Omega}=0$ for all   $0\leq k \leq m+1$, and let $\varepsilon=(\varepsilon_1,\ldots,\varepsilon_r) \in \C^r $. 
For $|\varepsilon|>0$ sufficiently small such that $\left\|\varepsilon_{1}f_1+ \cdots + \varepsilon_rf_r \right\|_{H^{m+\frac52, m+\frac52}(\Sigma)} \leq \varepsilon_{0}$, suppose that the function  $u_j (\cdot,\cdot; \varepsilon) \in H^{m+1,m+2}(Q)$, $j=1,2$, is a solution to the  IBVP
\begin{equation}
\label{eq:ibvp_higher_linearization}
\begin{cases}
\mathcal{L}_{\mathcal{A},q^{(j)}} u_j=0 & \text{ in } Q ,
\\
u_j=\varepsilon_{1}f_1+ \cdots + \varepsilon_rf_r & \text{ on } \Sigma,
\\
u_j(0, \cdot) =0 &\text{ in } \Omega.
\end{cases}
\end{equation}

We now perform the higher order linearization. 
Let us again denote $v_{j,p}:=
\partial_{\varepsilon_p}u_j(t,x;\varepsilon)\big|_{\varepsilon=0}$. By the induction hypothesis, as well as the uniqueness of the linearized IBVP, we have $v_p:= v_{1,p}=v_{2,p} $ for all $p=1,\ldots,r$. Hence, by applying   $\partial_{\varepsilon_1}\cdots\partial_{\varepsilon_r}
\big|_{\varepsilon=0}$ to  \eqref{eq:ibvp_higher_linearization},  direct computations yield that the function $w_j^{(r)}
:=
\partial_{\varepsilon_1}\cdots\partial_{\varepsilon_r}
u_j(t,x;\varepsilon)\big|_{\varepsilon=0}$ satisfies the IBVP 
\[
\begin{cases}
\begin{aligned}
\mathcal L_{\mathcal{A},q_1}w_j^{(r)}
&
+
r!q_r^{(j)}v_1v_2\cdots v_r
+
\mathcal F_r^{(j)}
=0
\end{aligned}
& \text{ in } Q,
\\ 
w_j^{(r)}=0
& \text{ on } \Sigma,
\\ 
w_j^{(r)}(0,\cdot)=0
& \text{ in } \Omega .
\end{cases}
\]
Arguing similarly as in  \cite{Krupchyk_Uhlmann_semilinear_partial,Liu_Wang}, we see that the function $\mathcal F_r^{(j)}$ contains only lower order coefficients, which are independent of $j=1,2$. Thus, by the induction hypothesis,  we have $\mathcal F_r^{(1)}=\mathcal F_r^{(2)}$. Hence, by setting $q_r:=q_r^{(1)}-q_r^{(2)}$, the function $w^{(r)}:=w_1^{(r)}-w_2^{(r)}$ is a solution to the following IBVP:
\begin{equation}
\label{eq:rth_linearized_difference}
\begin{cases}
\begin{aligned}
\mathcal L_{\mathcal{A},q_1}w^{(r)}
+
r!q_r v_1v_2\cdots v_r
=0
\end{aligned}
& \text{ in } Q,
\\
w^{(r)}=0
& \text{ on } \Sigma,
\\
w^{(r)}(0,\cdot)=0
& \text{ in } \Omega .
\end{cases}
\end{equation}
Multiplying \eqref{eq:rth_linearized_difference} by $\overline{w_0}$, where $w_{0}$ solves the problem \eqref{eq:adjoint_problem_w0}, and integrating over $Q$, we obtain the integral identity
\begin{align*}
0
=
\int_Q
& 
r!q_r v_1v_2\cdots v_r
\overline{w_0} dx dt .
\end{align*}

Next, we apply Corollaries \ref{Corollary: geom opt sol 1} and \ref{Corollary:adjoint geom opt sol 1} to choose geometric optics solutions of the form
\[
v_p(t,x)
=
e^{\mathrm{i}(\lambda x\cdot\alpha_p-\lambda^2|\alpha_p|^2t)}
\left(
\sum_{k=0}^{N}\frac{m_k^p(t,x)}{\lambda^k}
+
R_\lambda^p(t,x)
\right),
\quad p=1,\ldots,r,
\]
and
\[
w_0(t,x)
=
e^{\mathrm{i}(\lambda x\cdot\alpha_{r+1}-\lambda^2|\alpha_{r+1}|^2t)}
\left(
\sum_{k=0}^{N}\frac{m_k^{r+1}(t,x)}{\lambda^k}
+
R_\lambda^{r+1}(t,x)
\right).
\]
Here the smooth amplitude $m_k^p$, $p=1,2, \dots, r$, solves the transport equation \eqref{eqn: Transport} with $\alpha = \alpha_p$ and $A=\mathcal{A}$, and $m_k^{r+1}$ is a solution to the same equation with $\alpha = \alpha_{r+1}$ and $A=\mathcal{A}$. The remainder $R_\lambda^p$, $p=1, \dots, r$, satisfies the estimate \eqref{eq:est_remainder_higher_order}, while $R_\lambda^{r+1}$ satisfies the estimate \eqref{eq:R_lambda_Hm_estimate_adjoint}. Moreover, the vectors $\alpha_1,\ldots,\alpha_r,\alpha_{r+1}\in \R^n \setminus \{0\}$ are chosen such that
\[
\sum_{p=1}^{r}\alpha_p=\alpha_{r+1},
\qquad
\sum_{p=1}^{r}|\alpha_p|^2=|\alpha_{r+1}|^2.
\]
From here, we argue analogously to Subsection \ref{sec:second_order_linearization} to conclude that  $q_r=0$ in $Q$. 

Therefore, it follows from induction that  $q_{\ell}^{(1)}=q_{\ell}^{(2)}$ in $Q$ for all $\ell \ge 1$. Finally, in view of the series  expansion  \eqref{eq:expansion_q}, which converges in the $C^\infty$-topology, we conclude that  $q^{(1)}=q^{(2)}$ in $\overline{Q} \times \mathbb{C}$. This completes the proof of Theorem \ref{thm:main_result}.

\section{Proof of Theorem \ref{thm:main_result_partial_data}}
\label{sec:proof_partial_data}

The main objective of this section is to establish   Theorem \ref{thm:main_result_partial_data}.  The key step is to transform the partial data problem into a full data one by invoking the unique continuation principle given in  Lemma \ref{UCP}. 

Let us  begin this section with the following construction.  Let $\mathcal{O}\subseteq \overline{\Omega} $ be a non-empty set such that $\mathcal{O}$ is an open neighborhood of $\partial\Omega$. 
We first define the admissible class of coefficients $\mathcal{A}$ and $q_{1}$   as 
\[
\mathscr{A}
:=  
\left\{(\mathcal{A},q_{1})\in C^{\infty}(\overline Q)\times C^{\infty}(\overline Q):
\mathcal{A}=0 \text{ and } q_{1}=0
\text{ in } (0,T)\times\mathcal{O}
\right\}.
\]
Suppose that $\mathcal{O}_{j} \subseteq \mathcal{O}$, $j = 1, 2, 3$,  are open neighborhoods of $\partial\Omega$ such that $\overline{\mathcal{O}}_{j+1} \subset \mathcal{O}_{j}$.
We then define the set $\Omega_{j} := \Omega \setminus \overline{\mathcal{O}_{j}}$, 
and the cut-off function $\theta \in C^{\infty}(\overline\Omega)$ such that $0\le \theta \le 1$ by the following formula:
\begin{equation}\label{eqn: construction theta}
\theta:= \begin{cases}
0 \qquad\text{ in } \mathcal{O}_{3},\\
1  \qquad\text{ in } \Omega_{2}.
\end{cases}
\end{equation}

We are now ready to introduce the unique continuation property, which was originally established in \cite[Corollary 1]{Bellassoued_Fraj}. 

\begin{lem}
\label{UCP}
Let $\Omega \subseteq \R^n$, $n\ge 2$, be a bounded domain with smooth boundary $\p \Omega$.  For any $T>0$,  let $Q=(0,T) \times \Omega$.  Suppose that $\Gamma \subseteq \p \Omega$ is an arbitrary nonempty open set, and  $\Sigma^\sharp := (0,T)\times \Gamma$. Assume that  $(\mathcal{A},q_{1})\in \mathscr{A}$. Let  $v \in H^{1,2}(Q)$ be a solution to the linear IBVP
\[
\begin{cases}
\mathcal{L}_{\mathcal{A},q_{1}}v  = g_{0} & \text{ in } Q,
\\
v = 0 & \text{ on } \Sigma,
\\
v (0,\cdot)=0 & \text{ in } \Omega,
\end{cases}
\]
where 
$g_{0} \in L^{2}(Q)$ and $\operatorname{supp}(g_{0}) \subset (0,T) \times (\Omega\setminus\mathcal{O}),$
such that $\partial_{\nu}v = 0$ on $\Sigma^{\sharp}$. Then $v=0  \text{ in } (0,T) \times (\Omega_{3}\setminus \Omega_{2})$.
\end{lem}

Suppose that the function $v \in L^{2}\left(0,T;H^{2}(\Omega)\cap H^{1}_{0}(\Omega)\right) \cap H^{1}(0,T;L^{2}(\Omega))$ is a solution to the IBVP \eqref{eq:linearized_difference_v}. By a direct computation, the function $\tilde{v}:= \theta v$ satisfies the IBVP
\begin{equation}
\label{eqn: IBVP tildew}
\begin{cases}
\mathcal{L}_{\mathcal{A}^{(1)},q_{1}^{(1)}} \tilde{v} = \theta F - \left[\theta, \Delta+2\mathrm{i}\mathcal{A}^{(1)}\cdot \nabla\right]v & \text{ in } Q,
\\
\tilde{v} = 0 & \text{ on } \Sigma,
\\
\tilde{v} (0,\cdot)=0 & \text{ in } \Omega,
\end{cases}
\end{equation}
where $F:= \left(
2\mathrm{i}\mathcal{A}\cdot \nabla
+\mathrm{i}\nabla \cdot \mathcal{A}
-\left|\mathcal{A}^{(2)}\right|^2
+\left|\mathcal{A}^{(1)}\right|^2
+q_{1}
\right)v_{2}$, and $\left[\cdot,\cdot \right]$ denotes the commutator. 
Since $\mathcal{A}=0$ and $q_1=0$ in $(0,T)\times \mathcal{O}$, it follows immediately that $F = 0$ in $(0,T)\times \mathcal{O}$. Hence, $\supp F \subset (0,T) \times (\Omega \setminus \cO)$. Thus, by applying Lemma \ref{UCP}, we conclude that
\begin{equation}
\label{eq:v_vanish}
\tilde v=0 \quad \text{ in } (0,T)\times (\Omega_{3}\setminus \Omega_{2}).
\end{equation}
Moreover, as $\theta = 1$ in $\Omega_{2}$ and 
$\theta = 0$ in $\mathcal{O}_{3}\subset \mathcal{O}$, we get that 
\begin{equation}
\label{eqn:thetaF=f}
\theta F = F \quad \text{in } (0,T)\times\Omega .
\end{equation}
Let us now compute that  
\[
[\theta,\Delta+2\mathrm{i}\mathcal{A}^{(1)}\cdot\nabla]v
=-2\nabla\theta\cdot\nabla v
-\left(\Delta\theta+2\mathrm{i}\mathcal{A}^{(1)}\cdot\nabla\theta\right)v .
\]
Since $\theta$ is constant on $\mathcal{O}_{3} \cup \Omega_{2}$, we have $\nabla \theta = 0$ and $\Delta \theta = 0$ in $(0,T)\times (\mathcal{O}_{3}\cup\Omega_{2})$. Therefore, in view of \eqref{eq:v_vanish}, we have 
\begin{equation}\label{eqn: commutator vanishes}
[\theta,\Delta+2\mathrm{i}\mathcal{A}^{(1)}\cdot\nabla]v  = 0 \quad \text{ in } Q.
\end{equation}
Thus, using \eqref{eqn:thetaF=f} and \eqref{eqn: commutator vanishes}, we  rewrite the problem \eqref{eqn: IBVP tildew} as 
\begin{equation}
\label{eqn: IBVP tildew_final}
\begin{cases}
\mathcal{L}_{\mathcal{A}^{(1)},q_{1}^{(1)}} \tilde{v} =  F  & \text{ in } Q,
\\
\tilde{v} = 0 & \text{ on } \Sigma,
\\
\tilde{v} (0,\cdot)=0 & \text{ in } \Omega,
\end{cases}
\end{equation}
On the other hand, by  \eqref{eqn: construction theta}, we have
$\theta=0$ in $\mathcal{O}_{3}$. Hence, it holds that $\widetilde v=0$ in $\mathcal{O}_{3}$. Furthermore, since $\mathcal{O}_{3}$ is a neighborhood of
$\partial\Omega$, it follows that
\begin{equation}
\label{eq:normal_derivative_tilde_v_zero}
\partial_{\nu} \widetilde{v}\big|_{\Sigma}=0.
\end{equation}

Let $v_{1}$ be a solution to the problem  \eqref{eq:adjoint_problem_first_linearization}. We next multiply  the first equation of \eqref{eqn: IBVP tildew_final} by $\overline{v}_{1}$ and integrate over $Q$. Due to the conditions $\widetilde{v}=0$ on $\Sigma$ and  $\widetilde{v}(0,\cdot)= v_1(T,\cdot)=0$ in $\Omega$, in conjunction with \eqref{eq:normal_derivative_tilde_v_zero}, an integration by parts gives us the integral identity \eqref{eq:integral_identity_A0_q1}. 
Therefore, applying the same analysis in Subsection \ref{subsec:first_order_linearization} from the identity \eqref{eq:integral_identity_A0_q1} onward, we obtain $\mathcal{A}^{(1)}=\mathcal{A}^{(2)}$  and $q_1^{(1)}=q_1^{(2)}$  in $Q$. 

The same arguments also apply to higher order linearization. Indeed, let the function $w^{(r)}$ be as in the initial boundary problem \eqref{eq:rth_linearized_difference}. Due to the assumption that $\mathcal{A}=0$ and $q_1=0$ on $(0,T)\times \mathcal{O}$, the function $\theta w^{(r)}$ satisfies a similar problem whose  source term is supported away from $(0,T)\times \mathcal{O}$ with vanishing commutator terms. Hence, the partial data problem in question is reduced to a full data one after the $r^{\mathrm{th}}$ order linearization. Then the same arguments as those in Subsections \ref{sec:second_order_linearization} and  \ref{subsec:higher_order_linearization} yield that $q^{(1)}=q^{(2)}$ in  $\overline{Q} \times \mathbb{C}$.
This completes the proof of Theorem \ref{thm:main_result_partial_data}.

\bibliographystyle{abbrv}
\bibliography{bib_nonlinear_dynamic_Schrodinger}

\end{document}